\documentclass[a4paper,10.9pt]{amsart}

\DeclareRobustCommand{\SkipTocEntry}[5]{}

\usepackage{ amsmath }
\usepackage{marvosym}

\usepackage[T1]{fontenc}
\usepackage{empheq}
\usepackage{relsize}
\usepackage{amsmath}

\usepackage{bm}

\usepackage{upgreek}
\usepackage{ esint }
\usepackage{color}
\usepackage{comment}
\usepackage{amssymb}
\usepackage{amsfonts}
\usepackage{graphicx}

\usepackage{amscd}

\usepackage{slashed}
\usepackage{pdflscape}
\usepackage{tikz}
\usepackage{enumerate}
\usepackage{multirow}
\usepackage{stmaryrd}
\usepackage{xcolor}

\usepackage{scalerel,stackengine}

\usepackage{ifthen}

\usepackage{bbold}
\usepackage[linkcolor=blue]{hyperref}
\usepackage{mathrsfs}
\usepackage[colorinlistoftodos]{todonotes}

\definecolor{blue}{rgb}{.255,.41,.884} 
\definecolor{red}{rgb}{1, 0, 0} 
\definecolor{green}{rgb}{.196,.804,.196} 
\definecolor{yellow}{rgb}{1,.648,0} 
\definecolor{pink}{rgb}{1,0.5,0.5}

\newtheorem{theorem}{Theorem}[section]
\newtheorem{lemma}[theorem]{Lemma}

\newtheorem{hyp}[theorem]{Hypothesis}
  
\newtheorem{proposition}[theorem]{Proposition}
\newtheorem{corollary}[theorem]{Corollary}

\theoremstyle{definition}
\newtheorem{definition}[theorem]{Definition}
\newtheorem{example}[theorem]{Example}

\theoremstyle{remark}
\newtheorem{remark}[theorem]{Remark}

\newtheorem{problem}[theorem]{Problem}

\usepackage{pdfsync}

\usepackage{graphicx} 
\usepackage{amsmath} 
\usepackage{amsfonts}
\usepackage{amssymb}

\newcommand{\be}{\begin{equation}}

\newcommand{\ee}{\end{equation}}

\newcommand{\II}{{\rm  I\hspace{-.2mm}I}}
\newcommand{\IIo}{\hspace{0.4mm}\mathring{\rm{ I\hspace{-.2mm} I}}{\hspace{.0mm}}}
\newcommand{\Fo}{ {\hspace{.6mm}} \mathring{\!{ F}}{\hspace{.2mm}}}
\newcommand{\IIIo}{{\mathring{{\bf\rm I\hspace{-.2mm} I \hspace{-.2mm} I}}{\hspace{.2mm}}}{}}
\newcommand{\IVo}{{\mathring{{\bf\rm I\hspace{-.2mm} V}}{\hspace{.2mm}}}{}}
\newcommand{\Vo}{{\mathring{{\bf\rm V}}}{}}

\newcommand{\otop}{\mathring{\top}}
\newcommand{\ID}{{{{\bf\rm I\hspace{-.3mm} D}}{\hspace{0.1mm}}}}
\newcommand{\IDh}{{{{\bf\rm I\hspace{-.3mm} \hat{D}}}{\hspace{0.1mm}}}}

\newcommand{\III}{{{{\bf\rm I\hspace{-.2mm} I \hspace{-.2mm} I}}{\hspace{.2mm}}}{}}

\newcommand{\ba}{\begin{array}}

\newcommand{\ea}{\end{array}}

\newcommand{\beq}{\begin{eqnarray}}

\newcommand{\eeq}{\end{eqnarray}}

\newtheorem{lm}{lemma}

\newtheorem{thee}{theorem}

\newtheorem{proo}{proposition}

\newtheorem{co}{corollary}

\newtheorem{rem}{remark}

\newtheorem{deff}{definition}

\newcommand{\bd}{\begin{deff}}

\newcommand{\ed}{\end{deff}}

\newcommand{\bl}{\begin{lm}}

\newcommand{\el}{\end{lm}}

\newcommand{\bp}{\begin{proo}}

\newcommand{\ep}{\end{proo}}

\newcommand{\bt}{\begin{thee}}

\newcommand{\et}{\end{thee}}

\newcommand{\bc}{\begin{co}}

\newcommand{\ec}{\end{co}}

\newcommand{\brm}{\begin{rem}}

\newcommand{\erm}{\end{rem}}

\usepackage{t1enc}
\def\frak{\mathfrak}

\def\Cal{\mathcal}

\newcommand{\bS}{\mathbb{S}}

\newcommand{\newc}{\newcommand}

\let\ccdot.

\newc{\aR}{\mbox{\boldmath{$ R$}}}
\newc{\aS}{\mbox{\boldmath{$ S$}}}
\newc{\aT}{\mbox{\boldmath{$ T$}}}
\newc{\aW}{\mbox{\boldmath{$ W$}}}

\newc{\aD}{\mbox{\boldmath{$ D$}}\hspace{-.2mm}}

\newc{\aK}{\mbox{\boldmath{$ K$}}}
\newc{\aL}{\mbox{\boldmath{$ L$}}}

\newcommand{\ce}{{\Cal E}}

\newcommand{\ct}{{\Cal T}}

\newcommand{\bT}{{\Bbb T}}

\usepackage{amssymb}
\usepackage{amscd}

\newcommand{\Rho}{{\it P}}

\newcommand{\Sc}{{\it Sc}}

\let\hash=\sharp  

\newcommand{\nn}[1]{(\ref{#1})}

\newcommand{\bg}{\mbox{\boldmath{$ g$}}}

\newc{\obstrn}[2]{B^{#1}_{#2}}

\newcommand{\rpl}                         
{\mbox{$
\begin{picture}(12.7,8)(-.5,-1)
\put(0,0.2){$+$}
\put(4.2,2.8){\oval(8,8)[r]}
\end{picture}$}}

\newcommand{\lpl}                         
{\mbox{$
\begin{picture}(12.7,8)(-.5,-1)
\put(2,0.2){$+$}
\put(6.2,2.8){\oval(8,8)[l]}
\end{picture}$}}

\newc{\tensor}[1]{#1}
\newc{\Mvariable}[1]{\mbox{#1}}
\newc{\down}[1]{{}_{#1}}
\newc{\up}[1]{{}^{#1}}

\newc{\JulyStrut}{\rule{0mm}{6mm}}
\newc{\midtenPan}{\mbox{\sf S}}
\newc{\midten}{\mbox{\sf T}}
\newc{\midtenEi}{\mbox{\sf U}}
\newc{\ATen}{\mbox{\sf E}}
\newc{\BTen}{\mbox{\sf F}}
\newc{\CTen}{\mbox{\sf G}}

\def\sideremark#1{\ifvmode\leavevmode\fi\vadjust{\vbox to0pt{\vss
 \hbox to 0pt{\hskip\hsize\hskip1em
 \vbox{\hsize2cm\tiny\raggedright\pretolerance10000
  \noindent #1\hfill}\hss}\vbox to8pt{\vfil}\vss}}}

\numberwithin{equation}{section}

\newcommand{\hh}{{\hspace{.3mm}}}

\newcommand{\cc}{\boldsymbol{c}}

\newcommand{\sss}{\scriptscriptstyle}

\renewcommand\geq{\geqslant}
\renewcommand\leq{\leqslant}

\DeclareMathOperator{\EXT}{d}
\newcommand{\ext}{{\EXT\hspace{.01mm}}}

\stackMath
\newcommand\reallywidehat[1]{%
\savestack{\tmpbox}{\stretchto{%
  \scaleto{%
    \scalerel*[\widthof{\ensuremath{#1}}]{\kern-.6pt\bigwedge\kern-.6pt}%
    {\rule[-\textheight/2]{1ex}{\textheight}}
  }{\textheight}%
}{0.5ex}}%
\stackon[1pt]{#1}{\tmpbox}%
}

\definecolor{ao}{rgb}{0.0,0.0,1.0}
\definecolor{forest}{rgb}{0.0,0.3,0.0}
\definecolor{red}{rgb}{0.8, 0.0, 0.0}

\begin{document}

%
%
%


\renewcommand{\today}{}
\title{
{
 Conformal Fundamental Forms\\
 and the Asymptotically Poincar\'e--Einstein Condition 
}}
%
%
%

\author{ Samuel Blitz${}^\flat$, A. Rod Gover${}^\sharp$ \&  Andrew Waldron${}^\natural$}

\address{${}^\flat$
  Center for Quantum Mathematics and Physics (QMAP)\\
  Department  of Physics\\ 
  University of California\\
  Davis, CA95616, USA} 
   \email{shblitz@ucdavis.edu}
 
\address{${}^\sharp$
  Department of Mathematics\\
  The University of Auckland\\
  Private Bag 92019\\
  Auckland 1142\\
  New Zealand,  and\\
  Mathematical Sciences Institute, Australian National University, ACT 
  0200, Australia} \email{gover@math.auckland.ac.nz}
  
  \address{${}^{\natural}$
  Center for Quantum Mathematics and Physics (QMAP)\\
  Department of Mathematics\\ 
  University of California\\
  Davis, CA95616, USA} \email{wally@math.ucdavis.edu}

\vspace{10pt}

\renewcommand{\arraystretch}{1}

\begin{abstract}
  An important problem is to determine under which circumstances a metric on a conformally compact manifold  is conformal to a Poincar\'e--Einstein metric. 
  Such conformal rescalings are in general obstructed by conformal invariants of
  the boundary hypersurface embedding, the first of which is the
  trace-free second fundamental form and then, at the next order, the trace-free
   Fialkow tensor.  We show that these tensors are
  the lowest order examples in a sequence of conformally invariant
  higher fundamental forms determined by the data of a conformal
  hypersurface embedding. We give a construction of these canonical extrinsic curvatures. Our main result is that the vanishing of these fundamental forms is a
  necessary and sufficient condition for a conformally compact metric
  to be conformally related to an asymptotically Poincar\'e--Einstein
  metric.
More generally,  these higher fundamental forms are basic to the study of conformal hypersurface invariants.    
  Because Einstein metrics necessarily have constant scalar curvature, our method employs asymptotic solutions of the
singular Yamabe problem to select an asymptotically distinguished conformally compact metric.
Our approach relies on conformal tractor calculus as this 
  is key for an extension of the general theory of conformal
  hypersurface embeddings that we further develop  here.
In particular,
we give in full detail tractor analogs of the 
classical Gau\ss\ Formula and  Gau\ss\ Theorem for Riemannian hypersurface embeddings.

\medskip

\noindent
\begin{center}
\end{center}

\end{abstract}

\maketitle

\pagestyle{myheadings} \markboth{Blitz, Gover \& Waldron}{Conformal Hypersurfaces}

\thispagestyle{empty}

\newpage

\tableofcontents

\newcommand{\balpha}{{\bm \alpha}}
\newcommand{\balphas}{{\scalebox{.76}{${\bm \alpha}$}}}
\newcommand{\bnu}{{\bm \nu}}
\newcommand{\blambda}{{\bm \lambda}}
\newcommand{\bnus}{{\scalebox{.76}{${\bm \nu}$}}}
\newcommand{\bnuss}{\hh\hh\!{\scalebox{.56}{${\bm \nu}$}}}

\newcommand{\bmu}{{\bm \mu}}
\newcommand{\bmus}{{\scalebox{.76}{${\bm \mu}$}}}
\newcommand{\bmuss}{\hh\hh\!{\scalebox{.56}{${\bm \mu}$}}}

\newcommand{\btau}{{\bm \tau}}
\newcommand{\btaus}{{\scalebox{.76}{${\bm \tau}$}}}
\newcommand{\btauss}{\hh\hh\!{\scalebox{.56}{${\bm \tau}$}}}

\newcommand{\bsigma}{{\bm \sigma}}
\newcommand{\bsigmas}{{{\scalebox{.8}{${\bm \sigma}$}}}}
\newcommand{\bbeta}{{\bm \beta}}
\newcommand{\bbetas}{{\scalebox{.65}{${\bm \beta}$}}}

\renewcommand{\bS}{{\bm {\mathcal S}}}
\newcommand{\bB}{{\bm {\mathcal B}}}
\renewcommand{\bT}{{\bm {\mathcal T}}}
\newcommand{\bM}{{\bm {\mathcal M}}}

\newcommand{\go}{{\mathring{g}}}
\newcommand{\nuo}{{\mathring{\nu}}}
\newcommand{\alphao}{{\mathring{\alpha}}}

\newcommand{\Ell}{\mathscr{L}}
\newcommand{\density}[1]{[g\, ;\, #1]}

\renewcommand{\Dot}{{\scalebox{2}{$\cdot$}}}

\newcommand{\PanE}{P_{4}^{\sss\Sigma\hookrightarrow M}}
\newcommand\eqSig{ \mathrel{\overset{\makebox[0pt]{\mbox{\normalfont\tiny\sffamily $\Sigma$}}}{=}} }
\newcommand\eqtau{\mathrel{\overset{\makebox[0pt]{\mbox{\normalfont\tiny\sffamily $\tau$}}}{=}}}
\newcommand{\hd }{\hat{D}}
\newcommand{\hdb}{\hat{\bar{D}}}
\newcommand{\Two}{{{{\bf\rm I\hspace{-.2mm} I}}{\hspace{.2mm}}}{}}
\newcommand{\TwoN}{{\mathring{{\bf\rm I\hspace{-.2mm} I}}{\hspace{.2mm}}}{}}
\newcommand{\Fn}{\mathring{\mathcal{F}}}
\newcommand{\csdot}{\hspace{-0.55mm} \cdot \hspace{-0.55mm}}
\newcommand{\IdD}{(I \csdot \hd)}
\newcommand{\Kd}{\dot{K}}
\newcommand{\Kdd}{\ddot{K}}
\newcommand{\Kddd}{\dddot{K}}

 \newcommand{\bdot }{\mathop{\lower0.33ex\hbox{\LARGE$\cdot$}}}

\section{Introduction}
Codimension one embedded submanifolds, or {\em hypersurfaces}, in
Riemannian $d$-manifolds $(M^d\!,g)$ have long been understood to be important because (for example)
they arise naturally in foliations, and as the boundary of both domains and
manifolds with boundary. In this setting the construction of local
(and some global) invariants is well understood through the equations
of Gauss, Codazzi and Ricci
combined with Weyl's classical invariant theory.  The goal of this
article is the study of invariants of hypersurfaces embedded in a
conformal manifold $(M^d,\cc)$, 
where~$\cc$ denotes an equivalence class of smoothly conformally related metrics. The main results include a sequence of conformal fundamental
forms that should be viewed as higher order analogs of the
trace-free second fundamental form. They are natural symmetric
trace-free two tensors, on a hypersurface, that are extrinsic
conformal invariants. The higher forms capture higher jet data of the
interaction between the hypersurface and the ambient conformal
structure. Therefore they are well-suited to the  study of Poincar\'e--Einstein
metrics.

A manifold $(M\backslash \Sigma,g^o)$ with boundary~$\Sigma$ is said to be conformally compact when $g=s^2 g^o$ extends as smoothly as a metric to the boundary for any choice of {\it defining function} $s$ \linebreak for $\Sigma$---meaning that $\Sigma$ is the zero locus of $s\in C^\infty M$ and~$\ext s$ is nowhere zero along $\Sigma$.
Moreover, when the interior metric $g^o$ is Einstein, $(M\backslash \Sigma,g^o)$ is said to be {\it Poincar\'e--Einstein}.  These structures have attracted intense scrutiny, partly because of their relation to the AdS/CFT correspondence of~\cite{Maldacena} and anomalies~\cite{HS} in physics, as well as to basic geometric considerations~\cite{FG,GrahamLee,Biquard,Anderson,GZscatt,Albin,LeeFred,MazzPac,Chang,
FGbook}.
Given a conformally embedded hypersurface~$\Sigma\hookrightarrow (M,\cc)$, we say a
metric~$g^o$ on $M\backslash \Sigma$ is {\it asymptotically Poincar\'e--Einstein}
if  its  trace-free Schouten $\mathring P^{g^o}$ (or equivalently trace-free Ricci) tensor satisfies
\begin{equation}\label{aPE}
\mathring P^{g^o} = {\mathcal O}(s^{d-3})\, ,
\end{equation}
where 
$$
g^o=s^{-2} g\, ,
$$
for some $g\in \cc$ and $s$ a defining function for $\Sigma$. Here ${\mathcal O}(s^k)$ denotes any smooth rank two tensor times the function $s^k$.
We therefore ask the question: 
 when is
   the metric on a conformally compact manifold  conformal to an asymptotically Poincar\'e--Einstein metric?
More generally we treat the following problem:
\begin{problem}\label{P}
Given a conformally embedded hypersurface $\Sigma\hookrightarrow (M,\cc)$,
under what conditions does $\cc|_{M\backslash \Sigma}$ contain an  asymptotically Poincar\'e--Einstein metric?
\end{problem}
\noindent
It is well known that the embedding $\Sigma\hookrightarrow (M,\cc)$ must be umbilic meaning that the trace-free second fundamental form $\IIo$ 
must vanish ($\forall p\in \Sigma$)~\cite{Goal,LeBrun}. 
Thus, given a general conformal  embedding  $ \Sigma\hookrightarrow (M,\cc)$, its trace-free second fundamental form oughxewed as  an obstruction to the existence of a Poincar\'e--Einstein metric in the conformal class along $M\backslash\Sigma$; 
at the next order, an obstruction is the trace-free Fialkow tensor (of~\cite{Fialkow}; see also~\cite{Vyatkin}).
In Theorem~\ref{uber-PE} below we establish 
a necessary and sufficient condition for  solving Problem~\ref{P} in terms of vanishing conditions for higher fundamental forms, which we now discuss.

\medskip

For a hypersurface $\Sigma$, embedded in a dimension $d\geq 3$
Riemannian manifold $(M^d,g)$, the induced metric $\bar{g}$ on the hypersurface
is sometimes referred to as the first fundamental form. The second fundamental
form~$\II$ is a rank two, natural, symmetric tensor along $\Sigma$ that
captures the failure of the conormal to be parallel.
Such
diffeomorphism invariant tensors built from the metric $g$, its
derivatives, and the conormal are termed {\it natural}, see {\it
  e.g.}~\cite{Atiyah,Stredder,GPt,Will2}. 
 For simplicity we do not consider parity odd tensors constructed from the Levi--Civita symbol/Hodge dual.

The  part of $\II$  that is trace-free with respect to the induced metric~$\bar g$ is denoted $\IIo$. Then
$$
\IIo = \II - \bar g H\, ,
$$
defines the {\em mean curvature} $H$ of $\Sigma\hookrightarrow
(M,g)$. The {\em trace-free second fundamental form} $\IIo$ has the important property that when computed with respect
to a conformally related metric $\Omega^2 g$, with $0<\Omega\in
C^\infty M$, it obeys
$$\IIo^{\Omega^2 g} = \Omega \IIo^g\, .$$ The {\em Fialkow tensor} of expression~(\ref{Aaron}) below is another rank two tensor
with a similar conformal transformation.
 Such {\em
  conformal hypersurface invariants} (see \cite[Definition 4.1]{Will1} for a formal definition) are important because they
play a key {\it r\^ole} in  geometric PDE boundary problems (see for example~\cite{Case1, Case2,GPt,Bent}), in generalizing the scattering
program of Melrose, Graham-Zworksi and others~\cite{MELscatt,GZscatt}, and in the
mathematics of the AdS/CFT program~\cite{Maldacena,HS}.
Indeed, while the trace-free Schouten tensor of a metric $g^o=s^{-2}g$ on the interior of a conformally compact structure becomes singular as one approaches $\Sigma$, the quantity $s \mathring P^{g^o}$ is finite in this limit. Moreover this equals the trace-free part of the curvature-adjusted Hessian on $M\backslash \Sigma$ 
\begin{equation}\label{Rug}
(\nabla^2 s + s \Rho)_\circ\, .
\end{equation}
Here $\nabla$ is the Levi-Civita connection of $g$ and $\Rho$ is the corresponding Schouten tensor. Importantly the above display extends smoothly to the boundary and there equals the trace-free second fundamental form $\IIo$. This suggests building higher fundamental forms from   jets of the Hessian~\nn{Rug} transverse to~$\Sigma$.

Later, we define a notion of the {\em transverse order} of a natural
tensor along $\Sigma$ by suitably counting transverse derivatives of
the ambient metric $g$ that are used in the formula for the given
invariant along $\Sigma\hh$---see Section~\nn{herecomethI} for details.  The
 trace-free second fundamental form~$\IIo$  has
transverse order one and is the first example of what we
define to be a conformal fundamental form. The trace-free part of the Fialkow tensor 
is a transverse order two example.

\begin{definition} \label{ff-def}
 Let $n\in {\mathbb N}$.
A {\it conformal fundamental form} is {\it any} natural  trace-free section $L$ of~$\odot^2 T^*\Sigma$ with transverse order $n$ that obeys
$$
L^{\Omega^2 g} = \Omega^{2-n}L^g\, .
$$
\end{definition}

\noindent
We call a conformal fundamental form of transverse order $n-1>0$ an {\it $n$th conformal fundamental form}, and for brevity, often drop the adjective conformal when $n>2$.
Because, by their very definition, conformal fundamental forms measure how the ambient metric extends off of the hypersurface,  their leading transverse order terms are expressed in terms of ambient curvature quantities. Indeed, the first three non-trivial examples are made from the Weyl, Cotton, and Bach tensors, respectively.

In dimensions~$d\neq 3$, a third fundamental form is
$$
W_{\hat n ab \hat n}\, ,
$$
where $W$ is the Weyl tensor of $(M,g)$ and $\hat n$ is the unit conormal to $\Sigma$.
In dimensions $d\neq5$, a fourth fundamental form is 
\begin{equation}\label{mynumberisugly}
\tfrac1{d-5}\,  \bar \nabla^c W_{c(ab) \hat n}^\top+
C_{\hat  n(ab)}^\top  +H W_{\hat n ab \hat n} \, ,
\end{equation}
where $C$ is the Cotton tensor of $(M,g)$, $\top$ denotes the projection of $TM|_\Sigma$ to $T\Sigma$, and $\bar \nabla$ is the Levi-Civita connection of $(\Sigma,\bar g)$.
In dimensions $d\neq 3,5,7$, a fifth fundamental form is 
\begin{multline*}
\hspace{-.4cm}
-\tfrac {2(d-4)}{d-7}\, \big(\bar \Delta \Fo_{ab}-
\bar \nabla_{(a} \bar \nabla\csdot \Fo_{b)\circ}
-(d-3) \bar P_{(a} \csdot\, 
\Fo_{b)\circ}-2\bar J \Fo_{ab}\big)
-\tfrac{d-4}{d-5}\,  \bar{B}_{ab}
\\
+  B^\top_{(ab)\circ}  
- 2(d-4) H C^\top_{\hat n(ab)} 
-  (d-4)H^2 W_{\hat nab\hat n}  
- \tfrac{d-4}{6(d-1)(d-2)}\big( \bar{\nabla}_{(a} \bar{\nabla}_{b)\circ}\!-2\bar{\Rho}_{(ab)\circ} \big) \IIo^2\!
\\
+ \tfrac{d-4}{(d-2)(d-3)}  \TwoN_{ab} \big(\bar{\nabla}_c\bar{\nabla}_d
-(d-2)\bar \Rho_{cd}\big)
\TwoN^{cd} 
- \tfrac{d-4}{(d-2)^2} \bar{\nabla}\csdot\,  \TwoN_{(a} \bar{\nabla}\csdot\,  \TwoN_{b)\circ} 
 \, ,
\end{multline*}
where
 $B$ and $\bar B$ are the Bach tensors of $M$ and $\Sigma$, respectively (bars are universally used to denote hypersurface quantities and objects), see~\nn{notjohann}.
Our conventions for tensors and for the Schouten tensor $P$ and its trace $J$ are given in Section~\ref{RG}.
The above third and fourth fundamental forms vanish for generic embeddings in conformally flat spaces. It turns out that no fifth fundamental form exists with the same property.

\smallskip

The above tensor structures might at first seem arcane, but in fact they follow from a general
theory of conformally invariant normal operators along hypersurfaces.
The foundations of such were developed
in~\cite{GPt}, although important lacun\ae\ in precisely this setting
remain. We shall show that conformal fundamental forms are
 central to the theory of conformal
hypersurface embeddings. 
They are atoms for the construction of  {\it hypersurface invariants}, {\it i.e.}, natural tensors distinguished by
the property that they are unchanged, or transform in a simple
(``covariant'') way when the metric $g$ 
 is replaced by a conformally related metric
  $\widehat{g}=e^{2\omega} g$, with $\omega\in C^\infty(M)$.

\smallskip

The identification and construction of  conformal
hypersurface invariants is complicated on all fronts. On one hand,
the equations of Gauss, Codazzi, and Ricci are all badly
behaved under conformal transformation, and on the other, even
the construction of intrinsic conformal invariants is difficult and
incomplete~\cite{BEGr,GOadv}. 
There are two main tools which provide an effective route to resolving
these difficulties. The first is the conformally invariant local
calculus based around the conformal Cartan or tractor connection
\cite{BEG}. This enables a replacement of the Riemannian
Gauss-Codazzi-Ricci hypersurface theory with an analogous conformally
invariant set of basic equations and identities. While the basic
elements of this are established in~\cite{BEG,Grant, Stafford, Vyatkin, WillB, Will1, Will2}, gaps
have remained and we treat these  in Section~\ref{CEH}; see also Theorem~\ref{GTF} and its Corollary~\ref{GTE} below.  These results are crucial for the effective use of the second
main tool which is {\em holography}.

\medskip
The tractor approach to embedded hypersurfaces begins by extending the 
tangent bundle to~$M$ to a rank $d+2$ vector bundle $\ct M$ equipped with a Cartan connection $\nabla^\ct$ canonically determined  by the data $(M,\cc)$ which preserves a tractor metric $h:\Gamma(\odot^2\ct M)\to C^\infty M$. These are respectively termed the {\it (standard) tractor bundle} and {\it tractor connection}; the precise definitions are given in Section~\ref{t-sec}. Because the conformal class of metrics $\cc$ induces a conformal class of metrics~$\cc_\Sigma$ along~$\Sigma$, the hypersurface also comes equipped with a tractor bundle $\ct \Sigma$ and connection~$\bar \nabla^\ct$.  
It turns out that there is a simple relationship between $\ct M$ and $\ct \Sigma$~\cite{Goal}, and so it is natural to search for the conformal analogs of the Gau\ss\ formula and the Gau\ss\ equation for Riemannian hypersurface embeddings; early results for such were presented in~\cite{Will1,Will2,hypersurface_old}. Recall that the
Gau\ss\ formula and Gau\ss\ equation relate ambient and hypersurface connections and curvatures, respectively; see Equations~\nn{GF} and~\nn{GE}. To obtain optimal conformal analogs of these,
we use the Thomas-$D$ operator
which extends the tractor connection
to a  map from  tractors to tractors.
To understand this operator, one needs to enlarge the space of tractor objects to include sections of tensor products of tractor bundles as well as
weighted tractor bundles (obtained by tensoring tractor fields with conformal densities), see Section~\ref{c-sec}. 
Denoting a weight~$w$ tractor bundle of tensor type $\Phi$ by $\ct^\Phi M[w]$, the Thomas-$D$
operator is a map
$$D:\Gamma(\ct^\Phi M[w])\to \Gamma(\ct M\otimes \ct^\Phi M[w-1])\, .$$
A quick way to understand the Thomas-$D$ operator is that it is a conformally invariant
triplet of differential operators in correspondence with the weight, gradient and Laplace operators. Detailed formul\ae\
are given in Section~\ref{D-sec}. 
Because the Thomas-$D$ 
operator is a powerful tool for proliferating 
conformal invariants~\cite{BEG}, it is advantageous to have an analog of the Gau\ss\ formula relating the bulk and hypersurface Thomas-$D$ operators:
\begin{theorem}[Gau\ss--Thomas formula]\label{GTF}
Let $w+\frac d2 \neq 1,\frac 32,2$. 
Acting  on weight $w$ tractors, the bulk tangential and hypersurface Thomas-$D$ operators obey
$$
\hd^T_A\stackrel\Sigma =  \hdb_A +\Gamma_A {}^\sharp - \frac{X_A}{\bar{d} + 2w-2} \bigg{\{} 
2\Gamma_B {}^\sharp \circ \hdb^B + \Gamma^B{}^\sharp \circ \Gamma_B{}^\sharp 
+ \frac{1}{\bar{d}(\bar{d}-1)} \left[\left(\hdb K \right)  \wedge X \right]^\sharp - \frac{(3\bar{d}+2)wK}{2\bar{d}(\bar d-1)} \bigg{\}}\, .
$$
\end{theorem}
\noindent
The above explicit tractor result refines and simplifies the implicit one of~\cite{Will2}.
The notation $\stackrel\Sigma=$ denotes equalities that hold along the hypersurface.
The tangential Thomas-$D$ operator $\hd^T$ is introduced in Proposition~\ref{Dt-complex}.
Like the projection of the bulk Levi-Civita connection~$\nabla^\top$ along~$\Sigma$ to $T\Sigma$, the operator $\hat D^T$ obeys the tangentiality property 
$$\hat D^T U \stackrel\Sigma=\hat D^T U^\prime\, ,$$
for any pair of tractors $U$ and $U^\prime$ satisfying $U \eqSig U'$. The tractor $\Gamma_{ABC}$ is built from  the second and third conformal fundamental forms; see Equation~\nn{Gamma}. Also,~$K$ denotes $\IIo^2$
and is termed the {\it rigidity density}. Finally,  
the canonical tractor $X$  is defined 
in Section~\ref{t-sec}. Theorem~\ref{GTF} is proved  in Section~\ref{GTform-proof}.

The commutator of two Thomas-$D$ operators includes a curvature tractor called the $W$-tractor---see Proposition~\ref{DD-comm}---which unifies in a conformally covariant way, the Weyl, Cotton and Bach tensors~\cite{GOadv,GOpet}. It obeys an  analog of the Gau\ss\ theorem, which essentially follows as a corollary of the Gau\ss--Thomas formula:
\begin{corollary}[Gau\ss-Thomas Equation]\label{GTE}
Let $d>5$. Then the bulk and hypersurface $W$-tractors are related by
\begin{align}
\begin{split} \label{Wt-Wb}
W^\top_{ABCD}
\big|_\Sigma =\, \overline{W}_{ABCD}& - 2 L_{A[C} L_{D]B} - 2 \overline{h}_{A[C} F_{D]B} + 2 \overline{h}_{B[C} F_{D]A} - \tfrac{2}{(d-1)(d-2)} \overline{h}_{A[C} \overline{h}_{D]B} K \\[1mm]
&+ 2 X_{[A} T_{B]CD} +2 X_{[C} T_{D]AB}  
- 2 X_A X_{[C} V_{D]B} + 2 X_B X_{[C} V_{D]A} \\
&+ \tfrac{1}{3(d-1)(d-2)} X_A X_{[C} \hdb_{D]} \hdb_B K - \tfrac{1}{3(d-1)(d-2)} X_B X_{[C} \hdb_{D]} \hdb_A K,
\end{split}
\end{align}
where
\begin{align*}
T_{ABC} &:= 2 \hdb_{[C} F_{B]A} + \tfrac{1}{(d-1)(d-2)} \overline{h}_{A[B} \hdb_{C]} K\:\in\: \Gamma(\ct \Sigma \otimes \wedge^2 \ct \Sigma[-3] )\, , 
\end{align*}
and $V_{AB}\in \Gamma(\odot^2 \ct \Sigma[-4])$ is a symmetric tractor built from curvatures such that $X^A V_{AB} = X_B V$ for some $V\in \Gamma(\ce M[-4])$.
\end{corollary}
\noindent
The tractors $L_{AB}$ and $F_{AB}$ are defined in Section~\ref{HTC}. This corollary is proved in Section~\ref{GTE-proof}.

Our tractor hypersurface technology lays the groundwork for a forthcoming study of higher extrinsic Laplacian powers and generalized  Willmore energies~\cite{US}. One such energy that emerges directly from our fundamental form study is given below.

\begin{theorem}\label{goodwill}
Let $d=5$ and $\Sigma\hookrightarrow (M,\cc)$ be a closed conformally embedded hypersurface. Then the integral defined for the metric $\bar g$ induced by $g\in\cc$ and given by
\begin{equation}\label{Samaritan}
\int_\Sigma {\ext}{\rm Vol}(\bar g) \big(
-\tfrac12 \IIo \csdot \bar{\Delta} \IIo 
  +6 \IIo \csdot C_n^{\top}
  - 4 \IIo \csdot \bar{P} \csdot \IIo
  + \tfrac72 \bar{J} K 
 +6 H \IIo^3
  +6 H \IIo \csdot \IIIo\big)\, , 
  \end{equation}
 is independent of the choice of $g\in \cc$.
\end{theorem}

\smallskip

Theorem~\ref{GTF} and its Corollary~\ref{GTE} concern the relationship between hypersurface objects and their bulk counterparts. It is also interesting to study how extensions of hypersurface quantities to the bulk vary in directions transverse to the hypersurface. For this one would like to have a theory of conformally invariant normal operators. This has been developed in~\cite{GPt}. For example, if $f \in C^\infty M$ is some extension of $\bar f\in C^\infty \Sigma$, then a conformally invariant second order normal operator is given by
$$
\delta_R^{(2)} f =
\bar{\Delta}\bar f-
(d-3) \big( \hat n^a \hat n^b \nabla_a \nabla_b + H \nabla_{\hat n}\big)   
 f\big|_\Sigma\in \Gamma(\ce \Sigma[-2])\,.
$$
In Lemma~\ref{ho-robin}, in preparation for our study of conformal fundamental forms, we apply the theory of~\cite{GPt} to 
trace-free symmetric tensors.

\bigskip

We now turn to the second main tool: holography. Given an embedded hypersurface $\Sigma\hookrightarrow M$, the
holographic approach is to set up a PDE problem on the ambient
conformal manifold $M$ whose solution is determined sufficiently
accurately by the data $\Sigma\hookrightarrow M$ so that it encodes
 information about hypersurface invariants. Indeed, the
idea is that the PDE solution determines the ambient geometry, at least
to some order, so that  ambient invariants determine 
hypersurface invariants. 

For the construction of extrinsic conformal hypersurface invariants
one begins with the data of a conformal embedding $\Sigma\hookrightarrow(M,\cc)$ 
and seeks  a PDE problem that, given this data,  determines a distinguished metric (again, to some order) in the conformal class $\cc$. 
In~\cite{WillB,Will1}, it was shown that a suitable problem is the one of finding an asymptotic singular Yamabe metric. Uniqueness of solutions to the singular Yamabe problem on round~\cite{Loewner} and asymptotically hyperbolic structures~\cite{Aviles, Maz,ACF} is well understood and the appearance of hypersurface invariants as the obstruction to smooth solutions was first observed in~\cite{ACF}.
Detailed definitions are given in Section~\ref{herecomethI}. The key idea is to find a metric $g^o$ on $M\backslash\Sigma$ whose scalar curvature obeys
\begin{equation}\label{Scoy}
\Sc^{g^o}=-d(d+1) + {\mathcal O}(s^d)\, ,
\end{equation}
where
$$
g^o=s^{-2} g
$$
for some $g\in \cc$ and $s$ a defining function for $\Sigma$. The pair $(g,s)$ determines a distinguished conformal density $\sigma=[g;s]=[e^{2\omega} g;e^{\omega}s]\in \Gamma(\ce M[1])$ (again see Section~\ref{c-sec} for a detailed definition of a weight $w$ conformal density bundle~$\ce M[w]$) from which conformal hypersurface invariants can be efficiently extracted
using tractor calculus methods. 
This program is executed in the series of papers~\cite{WillB,Will1,Will2,hypersurface_old}.

Examining Equation~\nn{Scoy}, it is evident that the amount of extrinsic data encoded by an asymptotic singular Yamabe metric grows as the dimension increases.
On the other hand, the complexity of the extrinsic curvature quantities determined this way escalates rapidly with increasing dimension. 
This can be simplified and captured using 
the conformally invariant higher order normal derivative operators discussed above.

In contrast to an asymptotic singular Yamabe metric, which always exists for generic embedding data, the existence of an asymptotic Poincar\'e--Einstein metric restricts the allowed conformal embeddings. 
This motivates our search for further
conformally invariant obstructions to the existence of Poincar\'e--Einstein metrics. Indeed,
given a defining function $s$ for $\Sigma$, the standard approach, following~\cite{FG}, is to use $s$ as a coordinate 
in a collar neighborhood $I\times \Sigma$, with $I\ni s$, of the boundary and 
then attempt to solve 
 the asymptotic Poincar\'e--Einstein condition~\nn{aPE} 
via a
``Fefferman--Graham'' expansion
$$
g=\bar g + s h_1 + \cdots + s^{d-2} h_{d-2} + {\mathcal O}(s^{d-1})\, .
$$
Here, $\bar g$, $h_1,\ldots, h_{d-2}$ are symmetric rank two tensors on $M$ that are in the kernel of ${\mathcal L}_{\frac{\partial}{\partial s}}$ and~$\iota_{\frac{\partial}{\partial s}} h_k=0$. The tensor $h_1|_\Sigma$   is the second fundamental form of the Riemannian embedding $\Sigma\hookrightarrow (M,g)$ and there always exists a choice of $g\in\cc$ such that this embedding is minimal and~$h_1|_\Sigma$ is  the trace-free second fundamental form. 
A general conformal embedding is not umbilic but, as discussed above, does determine an asymptotic singular Yamabe metric $g^o=g/s^2$. Thus, asking whether the higher tensors $h_k$ with $k\in \{2,\cdots,d-2\}$ are compatible with the Poincar\'e--Einstein condition
gives another perspective for why
 the trace-free second fundamental form $\IIo$ is the first in a sequence of~$d-2$ 
conformally invariant 
tensors characterizing  
the local obstruction to a conformal embedding being asymptotically Poincar\'e--Einstein. 
This sequence of fundamental forms is provided by the  jets of the Hessian expression in Equation~\ref{Rug}; see Section~\ref{can-ext-IIo}.

Our main result establishes that conformal fundamental forms  are the obstruction to a conformal embedding admitting an asymptotically Poincar\'e--Einstein metric. However, the detailed picture here is both rich and subtle. For example, in $d=5$  dimensions the fourth fundamental form displayed in Equation~\nn{mynumberisugly} is  ill-defined. However, for umbilic embeddings, the hypersurface invariant
\begin{equation}
\label{mynumberisveryugly}
C_{\hat  n(ab)}^\top  +H W_{\hat n ab \hat n} 
\end{equation}
is both conformally invariant and has transverse order three. We call a transverse order $n-1$ hypersurface invariant a {\it conditional fundamental form} iff it is an $n$th conformal fundamental form on embeddings  for which some lower transverse order fundamental form vanishes.
This suggests two generalizations of  umbilicity:

\begin{definition}\label{SLAM}
 We say that a conformal hypersurface embedding $\Sigma \hookrightarrow (M^d,\cc)$ is \textit{hyperumbilic} if, 
 for  each and every $n\in \{2,\ldots, \lceil\frac{d+1}{2} \rceil\}$, an $n$th fundamental form vanishes.
\end{definition}

\medskip 

\noindent
Lemma~\ref{Sparky}) establishes that the above and following definitions are well-defined.
As an application,  it follows from Theorem~\ref{GTF}
that $\hd^T_A\stackrel\Sigma =  \hdb_A$ for hyperumbilic embeddings  in dimensions~$d\geq 4$.

\begin{definition}\label{DUNK}
We say that a conformal hypersurface embedding $\Sigma \hookrightarrow (M^d, \cc)$ is \textit{\"uberumbilic} if the embedding is hyperumbilic 
and,
 for  each and every $n\in \{ \lceil\frac{d+3}{2} \rceil,\ldots, d-1\}$, an $n$th conditional fundamental form vanishes.
\end{definition}

The above definitions concern the second through $(d-1)$th fundamental forms. At the weight where one would expect a $d$th fundamental form, 
an obstruction to the Poincar\'e--Einstein problem of a different nature is known to appear, namely the  symmetric, trace-free,
conformally invariant  Fefferman--Graham tensor~\cite{FG}. This tensor is intrinsic to the boundary conformal geometry~$(\Sigma,\cc_\Sigma)$ and must necessarily vanish for $\cc|_{M\backslash\Sigma}$ to contain a smooth Poincar\'e--Einstein metric. The appearance of  conditional fundamental forms for $n\geq \lceil\frac{d+3}{2} \rceil$
in Definition~\ref{DUNK} is a feature of the theory of normal operators of~\cite{GPt}, see in particular Lemma~\ref{tr-deg}.

\smallskip
Our main result solves Problem~\ref{P} in terms  of Definition~\ref{DUNK}:
\begin{theorem} \label{uber-PE}
A conformal class of metrics $\cc$ admits
 an asymptotic Poincar\'e--Einstein metric 
 iff the conformal embedding~$\Sigma \hookrightarrow (M^d, \cc)$  is  \"uberumbilic.
\end{theorem}

\subsection{Riemannian Geometry}\label{RG}

In this section we outline 
our notations for Riemannian  structures. When their meaning is obvious, these conventions will be extended to other bundles. Throughout the article, $M$ will denote a $d$-manifold, which we equip either with a Riemannian metric $g$ or a conformal class of metrics $\cc = [g]=[\Omega^2 g]$ (where $0 < \Omega\in C^\infty M$). 
Unless otherwise indicated, all structures will be assumed smooth. Also, while we work exclusively in Riemannian signature, many results carry over to the pseudo-Riemannian setting upon obvious adjustments.
To avoid confusion, the exterior derivative is denoted by~$ \ext$. For a given metric~$g$, the Riemannian curvature tensor~$R$ for the associated 
Levi-Civita connection~$\nabla$ (we adorn~$\nabla$ and other metric dependent operators/tensors with a superscript $g$ when this dependence is not clear) is defined by
$$
R(x,y)z=(\nabla_x\nabla_y - \nabla_y \nabla_x) z -\nabla_{[x,y]} z\, , 
$$
where $[\bdot,\bdot]$ is the Lie bracket and $x,y,z\in 
 \Gamma(TM)$
are (smooth) vector fields.

We will use an abstract index notation  in which a section of $TM$ is denoted~$v^a$ and, for example~$R(x,y)z$ becomes $x^c y^d R_{cd}{}^a{}_b z^b$. This has the advantage of directly yielding formula applicable to the case where local coordinate choices have been made. In this notation, the isomorphism between tangent and cotangent bundle $TM$ and $T^*M$ given by the metric tensor~$g_{ab}$ allows us to ``raise and lower indices'' and compute traces  in the standard way. A Kronecker delta~$\delta^a_b$ is used to denote the identity endomorphism of the tangent bundle. A dot will be used for abstract index contractions, so that $v^a w_a=:v\hh\hh\csdot\hh\hh w$, $\nabla_a v^a =: \nabla\hh \csdot \hh\hh v$ and $v\hh\csdot\hh v :=v^2$, this notation will also be applied to higher rank symmetric tensors in an obvious way. 
Another succinct contraction notation that we use is to write, for example,~$x^c y^d R_{cd}{}^a{}_b z^b=R_{xy}{}^a{}_z$ or even~$x^{ab}u_a w_b = x_{uw}$ and~$x_{uu}=(u\cdot)^2 x$.
We employ  bracket notations for groups of indices that are totally skew symmetric or symmetric. For example, $x^{ab}=x^{[ab]}+x^{(ab)}$ where $x^{[ab]}:=\frac12 (x^{ab}-x^{ba})$ and $x^{(ab)}:=\frac12 (x^{ab}+x^{ba})$. 
Also $u\wedge v$ denotes $\frac12 (u^a v^b - u^b v^a)$.
A $\circ$ will be used to denote the trace-free part of group of indices, so that
$
x_{(ab)\circ}:=x_{(ab)}-\frac1d g_{ab} x^c{}_c
$. We use standard $\wedge$ and~$\odot$ notations for antisymmetric and symmetric tensor products of vector bundles and~$\odot_\circ$ for  a further decomposition to the trace-free part. 
Given a tensor $v^{abc\cdots e}$  we will use the shorthand ${\mathcal E}(v)$ to denote $v^{abc\cdots e}  t$ where $t$ is an unspecified tensor or tensor-valued operator and $a,b,c,\ldots,e$ are any open indices.
Also,
given an endomorphism acting on sections of  a vector bundle,
 we define its  zeroth 
power  to be the identity.

When specializing to conformal geometries, the decomposition
$$
R_{abcd}=W_{abcd} + g_{ac} P_{bd} 
- g_{bc} P_{ad} 
+ g_{bd} P_{ac} 
- g_{ad} P_{bc}\, , 
$$
where $W_{abcd}$ is the trace-free {\it Weyl tensor}, is particularly useful. In the above, $P_{ab}$ is the symmetric {\it Schouten tensor}, which is related to the {\it Ricci tensor} $Ric_{bd}=R_{ab}{}^a{}_{d}$ in dimensions $d\geq 3$ via the trace adjustment
$$
Ric_{ab}=(d-2)P_{ab}+g_{ab} J\, ,\quad J=P_a^a\, .
$$
We mostly deal with the case $d>2$, but in two dimensions we define $J=Sc/2$ where $Sc:=Ric_a^a$ is the scalar curvature in any dimension. The covariant curl of the Schouten tensor gives the {\it Cotton tensor}
$$
C_{abc}=\nabla_a P_{bc}-\nabla_b P_{ac}\, .
$$

\section{Tractor Calculus}

A conformal structure $(M,\cc)$ is the data of a smooth $d$-manifold
$M$ equipped with a conformal (equivalence) class $\cc$ of Riemannian
metrics meaning that if $\cc$ is a non-empty  set of Riemannian metrics on
$M$, and if $g\in \cc$, then $\hat{g}\in \cc$ iff $\hat{g}=\Omega^2 g$
for some smooth positive function $\Omega\in C^\infty M$. Such
$(M,\cc)$ have, in general, no distinguished connection on the tangent
bundle but there is a canonical connection on a related bundle of rank
$d+2$. This conformal {\em tractor connection} and surrounding tractor
calculus provide the analog for conformal geometry of the
Levi-Civita connection and its (``Ricci'') calculus on Riemannian
manifolds \cite{BEG,CurryGo}.  To treat this carefully we require the notion of a
conformal density.

\subsection{Conformal Densities}\label{c-sec}

A weight $w$ conformal density is a section of a certain line bundle~$\ce M[w]$ defined as follows: A conformal manifold $(M,\cc)$ may be
viewed as a ray subbundle ${\mathcal G}\in \odot^2T^*M$ with fiber at
$P\in M$ given by all possible values of $g_P$ for $g\in\cc$. The
bundle ${\mathcal G}$ is a principal bundle with structure group
${\mathbb{R}}_+$. The density bundle $\ce M[w]$ is the line bundle associated to ${\mathcal G}$ via the irreducible representation ${\mathbb R}_+\in
t\mapsto t^{-\frac w2}\in {\rm End}({\mathbb R})$, for each $w\in
{\mathbb R}$.  Then 
sections of $\Gamma(\ce M[w])$ are equivalent to functions $F$ of
${\mathcal G}$ with the homogeneity property
$$
 F(P,\Omega^2 g)=\Omega^w  F(P,g)\, .
 $$
 Alternately, one may view these as equivalence classes $F = [g;f]=
 [\Omega^2 g; \Omega^w f]$ where $f \in C^\infty M$.  Note that the
 section space of the bundle $\ce M[0]$ is isomorphic to the space of
 smooth functions~$C^\infty M$; we often make this identification
 without comment.  Importantly, a nowhere vanishing weight $w=1$
 conformal density $ \tau=[g;t]$ canonically determines a metric
 $g_\tau\in \cc$ by choosing an equivalence class representative
 $\tau= [g_\tau;1]$. We refer to such a density as a {\it true scale}
 and the corresponding metric in~$\cc$ as a choice of scale.
We will
 also employ the term {\it scale} for weight $w=1$ conformal densities
 whose zero locus is a hypersurface $\Sigma$, since these determine a
 metric on $M\setminus \Sigma$.  
As a point of
 notation, given a vector bundle $\mathcal{B}$, we use
 $\mathcal{B}[w]$ as shorthand for~$\mathcal{B} \otimes \ce M[w]$.
Note that $\bg:=\tau^2g_\tau \in \odot^2T^* M [2]$ is independent of the choice of true scale $\tau$ and is called the {\it conformal metric}, which will be used for index raising and lowering in the obvious way. 
Note that on densities we get a connection 
on densities of weight $w$ by acting with the operator $\tau^{w} \circ \ext \circ \tau^{-w}$. In fact this is the Levi-Civita connection (of the bundle $\ce M[w])$ for the metric $g^\tau$ determined by the scale $\tau$.

A conformal structure also determines log-density bundles, $\mathcal{F}M[w]$; see~\cite{GW} for details. A weight-$w$ \textit{log-density} $\lambda \in \Gamma(\mathcal{F}M[w])$ is an equivalence class of (metric,function) pairs $\lambda = [g; \ell] = [\Omega^2 g; \ell + w \log\Omega]$. If~$\varphi= [g; f] $ is
a positive, weight~$w$ density $0<\varphi \in \Gamma(\ce M[w])$, then $\log \varphi = [g; \log f]$ is a weight-$w$ log density. Also if $\lambda$ is a weight $w$ log density, then for any $r\in {\mathbb R}$, the product $r\lambda$ is a weight $rw$ log density;  in many contexts it thus suffices to focus on weight~$1$ log densities.
\color{black}

\subsection{Tractor Bundles}\label{t-sec}

As noted above, on a general conformal manifold $(M,\cc)$ with $d \geq 3$, there is a canonical
connection on a closely related natural vector bundle of rank $d+2$.
This bundle~$\ct M$  is termed the {\it standard tractor bundle}.  A choice of metric~$g\in \cc$
determines an
isomorphism
\begin{equation}\label{iso}
\ct M \stackrel g\cong 
 \ce M[1]\oplus TM[-1]\oplus \ce M[-1]\, .
 \end{equation}
 This use of the metric to provide this isomorphism  is called a {\em choice
   of splitting}.  An abstract index notation (using capital latin letters) is often used for sections of the
 tractor bundle so, for example,~$T^A\in \Gamma(\ct
 M)$ denotes a standard tractor.  Given $g\in \cc$ and the corresponding
 splitting as in the above display, we can write $T^A\stackrel g =
 (t^+,t^a,t^-)$; we often employ a column vector notation for this and
 will be careful to display the metric dependence unless
 the context makes it clear.  Choosing a conformally related metric
 $\Omega^2 g$, with $\Omega$ a positive smooth function, the
 isomorphism changes according to 
$$
T^A\stackrel{\Omega^2 g}= (t^+,t^a + \Upsilon^a t^+, t^- - \Upsilon\csdot t -\frac 12 |\Upsilon|^2_g \, t^+)\, ,
$$
where $\Upsilon = \Omega^{-1} \ext \Omega$. General 
 tensor products of $\ct M$ with itself are often
denoted~$\ct^\Phi M$, where~$\Phi$ labels the particular tensor
product under consideration. Further tensoring with a weight $w$
conformal density bundle gives (weighted) tractor bundles $\ct^\Phi
M[w]:=\ct^\Phi M\otimes \ce M[w]$. For brevity we often 
also call these
{\it tractor bundles}.  The first non-zero component, as determined by the isomorphism~\nn{iso} above, of a tractor
$T^A\stackrel g = (t^+,t^a,t^-)$ is called its {\it projecting part},
 because it is necessarily conformally
invariant; this is denoted $q^*(T)$. The same notion extends to general tractor
tensors~\cite{BEG} (although  the projecting part may no longer be an irreducible element of the tensor bundle of~$M$). We term $q^*$ the {\it extraction
  map}.

\medskip
The tractor bundles $\ct M[1]$ and $\odot^2\ct^* M$ are endowed
with distinguished canonical sections~$X^A$ and~$h_{AB}$, respectively
termed the {\it canonical tractor} and {\it tractor metric}. In any
choice of splitting these are given by, 
respectively, 
$$
X^A\stackrel{g}=\begin{pmatrix}0\\0\\1
\end{pmatrix}\, ,\qquad
h_{AB}\stackrel{g}=\begin{pmatrix}
0&0&1\\
0& \bg_{ab}&0\\
1&0&0
\end{pmatrix}\, ,
$$
and the inverse tractor metric is denoted by $h^{AB}$. The tractor metric provides an isomorphism between $\ct M$ and its dual $\ct^* M \stackrel{g}{=} \ce
 M[1]\oplus T^*M[1]\oplus \ce M[-1]$.
Note that the canonical tractor is null, {\it i.e.}, $X^2:=h(X,X)=0$.

\smallskip 

The canonical tractor gives the complex
$$
\ce M[-1]
\stackrel {X^A}\longrightarrow
\ct M
\stackrel{X_{\!A}\,}
\longrightarrow
\ce M[1]\, .
$$
In the above, the first map is multiplication by the canonical  tractor. The second map denotes  contraction by $X_A:=h_{AB}X^B$. Here and henceforth, we employ the tractor metric to raise and lower tractor indices and to identify the tractor bundle with its dual. The above sequence of maps underlies the isomorphism of Equation~\nn{iso}. For later use, note that we will employ a tilde notation $\widetilde V^A :=[V^A]=[V^A+X^A U]$, where $U\in \Gamma(\ce M[-1])$ for elements of the cokernel of the map $X^A$ above. Note that, because the skew product of two canonical tractors is zero, then~$X^{[A} \widetilde V^{B]}:=
X^{[A}V^{B]}$ defines a section of $\wedge^2 \ct M[1]$. 
We will  employ a tilde notation for higher tensor analogs of  this cokernel. We will even use analogs of the notation $X^{[A} \widetilde V^{B]}$ involving sums of tilded terms when it can be proved that such combinations define tractor tensors.

The tractor bundle has a canonical {\it tractor connection}
$$\nabla^\ct:\Gamma(\ct M) \to \Gamma(\ct M \otimes T^*M)$$ 
that may be viewed as the conformal analog of the Levi--Civita connection. In the choice of splitting determined by $g\in\cc$, this acts according to
\begin{equation}\label{Iconnect}
\nabla_a^\ct T^B \stackrel g= \begin{pmatrix}\nabla_a t^+- t_a\\ \nabla_a t^b + \delta_{a}^b t^- + \Rho_{a}^b t^+
\\
\nabla_a t^- - P_{ab} t^b
\end{pmatrix}\, .
\end{equation}

\medskip
Recall from above that a 
true scale $\tau\in \Gamma(\ce M[1])$ canonically
defines a metric $g_\tau\in \cc$. Combining $\tau$, the canonical tractor, and connection, we may then form the (one-form)-valued standard tractor
$$
Z_a^A:=\tau \nabla_a^\ct \big(\tau^{-1}X^B\big) \stackrel {g_\tau}= \begin{pmatrix}0\\  \delta_{a}^b 
\\
0\end{pmatrix}\, .
$$
In terms of the tractor connection
coupled to the Levi--Civita connection (of the metric $g^\tau$)---which we henceforth denote simply~$\nabla$---the above display reads $Z=\nabla X$.

Together with the tractor metric, for each true scale $\tau$,  we then can uniquely define the weight~$-1$ tractor $Y^A$ by the decomposition
$$
h^{AB}=X^A Y^B + \bm g^{ab}
 Z_a^AZ_b^B + X^B Y^A\, .
$$
We refer to the triplet $(X,Z,Y)$ above as {\it injecting operators} (or {\it injectors}) because  given a standard tractor in a choice of scale
$T^A\stackrel {g_\tau} = (t^+,t^a,t^-)$, we may write
$$
T^A = t^+ Y^A + t^a Z_a^A + t^- X^A\, .
$$
For later use, let us record how the tractor connection acts on the injectors
\begin{align*}
\nabla _b X^A &= \:\:Z^A_b\, ,\\
 \nabla_b Z^A_a\: \:\:  &= -P_{ab}X^A  -   \bm g_{ab} Y^A\, , \\
\nabla_b  Y^A &= \:\: P^a_b Z^A_a\, .  
\end{align*}
This is simply a rewriting of Equation~\nn{Iconnect}; in the above the tensors $\Rho:=\Rho^{g^\tau}$ is determined by the scale $\tau$. 

By way of notation, given a tractor such as $T^{ABC\cdots}{}_{EFG\cdots}$, then a  quantity such as
$T^{AbC\cdots}{}_{EFg\cdots}$ denotes~$Z^b_B Z_g^G\, T^{ABC\cdots}{}_{EFG\cdots} $ so that in the running example above $T^a = t^a$. Moreover quantities like $X_A T^{ABC\cdots}$ will be denoted $X^{+BC}$ and $T_{EFG} Y^E$ as $T^-{}_{FG}$.
The next section develops machinery
 for efficiently passing between Riemannian quantities and their tractor counterparts.

\newcommand{\nabt}[1]{\stackrel{#1} \nabla \hspace{-1mm}{}^\ct}

\subsection{The Thomas-$D$ Operator}\label{D-sec}

The Levi--Civita-coupled tractor connection~$\nabla$ can be used to define an invariant, second order operator on tractors that plays the {\it r\^ole} of a conformal gradient.
This is termed the 
{\it Thomas-$D$ operator} which is the map 
$$\Gamma(\ct^\Phi M[w])\to \Gamma(\ct M\otimes \ct^\Phi M[w-1])$$
defined acting on $T\in \Gamma(\ct^\Phi M[w])$ by, in a scale $\tau$,
\begin{equation}
\label{vaccinateThomas}
D^A T := (d+2w-2)w\hh Y^A T + (d+2w-2)\bm g^{ab}Z^{A}_a \nabla_b T - X^A (\Delta+w J^{g_\tau})T\, .
\end{equation}
In the above $\Delta := \bm g^{ab} \nabla_a \hh \nabla_b$ and the injectors $(Y,Z)$ are determined in terms of $\tau$ as explained above.

\smallskip

The Levi--Civita connection can also act on log densities: 
for $\lambda\in  \Gamma({\mathcal F}M[w])$, we define, in the scale $\tau$,
$$
\nabla \lambda :=\ext (\lambda - w \log \tau )
\in \Gamma(T^*M)\, .
$$ 
Here we used that the combination $\lambda - w \log \tau$ defines an element of $C^\infty M$ in the obvious way. Note that 
$\nabla^{\tau'} \lambda$ in a scale $\tau'=\Omega\tau$ equals $\nabla^{\tau} \lambda - w \Upsilon$, where $\Upsilon = \ext \log \Omega$.

Acting on weight $w$ log densities we define
$$
D:\Gamma({\mathcal F}[w])\to \Gamma(\ct M[-1])
$$
acting on $\lambda\in \Gamma({\mathcal F}[w])$ by (in the scale $\tau$)
\begin{equation}
\label{vaccinateThomas-log}
D^A \lambda := (d-2)w\hh Y^A \lambda + (d-2)\bm g^{ab}Z^{A}_a \nabla_b \lambda
 - X^A (\Delta\lambda+w J^{g_\tau})\, .
\end{equation}

\begin{lemma}\label{Thomas-D}
$D^A T$ and $D^A \lambda$, as respectively given in Equations~\nn{vaccinateThomas} and~\nn{vaccinateThomas-log}, are defined independently of the choice of true scale $\tau$.
\end{lemma}
\begin{proof}
The proof of this lemma is contained in \cite[Section 2]{GOpet} for the case $D^AT$ where $T$ is a tractor; the log-density statement can established using the same proof {\it mutatis mutandis}.\end{proof}
\begin{remark}
Given the above lemma, we will sometimes write abbreviated formul\ae\ such as 
$$
D^A T = \begin{pmatrix}
w(d+2w-2) T\\[1mm]
(d+2w-2) \nabla T\\[1mm]
-(\Delta+ w J) T
\end{pmatrix}
$$
for the action of the Thomas-$D$ operator on weighted tractors. 
We also omit the decoration $\tau$ when the dependence on it drops out.
\end{remark}

\noindent
A remarkable property of the Thomas-$D$ operator is that it is null~\cite{GOpar}, meaning 
$$D^A \circ D_A = 0\, .$$ 

Because the Thomas-$D$ operator is second order, it does not obey a Leibniz rule. However, upon a slight modification, the failure of the Leibniz property is rather mild. For that we first make a definition:
\begin{definition}
Suppose that $w \neq 1 - \tfrac{d}{2}$. The {\it hatted-$D$ operator}
$$\hd^A : \Gamma(\ct^\Phi M[w]) \rightarrow \Gamma(\ct M\otimes \ct^\Phi M[w-1])$$
is defined acting on $T\in\Gamma(\ct^\Phi M[w]) $  by
$$\hd^A T := \frac{1}{d+2w-2} \hh D^A T\,.$$
\end{definition}

\noindent
The following result of~\cite{Taronna} is proved in~\cite{Will1}.
\begin{proposition}\label{leib-failure}
Let $T_i \in \Gamma(\ce^\Phi M[w_i])$ for $i = 1,2$, and $h_i := d+2w_i$, $h_{12} := d+2w_1 + 2w_2 - 2$ with $h_i \neq 0 \neq h_{12}$. Then,
$$\hd^A (T_1 T_2) - (\hd^A T_1) T_2 - T_1 (\hd^A T_2) = -\frac{2}{h_{12}} X^A (\hd_B T_1) (\hd^B T_2)\,. $$
\end{proposition}
\noindent
Another important identity is
\begin{equation}\label{DX}
\hd^A X^B = h^{AB}\, .
\end{equation}
Using that
$$
X^A D_A T=w(d+2w-2) T\, ,
$$
for any weight $w$ tractor, Equation~\nn{DX} is then an easy corollary of Proposition~\ref{leib-failure}.

\medskip

It is useful to introduce the {\it weight operator} $\underline{w}$
defined on sections $T \in \Gamma(\ce M[w])$ by $\underline{w} T = wT$. 
It is easy to check that $\underline{w} $ is  a derivation acting on tractors. Acting on log-densities~$\lambda \in \Gamma(\mathcal{F}M[w])$, we define (see~\cite{GW})
$$
\underline{w} \lambda=w\, .
$$
Moreover, for $T \in \Gamma(\ct^\Phi M[w])$ and $\lambda \in \Gamma(\mathcal{F}M[w'])$, we extend  $\underline{w}$ (and similarly $\nabla$) to act on the product $\lambda T\in \Gamma( \ct^\Phi M[w] \otimes \mathcal{F}M[w'])$ by the Leibniz property:
$$\underline{w} (\lambda T) 
= \lambda \underline{w} T + T \underline{w} \lambda
=w \lambda T + w' T  \in \Gamma( \ct^\Phi M[w] \, \oplus\, \ct^\Phi M[w] \otimes \mathcal{F}M[w'])\,.$$

We can now write a universal formula for the Thomas-$D$ operator on any tensor product bundle of density and log-density bundles:
$$D^A := Y^A (d+2\underline{w}-2)\underline{w}\hh + \bm g^{ab}Z^{A}_a \nabla_b (d+2\underline{w}-2) - X^A (\Delta+J\underline{w})\, .$$

Using these notations, acting on weight $w\neq 1-\frac d2$ tractors, the operator $\hd$ can be written in terms of $D$ via $\hat D= D \circ \frac{1}{d+2\underline w-2}$, where we define the operator $\frac{1}{d+2\underline w-k}$ on a weight $w'$  log-density~$\lambda$ (for $k\neq d)$ by
\begin{equation}\label{theonebeforethenextone}
\frac1{d+2\underline w-k} 
\,  \lambda := \frac{\lambda}{d-k}-\frac{2w'}{(d-k)^2}\in \Gamma\big(
\mathcal{F}M[w']\oplus\ce M[0]\big)\, ,
\end{equation}
and
\begin{equation}\label{w-expansion}\frac1{d+2\underline w-k} \, (\lambda T) 
:= T \frac1{d+2\underline w+2w-k} \hh
\lambda   \in \Gamma( \ct^\Phi M[w] \, \oplus\, \ct^\Phi M[w] \otimes \mathcal{F}M[w'])\,.
\end{equation}
\begin{remark}
It is even possible to define
an operator $\frac{1}{\alpha \underline{w}+\beta}$ acting on sections of 
$\otimes^k \mathcal{F}M[w']$
by appealing to the formal power series expression
\begin{align*} \label{w-expansion}
\frac{1}{\alpha \underline{w} + \beta} = \frac{1}{\beta} - \frac{\alpha}{\beta^2} \, \underline{w} + \frac{\alpha^2}{\beta^3}\,  \underline{w}^2 + \ldots\, ,
\end{align*}
since only finitely many terms are ever needed.
\end{remark}

In general we prefer to avoid working with sections of Whitney sum bundles such as those arising when the inverse weight operator above acts on log-density bundles. Happily, the following result shows that composing the Thomas-$D$ operator with the action of 
$\frac{1}{d+2\underline w-2}$ on log densities results in a tractor.
\begin{lemma}
For any $0\neq \beta \in \mathbb{R}$,
$$\Big(D\circ \frac{1}{\alpha \underline w + \beta}\Big)\lambda =\frac1{\beta}
D \lambda
\in \Gamma(\ct M[-1])\, . 
$$
\end{lemma}

\begin{proof}
This follows from a straightforward computation in a choice of scale, or otherwise can be seen as a direct consequence of Equation~\nn{theonebeforethenextone}. 
\end{proof}

There is an analog of Proposition~\nn{leib-failure} for the algebra of Thomas-$D$ operators and log densities; this is another case where the need for Whitney bundles is obviated.

\begin{lemma}\label{Dlog} 
Let $\lambda$ be any log density and let $w \neq \frac{2-d}{2}$. Then,
$$\hd \circ \lambda- \lambda \circ \hd : \Gamma(\ct^\Phi M[w]) \rightarrow \Gamma(\ct M \otimes \ct^\Phi M[w-1])\,,$$
 and moreover 
$$\hd \circ \lambda- \lambda \circ  \hd =  
(\hd \lambda) 
- \tfrac{2}{d+2 \underline{w}}  X\hh (\hd \lambda)\hh \csdot \hh \hd\,.$$
\end{lemma}
\begin{proof} 
This is an easy computation in a choice of scale.
\end{proof}

\medskip

A particularly  useful application of the Thomas-$D$ operator is  the construction of tractors from 
 Riemannian tensors.

\subsection{Tractor Insertion}
Given a conformal density-valued, weight  $w+r$, trace-free, rank $r$ Riemannian tensor $t_{abc\cdots}$, there exists a canonical map (see for example~\cite{GOadv}) to insert $t_{abc\cdots}$ into a tractor $T^{ABC\cdots}$ with the same symmetries as $t_{abc\cdots}$. This map is denoted by $q$; there are three particular instances relevant for our computations:

\begin{lemma} \label{insertions} 
Let $g\in \cc$.
\begin{enumerate}[(i)]
\item
Given $v_a\in \Gamma(T^*M[w+1])$ where $w\neq 1-d$, then
\begin{align*}
q(v_a) &=: V^A \in\Gamma(\ct M[w]) \\[1mm]
&\stackrel{g}= \begin{pmatrix}
0 \\
v^a \\
- \frac{\nabla \cdot v}{d + w-1} \end{pmatrix}\, ,
\end{align*}
where
$$
D_A V^A = X_A V^A =0\, .
$$
\item
Given $t_{ab}\in \Gamma(\odot_\circ^2T^*M[w+2])$ where $w\neq -d,1-d$, then
\begin{align*}
q(t_{ab}) &=: T^{AB} \in\Gamma(\odot^2_\circ \ct M[w]) \\[1mm]
&\stackrel{g}= 
\begin{pmatrix}
0 & 0 & 0 \\
0 &t^{ab} & -\frac{\nabla \cdot t^a}{d+w} \\
0 & -\frac{\nabla \cdot t^b}{d+w} & \frac{\nabla \cdot \nabla \cdot t + (d+w) P_{ab} t^{ab}}{(d+w)(d+w-1)} \end{pmatrix}\, ,
\end{align*}
where
$$
D_A T^{AB} =0= X_A T^{AB} \, .
$$
\item
Given $t_{abcd}\in \Gamma(\otimes^4 T^*M[w+4])$, where $w\neq 1-d,2-d$, such that $t$ has the algebraic symmetries of the Riemann tensor and is trace-free, then
$$
q(t_{abcd}) =: T^{ABCD} \in\Gamma(\otimes^4\ct M[w])
$$
where
\begin{align*}
{T}^{abcd} &\stackrel g= t^{abcd} \\
{T}^{abc-} &\stackrel g=- \frac{\nabla_d t^{dabc}}{d+w-1} \\
{T}^{a-b-} &\stackrel g= \frac{\nabla_a \nabla_c t^{abcd} + (d+w-1) P_{ac} t^{abcd}}{(d+w-1)(d+w-2)}\, ,
\end{align*}
where $T$ also has the algebraic symmetries of the Riemann tensor and 
$$
D_A T^{ABCD} = X_A T^{ABCD} = 0 = h_{AC} T^{ABCD}\, .
$$
\end{enumerate}
\end{lemma}

\begin{proof}
The proofs of the above three results can either be given in a much more general setting (see~\cite{GOadv}), or by explicit computation whose intricacy increases in concordance with tensor rank. Here we give the lowest rank example.

The ``bottom slot'' $V^-$ of $q(v_a) := V^A\stackrel g=(v^+,v^a,v^-)$ can be computed by writing out  the constraint~$D_A V^A \stackrel g= 0$ for some $g\in \cc\hh$:
\begin{align*}
0 &= (d+2w-2)(w v^- + \nabla \cdot v + Jv^+ + dv^-) - (\Delta + (w-1)J)v^+ + 2 \nabla \cdot v + dv^- \\
&= (d+2w)(\nabla \cdot v + (d+w-1)v^-)\, ,
\end{align*}
where the second equality comes from the requirement that $X_A V^A = 0$. When $w\neq -\frac d2$, this  yields the quoted result. If $w=-\frac d2$, we need to verify that the result for $V^A$ given for a choice of~$g\in \cc$ defines a section of $\Gamma(\ct M[w])$. This is easily established by transforming the quoted result to a conformally related metric. 
\end{proof}

Note that the insertion operator $q$ is a right-inverse of the extraction map  $$q^*\hspace{-0.5mm}\circ q = \text{Id}\,,$$ but $q$ is \textit{not} a left-inverse of the extraction map, $$q \circ q^* \neq \text{Id}\,.$$ The above operator, acting on rank-$2$ tractors, is computed in the following lemma for use later. 

\begin{lemma} \label{D-free-tractor} 
Let the tractor $T \in \Gamma(\odot^2_\circ \ct M[w])$, where $w \neq 1 - \frac{d}{2},-\frac d2,-d,1-d$, obey  $$X_A T^{AB} = 0\:
\mbox{ and }\: q^*(T)\in \Gamma(\odot_\circ^2 T^*M[w+2])\, .$$
Then 
$$(q \circ q^*)(T)=
\tilde T\, ,
$$
where
$$
\tilde T_{AB}:=
 T_{AB} - 
\tfrac{2}{(d+w)(d+2w)} X_{(A} D^C T_{B)C} + \tfrac{1}{(d+w)(d+w-1)(d+2w)} X_A X_B D^C \hd^D T_{CD}\, .$$
\end{lemma}
\begin{proof}
We will establish that there is a unique $\tilde T$ that
satisfies
$$\hd^A \tilde{T}_{AB} = X^A \tilde{T}_{AB} = 0 = h^{AB} \tilde{T}_{AB}\, ,$$
and obeys $q^*(\tilde T) = q^*(T)$ whenever $q^*(T)\in \Gamma(\odot_\circ^2 T^*M[w+2])$. This ensures that $(q\circ q^*)(T)=\tilde T$.
For that, we use the operator version of Proposition~\ref{leib-failure}, valid acting on tractors of weight $w\neq 1-d/2,-d/2$:
\begin{equation}
\label{usemeoft}
\hd^A \circ X^B = X^B \hd^A+  h^{AB} - \tfrac{2}{d+2w} X^A \hd^B\,.
\end{equation}

We first verify that $X \csdot \tilde{T} = 0$. Because $X^2 = 0 = X^A T_{AB}$, we simply need to check that~$X^A \hd^B T_{AB}$ vanishes. Applying Equation~\nn{usemeoft} we have that
\begin{align*}
X^A \hd^B T_{AB} &= \left(\hd^B X^A - h^{AB} + \tfrac{2}{d+2w} X^B \hd^A \right)T_{AB} \\
&= \tfrac{2}{d+2w} X^B \hd^A T_{AB}\,,
\end{align*}
where we have used that $h^{AB} T_{AB} = 0 = X^A T_{AB}$. Because $T$ is symmetric, we are left with the identity
$$\tfrac{d+2w-2}{d+2w} X^A \hd^B T_{AB} = 0\,.$$
Thus, thanks to the weight assumptions, it follows that $X^A \hd^B T_{AB} = 0$, and hence $X^A \tilde{T}_{AB} = 0$. Similarly, we have that $h^{AB} \tilde{T}_{AB} = 0$.

Finally, we check that $\hd^A \tilde{T}_{AB} = 0.$ We do this in stages. First, we evaluate $\hd^A (X_A \hd^C T_{BC})$:
\begin{align*}
\hd^A( X_A \hd^C T_{BC}) &= \left[  (w-1) + d+2 - \tfrac{2(w-1)}{d+2(w-1)} \right] \hd^C T_{BC} \\
&= \tfrac{(d+w-1)(d+2w)}{d+2w-2} \hd^C T_{BC}\,.
\end{align*}
Next, we evaluate the term $\hd^A (X_B \hd^C T_{AC})$:
\begin{align*}
\hd^A (X_B \hd^C T_{AC}) &= X_B \hd^A \hd^C T_{AC} + \hd^C T_{BC} - \tfrac{2}{d+2(w-1)} X^A \hd_B \hd^C T_{AC} \\
&= X_B \hd^A \hd^C T_{AC}  + \hd^C T_{BC} 
\\&\hspace{1.3cm}
- \tfrac{2}{d+2(w-1)} \left[ \hd_B (X^A \hd^C T_{AC}) - h^A_B \hd^C T_{AC} + \tfrac{2}{d+2(w-1)} X_B \hd^A \hd^C T_{AC} \right] \\[1mm]
&= \tfrac{(d+2w)(d+2w-4)}{(d+2w-2)^2} X_B \hd^A \hd^C T_{AC} + \tfrac{d+2w}{d+2w-2} \hd^C T_{BC}\,.
\end{align*}
Last, we evaluate the term $\hd^A (X_A X_B \hd^C \hd^D T_{CD})$:
\begin{align*}
\hd^A (X_A X_B \hd^C \hd^D T_{CD}) =  \tfrac{(d+w-1)(d+2w)}{d+2w-2} X_B \hd^C \hd^D T_{CD}\,.
\end{align*}
Combining these terms, we find that $\hd^A \tilde{T}_{AB} = 0$, thus completing the proof.
\end{proof}

\medskip

\begin{remark}
Note that if a weight $w\neq 1-\frac d2,-\frac d2$ tractor $\tilde T^{AB\cdots}$ obeys $$X_A \tilde T^{AB\cdots}=0=\hd_A \tilde T^{AB\cdots}\, ,$$ then it follows directly from Equation~\nn{usemeoft} that
$$
X_A \hd^C \tilde T^{AB\cdots} = - \tilde T^{CB\cdots}\, .
$$
\end{remark}

\color{black}

\noindent
Often it is useful to change the projecting part of tractor; the operator in  the following lemma is an instance of this.
\begin{lemma}\label{Pantheon}
Let $V^A \in \Gamma(\ct M[w])$ and $T \in \Gamma(\odot^2_{\circ} M[w])$. Then if $w\neq -1,-1-\frac d2$,
$$r(V^A) := V^A - \frac{1}{w+1}\hh  \hd^A \big(X_B V^B\big)
$$
obeys
$$
X_A \, r(V^A)=0\, ,
$$
while if $w\neq 0,-1,-\frac d2,-1-\frac d2,
-2-\frac d2$,
\begin{multline*} 
r(T^{AB \circ}):=
T^{(AB)\circ}-
\tfrac{2}w \hd^{(A} \big(X_C T^{|C|B)\circ }\big)
+\tfrac1{w(w+1)} \hd^{(A} \hd^{B)\circ} \big(X_C X_D T^{CD}\big)
\\
-\tfrac{8}{wd(d+2w+2)}
X^{(A} \hd^{B)\circ} \big(\hd_C (X_D T^{CD})\big)\, ,
\end{multline*} 
obeys
$$
X_A \, r(T^{AB}) = 0=h_{AB}\,  r(T^{AB})\, .
$$
\end{lemma}

\begin{proof}
The proof is an elementary application of the  identity
$$
X_A \hd^A T = w T\, ,
$$
valid for any weight $w\neq 1-\frac d2$ tractor $T$, the fact that $X$ and $D$ are null, and Equation~\nn{usemeoft}.
\end{proof}

\begin{remark}
The map $r$ of the above lemma can be generalized, modulo distinguished weights, to tractors $T$ of arbitrary tensor type so that $r(T)$ is both tractor trace-free and in the kernel of contraction by the canonical tractor $X$.
\end{remark}

\medskip

The algebra of multiple Thomas-$D$ operators, and in particular commutators of these, will be of particular importance. This leads to a 
tractor tensor that generalizes the Weyl tensor.

\subsection{The $W$-Tractor}

To begin with, we consider the curvature of the tractor connection.  The commutator of tractor connections 
gives the {\it tractor curvature} $\Omega_{ab}{}^{AB}\in \Gamma(\wedge^2 T^*M\otimes \wedge^2 \ct M)$ which acts on a standard tractor~$V^A$ according to 
$$
\left[ \nabla_a, \nabla_b \right] V^A
={{\Omega_{ab}}^{A}}_{B} V^B =:
\Omega_{ab}{}^\hash V^A\in \Gamma(\wedge^2 T^*M\otimes \ct M)\, .
$$
The operator $\Omega_{ab}{}^\hash$
extends by linearity to act on general tractor tensors.
In general, if $\Omega^{AB}$ is any skew-symmetric tractor, then $\Omega^\hash$ denotes the operator on tractors $V^{AB\cdots}$ defined by~$\Omega^\hash V^{AB\cdots}:=\Omega^A{}_C V^{CB\cdots}+\Omega^B{}_C V^{AC\cdots}+\cdots$. 
For a choice of metric  $g \in \mathbf{c}$, the tractor curvature is given in an obvious matrix notation by
\begin{align*}
{{\Omega_{ab}}^C}_D = \begin{pmatrix}
    0       & 0 & 0  \\
    {C_{ab}}^c       & {{W_{ab}}^c}_d & 0  \\
    0       &  -C_{abd}  & 0
\end{pmatrix}\, .
\end{align*} 
The projecting part of the tractor curvature  is the Weyl tensor, which is indeed conformally invariant. In three dimensions, the Weyl tensor vanishes, and the Cotton  tensor is then the projecting part and hence invariant in this dimension.

\newcommand{\Wt}{{\hh\widetilde{\!W\!}\hh\hh}{}}

At this juncture, to build its tractor analog, we could either  insert the Weyl tensor into  a tractor or instead view the Thomas-$D$ operator as the tractor  analog of the gradient operator and  study commutators of Thomas-$D$ operators. Both routes will lead us to  the so-called~$W$-tractor. 
Let us begin with the first route.
In dimension $d>4$,
using Proposition~\ref{insertions}, we  can define the weight~$-2$, 
 rank~4, 
 {\it $W$-tractor} 
  by 
 $$
{W}^{ABCD}:=q(W_{abcd})\, ,
 $$
 where~$W_{abcd}$ is the Weyl tensor. By construction, $W$ has the symmetries of the Weyl tensor 
 $$
 {W}^{[ABC]D}={W}^{ABCD}+{W}^{BACD}={W}^{ABCD}-{W}^{CDAB}=0={W}^A{}_{A}{}^{CD}\, .
 $$
We have used the same letter $W$ for both the Weyl curvature and $W$-tractor because the projecting part of the latter, $W^{abcd}$ is---by construction---the Weyl tensor. 
The remaining tensor
   content of the $W$-tractor  is given below.
\begin{lemma} \label{W-tractor}
Let $d>4$. Then
in any scale $g\in\cc$, the ${W}$-tractor is given by 
\begin{align}
{W}^{abc-} &= 
C^{abc}
 \, ,\nonumber\\\label{Warsaw}
{W}^{a-b-} &= \frac{ \Delta P^{ab} - \nabla_c \nabla^b P^{ca} + P_{cd} W^{adbc}}{d-4}\, .
\end{align}
All other components are either zero or determined by the symmetry of the $W$-trac\-tor.
\end{lemma}
\begin{proof}
This lemma follows from Proposition~\ref{insertions} 
and the curvature identities
 $$C^{abc}=\frac{\nabla_d W^{dcab}}{d-3}\, ,\qquad
 \nabla_a \nabla_c W^{abcd}=(d-3) \big(\Delta P^{bd} -\nabla_c \nabla^b P^{cd}\big)\, .
 $$
\end{proof}

\begin{remark}
In $d=4$ dimensions, the ${W}$-tractor is not well-defined. Instead we may consider the equivalence class of tractors defined by the relation $\Wt^{ABCD}\sim \Wt^{ABCD}+ X^{[A} V^{B][D}X^{C]}$ for~$V^{BD}$ any rank 2, symmetric, trace-free, weight $-4$ tractor. In any choice of scale $g\in \cc$
there is always a representative 
$\Wt^{abcd} = W^{abcd}$,  
$\Wt^{abc-} = C^{abc} $
and $\Wt^{a-b-} = 0$.

Note that the residue of the $d-4$ pole in Equation~\nn{Warsaw} continued to four dimensions is precisely the four dimensional  Bach tensor
\begin{equation}\label{notjohann}
B^{ab} = \Delta P^{ab} - \nabla_c \nabla^b P^{ca} + P_{cd} W^{adbc}\, .
\end{equation}
For that reason, the $W$-tractor  is often defined with an overall factor of $d-4$ so that it is well-defined in four dimensions; this has the disadvantage 
that in higher dimensions, the projecting part only equals the Weyl tensor up to an overall non-zero multiplicative factor.
Also note that we use the above tensor expression to define a Bach tensor $B^{ab}\stackrel g= (d-4) W^{a-b-}$ in dimensions higher than four.
\end{remark}

Turning to the other route, the commutator of Thomas-$D$ operators is given below:
\begin{proposition} \label{DD-comm}
In dimension $d\neq 4$, the commutator of Thomas-$D$ operators obeys
\begin{align*}
\left[D_A, D_B \right] = (d+2w-4)(d+2w-2)W_{AB}{}^\hash +4 X_{[A} W_{B]C}{}^\hash\circ  D^C.
\end{align*}
In $d=4$, the commutator of two Thomas-$D$ operators is
\begin{align*}
\left[D_A, D_B \right] = 4w(w+1) \hh
\Wt_{AB}{}^\sharp +4 X_{[A} \Wt_{B]C}{}^\sharp \circ D^C,
\end{align*}
and the sum of equivalence classes on the right hand side is a tractor-valued operator.
\end{proposition}

\begin{proof}
This result was first given in~\cite{GOmin} 
using an 
 ambient construction.
 Alternatively one can---at the cost of quite some computation---verify the first displayed equation using  Equation~\nn{vaccinateThomas} and 
 the injectors introduced above. 
 The same applies for the second display, but we must first verify that its right hand side is defined as a tractor operator independently of the choice of equivalence class representative chosen for $\Wt$. It is sufficient to check that the right-hand side of the second display, acting on a tractor $V^C \in \Gamma(\ct M[w])$ and contracted with $Y^A Y_C$, does not depend on the choice of metric representative.
\begin{align*}
Y_A Y_C \Big[4w(w\!  &+\hh 1)\hh \Wt^{ABCE} V_E  +\,  4 X^{[A} \Wt^{B]ECF} D_E V_F \Big] \\
&=4w(w+1) \Wt^{-B-E} V_E + 2 \Wt^{BE-F} D_E V_F - 2X^B \Wt^{-E-F} D_E V_F \\
&= 4w(w+1) \Wt^{-B-E} V_E + 2 w (d+2w-2)\Wt^{B--F} V_F \\
&\qquad\quad+ 2 (d+2w-2) \Wt^{Be-F} \nabla_e V_F - 2X^B \Wt^{-E-F} D_E V_F \\
&=  4(w+1) \Wt^{Be-F} \nabla_e V_F  - 2X^B \Wt^{-E-F} D_E V_F \\
&= 4(w+1) \left[X^B \Wt^{-e-F} + Z^B_b \Wt^{be-F} \right] \nabla_e V_F - 2(d+2w-2)X^B \Wt^{-e-F} \nabla_e V^F \\
&=4(w+1) Z^B_b \Wt^{be-F} \nabla_e V_F.
\end{align*}
In the above computation, we used that $ \Wt^{+ABC}=0$ and that the Thomas-$D$ can be expanded in terms of the injectors according to~\nn{vaccinateThomas}. 
Importantly, in the second and fourth equalities, we rely on cancellations peculiar to $d=4$.
Because $\Wt^{be-F} \sim \Wt^{be-F} + X^{[b}V^{e][F} X^{-]} = \Wt^{be-F}$, it follows that the above display defines  a tractor.
\end{proof}

\medskip

\renewcommand{\II}{\Two}
\renewcommand{\IIo}{\TwoN}

\section{Conformally Embedded Hypersurfaces} \label{CEH}

We now develop the invariant theory for a  hypersurface $\Sigma$ smoothly embedded in 
a conformal manifold~$(M,\cc)$.
Here we mean that $\Sigma$ is a smoothly embedded codimension-$1$ submanifold of~$(M,\cc)$, and denote this by $\Sigma\hookrightarrow (M,\cc )$.
It is simplifying, although not strictly necessary for the study of local invariant theory, to assume that $\Sigma$ is closed and orientable, and that its embedding
in $M$ is separating. In this case $M=M^+\sqcup\Sigma\sqcup M^-$. 
Given a metric $g\in\cc$, we denote the unit conormal by $\hat n$, and take as convention that it points in the direction of $M^+$. 
The first fundamental form ${\rm I}_{ab}=(g_{ab}-\hat n_a \hat n_b)|_{\Sigma}$
equals the induced metric $\bar g_{ab}$ of $\Sigma$.
The second fundamental form is then 
\begin{equation}\label{IIis}
\II_{ab} = {\rm I}_a^c\nabla_c \hat n_b^{\rm e} \big |_\Sigma=\nabla^\top_a \hat n_b^{\rm e}\big|_\Sigma \in \Gamma(\odot^2 T^*\Sigma) \, ,
\end{equation}
where 
$\hat n^{\rm e}$ on the right hand side of the above display is any smooth extension of $\hat n$ to $M$. 
Note that the image of the projection of the tangent bundle $TM$ along the hypersurface by the endomorphism $I^a_b=g^{ac}I_{bc}$ is isomorphic to $T\Sigma$. We therefore identify these spaces and use the same abstract index notation for hypersurface tensors as for ambient ones.
We use a $\top$ notation 
for projection to $T\Sigma$, and also for more general tensors. For example if $\omega$ is a one-form in $\Gamma(T^*M)|_\Sigma$, then $\top (\omega_a)\equiv\omega^\top_a:=I_a^b \hh \omega_b\in \Gamma(T^*\Sigma)$. Contraction by the unit conormal, followed by projection is denoted, for example  for a two-form, by $\omega_{\hat n b}^\top:= (\hat n^a  \omega_{ac})I^c_b$. 
When $\omega\in \Gamma(T^*M)$ we use $\omega^\top$ to denote~$(\omega|_\Sigma)^\top$.
A ring over the projection symbol~$\mathring\top$ denotes an additional projection to the (hypersurface) trace-free part.
In these terms, the Gau\ss\ formula then reads
\begin{equation}\label{GF}
\bar \nabla_a \bar v^b=
\nabla_a^\top v^b\big|_\Sigma
+ \hat n^b \II_{ac} \bar v^c
  \, ,
\end{equation}
 where in the first term on the right hand side, $v^a$ denotes any smooth extension of a hypersurface vector  $\bar v^a\in \Gamma(T\Sigma)$ to~$M$. Also, $\bar \nabla$ is the Levi-Civita connection of the induced metric $\bar g$. 
 Quite generally, we use a bar notation on objects to indicate when they belong to the hypersurface $\Sigma$.
 For example, curvature quantities 
 intrinsic to the hypersurface will be denoted by bars as will be the  hypersurface dimension, so $d-1=\bar d$. 
 The Gau\ss\ equation is then given by
 \begin{align}\label{GE}
R^\top_{abcd} = \overline{R}_{abcd} - \Two_{ac} \Two_{bd} + \Two_{ad} \Two_{bc}\, .
\end{align}
We also record the trace-free and traced Codazzi--Mainardi equations:
\begin{eqnarray}
\bar \nabla_{[a} \IIo_{b]c}-
\tfrac1{d-2} \hh\bar g_{c[a} \bar \nabla \csdot \IIo_{b]}
=
\tfrac12 W^\top_{abc\hat n }\, ,\label{trfreecod}\\
\bar \nabla_a H-\tfrac1{d-2}
\bar \nabla\csdot \IIo_{a}
=-  \Rho^\top_{a\hat n}\, .\qquad
\label{trcod}\nonumber
\end{eqnarray}

\subsection{Hypersurface Tractor Calculus}\label{HTC}

Tractor calculus is a  key technical tool for studying conformal embeddings. A basic observation~\cite{BEG}  is that the unit conormal $\hat n^g$ and mean curvature~$H^g$ of an embedded hypersurface depend on the conformal representative $g\in\cc$ as
$$
\big(\hat n^{\Omega^2 g}, H^{\Omega^2 g}\big)=\big(\Omega \hh \hat n^g , \Omega^{-1} (H^g + \hat n^g \cdot \ext\log \Omega)\big)\, ,
$$
where $0<\Omega\in C^\infty M$ and its exterior derivative appearing above are restricted to $\Sigma$. 
Henceforth we reuse the notation $\hat n^a$ and $H$ for the corresponding  density-valued objects.
So, given a metric $g\in \cc$, the triple 
$$
N^A\stackrel g=\begin{pmatrix}
0\\  \, \hat n^a
\\ 
-H
\end{pmatrix}\stackrel\Sigma= \hat n^a Z^A_a
- H Y^A
$$
labels a section of the tractor bundle $\ct M|_\Sigma$. The section $N^A$ is termed the {\it normal tractor} and obeys
$$
h(N,N)=1\, .
$$ 
Note that the projecting part of $N$ is the unit normal, which  indeed defines a vector-valued weight $w=-1$ conformal density.
The normal tractor mimics the {\it r\^ole} of the unit conormal in Riemannian hypersurface theory. In particular, 
projection of the tractor bundle along~$\Sigma$ by~${\rm I}^A_B:= \delta^A_B-N^A N_B$ 
  gives a bundle  isomorphic to the   tractor bundle~$\ct \Sigma$ of $(\Sigma, \bar \cc)$, where $\bar \cc$ is the induced conformal class of metrics~\cite{Goal}.
The tractor ${\rm I}_{AB}:= h_{AB}|_\Sigma-N_A N_B$ is the
  {\it tractor first fundamental form}.
Hence, we shall identify these spaces and  use the same abstract index notation for ambient and hypersurface tractors.
In particular the hypersurface tractor metric is given by~$\bar h_{AB}={\rm I}_{AB}$. 
Moreover, together the normal tractor and tractor metric give an isomorphism
$$
\ct M|_\Sigma \cong {\mathcal N}\Sigma \oplus \ct \Sigma\, ,
$$
where  ${\mathcal N}\Sigma$ is the tractor normal bundle (defined by orthogonal decomposition completely analogously to the normal bundle for Riemannian hypersurface embeddings).
Note that for the canonical tractor $X^A$, we have ${\rm I}^A_B X^B=X^A$, whose image under the bundle isomorphism is the hypersurface canonical tractor $\bar X^A$; thus we often drop the bar and denote this by $X^A$.

Recall that the trace-free combination 
$$
\TwoN:=\Two-H\, \bar g\, ,
$$
is conformally invariant and defines a section of $\Gamma(\odot^2_\circ T^*\Sigma[1])$. In dimensions $d>3$, it can be inserted into a symmetric, trace-free hypersurface tractor 
$$
L^{AB}:= \bar q(\TwoN_{ab})\in \Gamma(\odot^2_\circ\ct \Sigma[-1])
$$
termed the {\it tractor second fundamental form}~\cite{Grant,Stafford}. In the above, we have applied the insertion map of Lemma~\ref{insertions}  on the conformal manifold~$(\Sigma,\bar \cc)$. Note that in a choice of scale $g$, 
$$
L^{AB}\stackrel g=
\begin{pmatrix}
0&0&0\\
0&\IIo^{ab}&
-\tfrac1{d-2} \bar \nabla\csdot \IIo^a\\[1mm]
0& -\tfrac1{d-2} \bar \nabla\csdot \IIo^b &
\frac{\bar \nabla\cdot\bar\nabla \cdot \IIo + (d-2)\bar \Rho^{ab} \IIo_{ab}}
{(d-2)(d-3)}
\end{pmatrix}.
$$
The hypersurface conformal curvature quantity $L^2 = \TwoN^2 =: K \in \Gamma(\ce \Sigma[-2])$ was termed the 
rigidity density
because it was employed to describe rigid strings in~\cite{Polyakov}. 

\medskip

\begin{remark}
In the theory of surfaces in ${\mathbb R}^3$, the square of the second fundamental form  is sometimes termed the third fundamental form $\III$. If $\Sigma\hookrightarrow {\mathbb R^3}$ is a surface expressed as the level set of a unit defining function $s:{\mathbb R^3}\to {\mathbb R}$, then the second fundamental form $\II$ is the Hessian of~$s$, so that $\II=\nabla n$ where $n=\nabla s$. The third fundamental form then obeys
$$
\III_{ab}= -\nabla_n \nabla_a n_b\big|_\Sigma = \II_{ab}^2\, .
$$
Indeed  one might even  define higher fundamental forms by taking successive normal derivatives, because for $k\in {\mathbb Z}_{\geq 0}$,
$$
\II^{k+1}_{ab}=\tfrac{(-1)^{k}}{ k!} \nabla_n^{k} \nabla_a n_b\big|_\Sigma\, .
$$ 
In Section~\ref{herecomethI} we give a  definition of the {\it transverse order} of an hypersurface invariant by viewing formul\ae\ such as that above as functionals of a general curved ambient metric and then counting normal/transverse derivatives on this metric. The third fundamental form for a hypersurface embedded in a Riemannian manifold
defined by $\III_{ab}:= -\nabla_n \nabla_a n_b\big|_\Sigma=\II_{ab}^2-R_{\hat n ab \hat n}$
 then has transverse order two. A main aim of this article is to develop a theory of conformally invariant higher fundamental forms. 
\end{remark}

\medskip
In dimensions $d>3$, manipulating the trace of the Gau\ss\ equation~\nn{GE} produces a conformally invariant third fundamental form, termed the Fialkow tensor $F \in\Gamma(\odot^2T^* \Sigma[0])$ and defined by~\cite{Fialkow} (see also~\cite{Vyatkin}),
\begin{equation}\label{Aaron}
\Big(
\TwoN^2_{ab} - \tfrac{1}{2(d-2)} 
 \TwoN^2
\bar g_{ab} - W_{\hat n ab \hat n} \Big)
=
(d-3)\big(
\Rho^\top_{ab} - \bar \Rho_{ab} + H \TwoN_{ab} + \tfrac{1}{2} H^2 \bar g_{ab}\big) 
=:(d-3) F_{ab}\, .
\end{equation}
We shall call the above relation the {\it Fialkow--Gau\ss\ equation}.

The trace-free Fialkow tensor, denoted $\Fo$, is also of special interest. In particular, in dimensions~$d > 4$, it too can be inserted (using Lemma~\ref{insertions}) into a symmetric, trace-free tractor
$$F^{AB} := q(\Fo_{ab}) \in \Gamma(\odot^2_\circ \ct \Sigma[-2])$$
termed the \textit{Fialkow tractor}. In a choice of scale $g$, 
$$
F^{AB}\stackrel g=
\begin{pmatrix}
0&0&0\\
0&\Fo^{ab}&
-\tfrac1{d-3} \bar \nabla\csdot \Fo^a\\[1mm]
0& -\tfrac1{d-3} \bar \nabla\csdot \Fo^b &
\frac{\bar \nabla\cdot\bar\nabla \cdot \Fo + (d-3)\bar \Rho^{ab} \Fo_{ab}}
{(d-3)(d-4)}
\end{pmatrix}.
$$
It is also useful to define a weight $-1$, rank three tractor that combines the Fialkow and trace-free second fundamental forms, 
\begin{align}\label{Gamma}
\Gamma_{ABC} := 2 N_{[C} L_{B]A} + 2 X_{[C} F_{B]A} + \tfrac{K}{\bar{d}(\bar{d}-1)} X_{[C} \bar h_{B]A}.
\end{align}

From its Definition~\nn{Aaron}, we see that Fialkow tensor measures the difference between the bulk and hypersurface tractor connections through their 
 respective Schouten tensors. 
This leads to an analog of Equation~\nn{GF} (see~\cite{Will2}), that we term the Fialkow--Gau\ss\ formula
\begin{align} \label{FG-simple}
 \bar{\nabla}_a  \bar{V}^B
=\nabla^\top_a {V}^B\big|_\Sigma  -  {{\Gamma_a}^B}_C \bar{V}^C \, .\end{align}
Here $\bar{V} \in \Gamma(\ct \Sigma)$ is any standard hypersurface tractor and we have used the isomorphism discussed above between the hypersurface tractor bundle  and the projection of the bulk tractor bundle  restricted to~$\Sigma$, to extend this to $V\in \Gamma(\ct M)$. The displayed formula does not depend on the choice of such an extension.

Formul\ae\ like the right hand side of Equation~\nn{FG-simple} 
that are expressed in terms of an operator acting on an extension of an object defined only along the hypersurface play an important {\it r\^ole} in the study of extrinsic geometry.
 In general, we call an operator  ${\rm O}$ acting on sections $a$ of a 
 vector bundle over the bulk manifold
  \textit{tangential}
  if $(Oa)|_\Sigma$ only depends on the restriction~$\bar a = a|_{\Sigma}$
  ({\it i.e.}, it is independent of  the choice of extension $a$ of $\bar a$).
  This defines an operator ${\rm O}_{\sss\Sigma}$
 along~$\Sigma$ by ${\rm O}_{\sss\Sigma} \, \bar a := ({\rm O} a) |_\Sigma$. 
 Note that the right hand side 
is 
 a {holographic formula} for the quantity ${\rm O}_{\sss\Sigma} \hh \bar a$; this  notion is developed in detail in Section~\ref{herecomethI}.
  The operator $\nabla^\top$ is tangential. Our next task is to develop a tangential analog of the Thomas-$D$ operator in order to relate the bulk Thomas-$D$ operator to its hypersurface counterpart.

\begin{proposition} \label{Dt-complex}
Let  $w + \tfrac{d}{2} \neq 1, \tfrac{3}{2},2$
and $N^{\rm e}$ be any extension of the normal tractor. Then, the operator $$\hd^T : \Gamma(\ct^\Phi M[w]) \to \Gamma(\ct M \otimes \ct^\Phi M[w-1])\, ,
$$  given by 
$$\hd^T_A := \hd_A - N_A^{\rm e} N^B_{\rm e} \hd_B + \frac{X_A}{d+2w-3} \left(N^B_{\rm e} N^C_{\rm e} \hd_B \hd_C + \frac{wK}{d-2} \right)\, ,$$
is tangential. Moreover, if $\operatorname{O}^A$ is any operator acting on tractors  of weight $\frac{1-d}2$ that obeys
 $$\operatorname{O}^A\circ X_A=0\, ,$$ then the operator
$$
\operatorname{O}\, \csdot\,\hh  \tilde D^T:=
\operatorname{O}^A\circ \hh \big(\hd_A - N_A^{\rm e} N^B_{\rm e} \hd_B\big)
$$
is also tangential.
\end{proposition}
\noindent
The tangential Thomas-$D$ operator was first introduced in~\cite{Forms}; the proof of the above result is 
 given in~\cite{Will2}, 
 and is simplified using holographic ideas. Note that, unlike the Levi-Civita connection, which becomes tangential upon projection by the first fundamental form, projection by the tractor first fundamental form alone does not render the Thomas-$D$ tangential. Also note that the last term proportional to $K$ in the definition of $\hd^T$ is separately tangential, but has been added because the operator in parentheses will play a special {\it r\^ole}.

 An immediate application of the  {\it tangential Thomas-$D$ operator} is the tractor analog of the defining equation for the  second fundamental form~\nn{IIis}:
\begin{lemma} \label{thisismylabel}
Let $d>3$ and $N^{\rm e}$ be any extension of the normal tractor. Then the tractor second fundamental form obeys 
$$
L_{AB}=\hd^T_{(A} N_{B)}^{\rm e}\big|_\Sigma -\tfrac1{d-3} \, \big(X_{(A} N_{B)} K 
+X_A X_B M\big)\, ,
$$
with $K=L^2=(\hd N^{\rm e})^2\big|_\Sigma$ 
and $M=L_{AB} F^{AB}=F^{AB}(\hd_A ^{\phantom{e}}N^{\rm e}_B) \big|_\Sigma$.
\end{lemma}
\noindent
The proof of this lemma is given in Subsection~\ref{can-ext-IIo} and relies on holographic ideas.

As promised in the introduction, the following subsections provide proofs  of Theorem~\ref{GTF} and Corollary~\ref{GTE}.

 \bigskip

\subsection{Proof of Theorem~\ref{GTF}} \label{GTform-proof}
We prove the Gauss-Thomas formula~\ref{GTF} acting on a tractor vector of arbitrary weight and then generalize. That is, we look to prove the following lemma.

\begin{lemma}\label{Dt-Dbar}
Let $V \in \Gamma(\ct \Sigma[w])$. Then, the bulk tangential and hypersurface Thomas-$D$ operators obey
\begin{multline*}
\big(\hd^T_A\big)_\Sigma V^B =  \hdb_A V^B +\Gamma_A {}^\sharp V^B \\- \frac{X_A}{\bar{d} + 2w-2} \bigg{\{} 
2\Gamma_C{}^\sharp \circ \hdb^C V^B + \Gamma^C{}^\sharp \circ \Gamma_C{}^\sharp V^B
\\+ \frac{1}{\bar{d}(\bar{d}-1)} \left[\left(\hdb K \right)  \wedge X \right]^\sharp V^B - \frac{(3\bar{d}+2)wK}{2\bar{d}(\bar d-1)} V^B \bigg{\}}\, .
\end{multline*}
\end{lemma}

\begin{proof}
We rely heavily on~\cite[Lemma 4.9]{Will2}, which states that
\begin{align}\label{Dt-matrix}
\left(\hd^T{}^A\right)_{\Sigma} = \begin{pmatrix}
w \\[1mm]
\nabla^\top{}^a\\[1mm]
-\frac{\Delta^\top + w \bar{J}}{\bar d+2w-2} + \frac{wK}{2(\bar d -1)(\bar d+2w-2)} 
\end{pmatrix}
\end{align}
We check the lemma by contracting both sides with the injectors and checking that it holds for all three. First note that $w = \bar{w}$ and $X_A \Gamma^A{}_{BC} = 0$. Thus, the lemma holds upon contraction with $X^A$.

Next, we check the lemma upon contraction with $\bar{Z}^A_a$. Note that, according to Equation~\nn{Dt-matrix}, $\bar{Z}^A_a \hd^T{}_A = \left(\nabla^\top_a\right)_{\Sigma}$. Because $X_A \bar{Z}^A_a = 0$, we have that
\begin{align*}
\left(\nabla^\top_a \right)_{\Sigma} V^B = \bar{\nabla}_a V^B + {{\Gamma_a}^B}_C \bar{V}^C \,,
\end{align*}
which is the Fialkow-Gauss equation. Thus, the lemma holds upon contraction with $\bar{Z}^A_a$.

Finally, we  check that the identity holds upon contraction with $\bar{Y}^A$. From Equation.~\nn{Dt-matrix}, we have that 
\begin{align*}
\bar{Y}^A \left[\hd^T_A - \hdb_A \right] =& \left(-\frac{\Delta^\top + w \bar{J}}{\bar d+2w-2} + \frac{wK}{2(\bar d -1)(\bar d+2w-2)} \right) +\frac{1}{\bar d +2w-2} \left(\bar \Delta +w \bar J \right) \\
=& -\frac{1}{\bar d +2w-2} \left(\Delta^\top - \bar \Delta - \frac{wK}{2(\bar d-1)} \right).
\end{align*}
We explicitly compute the tractor Laplacian difference $\Delta^\top - \bar{\Delta}$.

From the definition of the tractor Laplacian and defining $\Gamma^{ABC}_\perp := \Gamma^{ABC} - 2 N^{[C} L^{B]A} \in \Gamma(\ct^3 \Sigma[-1])$, we have
\begin{align*}
\Delta^\top V^B & = \nabla_a^\top \left(\bar \nabla^a V^B + {\Gamma^{aBC}} V_C \right) \\
&= \nabla^\top_a \left(\bar \nabla^ a V^B + \Gamma^{aBC}_\perp V_C + 2 N^{[C} L^{B]a} V_C \right) \\
& = \bar \Delta V^B + {\Gamma^{aB}}_C \bar \nabla_a V^C 
+ \bar \nabla_a (\Gamma^{aBC}_\perp V_C) + {{\Gamma_a}^B}_E \Gamma^{aEC}_\perp V_C+ \nabla^\top_a \left( 2 N^{[C} L^{B]a} V_C \right)\\
& = \bar \Delta V^B + {\Gamma^{aB}}_C \bar \nabla_a  V^C + \left( \bar \nabla_a \Gamma^{aBC}_\perp \right)  V_C + \left({\Gamma^a}^{BC}- 2 N^{[C} L^{B]a} \right) \bar \nabla_a V_C \\&\;\;\;\;+ {{\Gamma_a}^B}_E \Gamma^{aEC} V_C -  2 {{\Gamma_a}^B}_EN^{[C} L^{E]a}  V_C + 2\left(\nabla^\top_a (N^{[C} L^{B]a}) \right)  V_C + 2 N^{[C} L^{B]a} \nabla^\top_a  V_C \\
& = \bar \Delta  V^B + 2{\Gamma^{aB}}_C \bar \nabla_a  V^C +  \left( \bar \nabla_a \Gamma^{aBC}_\perp \right) V_C  \\&\;\;\;\;+ {{\Gamma_a}^B}_E \Gamma^{aEC} V_C -  2 {{\Gamma_a}^B}_E  N^{[C} L^{E]a}    V_C + 2\left(\nabla^\top_a (N^{[C} L^{B]a}) \right) V_C+2 N^{[C} L^{B]a} {\Gamma_{aC}}^{E} V_E\, .
\end{align*}
Writing $\Gamma = \Gamma_\perp+ 2NL$
allowed us to use the Fialkow--Gau\ss\ equation~\nn{FG-simple} for the above simplifications.

We now break up this calculation into smaller parts.
\begin{align*}
\bar \nabla_a \Gamma^{aBC}_\perp &= \bar \nabla_a \left(\bar Z^a_A \Gamma^{ABC}_\perp \right) \\
&= \bar Z^a_A \bar\nabla_a \Gamma^{ABC}_\perp + \Gamma^{ABC}_\perp \left(\bar J \bar X_A - \bar d \bar Y_A \right) \\
&= \bar Z^a_A \bar\nabla_a \Gamma^{ABC}_\perp - \bar d \Gamma^{-BC}_\perp.
\end{align*}
Here, the second equality comes from the fact that $\bar \nabla_a \bar Z^b_A = -\bar P^b_a \bar X_A - \bar g^b_a \bar Y_A$ and the last equality holds because $X_A \Gamma^{ABC} = 0$. Similarly,
\begin{align*}
\nabla^\top_a N^{[C} L^{B]a} &=\left(\nabla^\top_a N^{[C} \right) L^{B]a} + N^{[C} \left[\bar \nabla_a \left( L^{B]A} \bar Z^a_A \right) + {{\Gamma_a}^{B]}}_E L^{Ea} \right] \\
&= \left(\nabla^\top_a N^{[C} \right) L^{B]a} + N^{[C} \left[\left(\bar \nabla_a  L^{B]A} \right) \bar Z^a_A - \bar d L^{B]-} + {{\Gamma_a}^{B]}}_E L^{Ea} \right]
\end{align*}

In order to simplify the above two displays, we need results that follow from Equation~\nn{Iconnect}:
\begin{align*}
\bar Z^a_A \bar \nabla_a L^{AC} &= 2 L^{-C}\,, \\
\bar Z^a_A \bar \nabla_a F^{AC} &= 3F^{-C}\,, \\
\nabla^\top_a N^B &= L_a^B.
\end{align*}
Using the above identities, we can write
\begin{align*}
\bar \nabla_a \Gamma^{aBC}_\perp &= 2 \left(\bar \nabla_a X^{[C} \right) F^{B]a} + 6X^{[C} F^{B]-} + \frac{1}{\bar d(\bar d-1)} \left(\nabla_a K \right) X^{[C} \bar h^{B]a} - \bar d \Gamma^{-BC}_\perp \\
&= 4 X^{[C} F^{B]-} + \frac{1}{\bar d (\bar d -1)} \left(X^{[C} \hdb^{B]} K + 2 K X^{[C} \bar h^{B]-} \right)  - \bar d \Gamma^{-BC}_\perp\\
&= (2 - \bar d) \Gamma^{-BC}_\perp + \frac{1}{\bar d(\bar d-1)} X^{[C} \hdb^{B]} K
\end{align*}
and
\begin{align*}
2 \nabla^\top_a N^{[C} L^{B]a} = 2(2-d) N^{[C} L^{B]-} + N^C \Gamma^{ABE} L_{AE} - N^B \Gamma^{ACE} L_{AE}.
\end{align*}

We can now use these formulas to write
\begin{align*}
\Delta^\top \bar V^B &= \bar \Delta V^B + 2{\Gamma^{aBC}} \bar \nabla_a V_C +  (2 - \bar d) \Gamma^{-BC}_\perp V_C + \frac{V_C}{\bar d(\bar d-1)} X^{[C} \hdb^{B]} K + {{\Gamma_a}^B}_E \Gamma^{aEC} V_C \\&\;\;\;\; -  2 {{\Gamma_a}^B}_E  N^{[C} L^{E]a}   V_C +   2 N^{[C} L^{B]a} {\Gamma_{aC}}^{E} V_E \\&\;\;\;\; + \left(2(2-d) N^{[C} L^{B]-} + N^C \Gamma^{ABE} L_{AE} - N^B \Gamma^{ACE} L_{AE} \right) V_C \\
&= \bar \Delta \bar V^B + 2{\Gamma^{aBC}} \bar \nabla_a V_C +  (2 - \bar d) \Gamma^{-BC} V_C + \frac{ V_C}{\bar d(\bar d-1)} X^{[C} \hdb^{B]} K  + {{\Gamma_a}^B}_E \Gamma^{aEC} V_C \\&\;\;\;\; -  2 {{\Gamma_a}^B}_E  N^{[C} L^{E]a}   V_C +   2 N^{[C} L^{B]a} {\Gamma_{aC}}^{E} V_E + \left(N^C \Gamma^{ABE} L_{AE} - N^B \Gamma^{ACE} L_{AE} \right) V_C \\
&= \bar \Delta V^B + 2{\Gamma^{aBC}} \bar \nabla_a V_C +  (2 - \bar d) \Gamma^{-BC} V_C + \frac{ V_C}{\bar d(\bar d-1)} X^{[C} \hdb^{B]} K  + {{\Gamma_a}^B}_E \Gamma^{aEC} V_C \\ 
&= \bar \Delta V^B + 2{\Gamma^{ABC}} \hdb_A V_C - (\bar d+2w-2) \Gamma^{-BC} V_C + \frac{V_C}{\bar d(\bar d-1)} X^{[C} \hdb^{B]} K  + {{\Gamma_A}^B}_E \Gamma^{AEC} V_C \\
&= \bar \Delta V^B + 2 \Gamma^A{}_A{}^C \hdb_C V^B - \frac{(\bar d +1) w K}{\bar d (\bar d - 1)} V^B + 2{\Gamma^{ABC}} \hdb_A V_C - (\bar d+2w-2) \Gamma^{-BC} V_C \\&\;\;\;+ \frac{V_C}{\bar d(\bar d-1)} X^{[C} \hdb^{B]} K  + {{\Gamma_A}^B}_E \Gamma^{AEC} V_C \\
&= \bar \Delta V^B + 2 \Gamma^A{}^\sharp \circ \hdb_A V^B  - (\bar d+2w-2) \Gamma^{-BC} V_C + \frac{V_C}{\bar d(\bar d-1)} X^{[C} \hdb^{B]} K  \\&\;\;\;+ \Gamma^A{}^\sharp \circ \Gamma_A{}^\sharp V^B - \frac{(\bar d +1) w K}{\bar d (\bar d - 1)} V^B\, .
\end{align*}
In the display above, the second equality comes the definition of $\Gamma_\perp$. The third equality is a result of the last four terms canceling, and the fourth equality comes from the fact that $\Gamma \in \ker X_\lrcorner$ (where $X_\lrcorner$ denotes contraction by $X$). The last inequality follows from $\Gamma_A{}^A{}_E \Gamma^{EBC} V_C = 0$.

But,
\begin{align*}
\left(\hd^T_A - \hdb_A \right) V^B = \bar Z_a^A {{\Gamma_A}^B}_C  V^C - \frac{X_A}{\bar d +2w-2} \left(\Delta^\top - \bar \Delta - \frac{wK}{2(\bar d-1)} \right) V^B,
\end{align*}
so the proof is completed by combining terms involving $wK$.
\end{proof}

\begin{remark}
Applying the same techniques in the proof above but accounting for the possible linear actions on multiple indices, the proof of the general Gau\ss--Thomas formula follows easily.
\end{remark}

\noindent

\subsection{Proof of Corollary~\ref{GTE}} \label{GTE-proof}

In this section, we provide a proof of Corollary~\ref{GTE} as well as an additional corollary.

\begin{proof}[Proof of Corollary~\ref{GTE}]
Recall that the Gau\ss\ equation is a corollary of the 
Gau\ss\ formula, in the sense that it is obtained by applying the latter to $[\nabla_a^\top, \nabla_b^\top] v_c$ where $v_c$ is an extension of a hypersurface tangent vector. Similarly, the present proof could be completed by applying the Gau\ss--Thomas formula to  $[\hd^T_A, \hd^T_B] V_C$. But, because $\hd$ is not a derivation, that computation is rather involved. Instead, we approach the proof via equality of all possible contractions (in some scale $g\in \cc$) 
by hypersurface injectors $(X^A,\bar Z^A_a,\bar Y^A)$
on both sides of the lemma's result. 
Note that it is unnecessary to check  contractions with more than one~$\bar{Y}$---this only probes $V_{AB}$. Also, without loss of generality, we may choose $g$ to be a scale in which the mean curvature $H^g$ of the embedding $\Sigma\hookrightarrow (M,g)$ vanishes.

We begin by contracting with a single $X$. For that, we first use  
 Proposition~\ref{leib-failure} and the Fialkow tractor identities 
$\hdb^A F_{AB}=0=X^A F_{AB}$, $0=F_A{}^A$ as well as the ansatz $X^AV_{AB}=X_B V$,
to obtain 
\begin{align*}
X^A T_{ABC}\, =& \tfrac{1}{(d-1)(d-2)} X_{[B} \hdb_{C]} K\, , \\
X^C T_{ABC} \,=& -F_{AB} - \tfrac{K}{(d-1)(d-2)} \bar{h}_{AB} - \tfrac{1}{2(d-1)(d-2)} X_A \hdb_B K \, ,\\
 X^A X_{[B} V_{C]A}=&\: 0\, .
\end{align*}
Now $X^A W_{ABCD}^\top = 0$, so we need to show contraction of the right-hand side of Equation~\nn{Wt-Wb} with $X^A$ vanishes. Clearly $X^A \bar{W}_{ABCD}=0$ and the contraction of $X$ with the second term also vanishes because $X^A L_{AB}  =0$. Using $X^AF_{AB}=0$ along with the identities of the above display, the remaining terms are
\begin{align*}
-2 X_{[C} F_{D]B} + \tfrac{2}{(d-1)(d-2)} X_{[D} \bar{h}_{C]B} K &\\
+2 X_{[C} F_{D]B} - \tfrac{2}{(d-1)(d-2)} X_{[D} \bar{h}_{C]B} K - \tfrac{1}{(d-1)(d-2)} X_B X_{[C} \hdb_{D]} K  & \\
+ \tfrac{1}{(d-1)(d-2)} X_B X_{[C} \hdb_{D]} K & = 0\, .
\end{align*}
Because the $W$-tractor has Weyl curvatures symmetries this establishes consistency of the identity when any index is contracted with a canonical tractor.
\medskip

Next, note that $\bar{Z}^A_a \bar{Z}^B_{b} \bar{Z}^C_c \bar{Z}^D_dW^\top_{ABCD}=W^\top_{abcd}$ and that the
 trace-free Gau\ss\ equation says
\begin{align*} \label{weylt-weylb}
W_{abcd}^\top = \bar{W}_{abcd} - 2 \IIo_{a[c} \IIo_{d]b} - 2 \bar{g}_{a[c} \Fo_{d]b} + 2 \bar{g}_{b[c} \Fo_{d]a} - \tfrac{2}{(d-1)(d-2)} \bar{g}_{a[c} \bar{g}_{d]b} K.
\end{align*}
It is easy to check, using $X_A \bar Z^A_a=0$, that this is the right hand side of
Equation~\nn{Wt-Wb}  when contracted with this combination of injectors.
\medskip

The last case to check is contraction of Equation~\nn{Wt-Wb} by $\bar{Z}^A_a \bar{Z}^B_b \bar{Z}^C_c \bar{Y}^D$.  By directly applying the definitions of $L_{AB}$, $F_{AB}$, $\bar W_{ABCD}$, and the hatted hypersurface Thomas-$D$ operator, after some computation, 
we find for the right-hand side
\begin{align*}
&\bar{Z}^A_a \bar{Z}^B_b \bar{Z}^C_c \bar{Y}^D \Big[\overline{W}_{ABCD} - 2 L_{A[C} L_{D]B} - 2 \overline{h}_{A[C} F_{D]B} + 2 \overline{h}_{B[C} F_{D]A} - \tfrac{2}{(d-1)(d-2)} \overline{h}_{A[C} \overline{h}_{D]B} K \\[1mm]
&\phantom{\bar{Z}^A_a \bar{Z}^B_b \bar{Z}^C_c \bar{Y}^D \Big[\overline{W}_{ABCD}}+ 2 X_{[A} T_{B]CD} +2 X_{[C} T_{D]AB} \Big] \\
=&\ \overline{C}_{abc}  + \tfrac{2}{d-2} \IIo_{c[a} \bar \nabla \csdot \IIo_{b]}
+ 2\bar \nabla_{[a} \Fo_{b]c}
 - \tfrac{1}{(d-1)(d-2)} \bar{g}_{c[a} \bar \nabla_{b]} K.
\end{align*}
We must then contract with the left-hand side with the same injector product. 
Because we use a scale where $H^g=0$,
$$\bar{Z}^A_a \bar{Z}^B_b \bar{Z}^C_c \bar{Y}^D W_{ABCD}^\top = C_{abc}^\top \big|_{\Sigma}\, .$$
Showing that this contraction yields equality in Equation~\nn{Wt-Wb} is now equivalent to showing that, when $H^g = 0$,
the projected Cotton tensor is related to the hypersurface Cotton tensor by
$$C_{abc}^\top \big|_{\Sigma} = \overline{C}_{abc} + 2\bar \nabla_{[a} \Fo_{b]c} + \tfrac{2}{d-2} \IIo_{c[a} \bar \nabla \csdot \IIo_{b]} - \tfrac{1}{(d-1)(d-2)} \bar{g}_{c[a} \bar \nabla_{b]} K.$$
For that, first observe that
the projected covariant derivative of the first fundamental
  form obeys
$$\nabla^\top_a {\rm I}_{bc} \big|_{\Sigma} = -\II_{ab} \hat n_c - \II_{ac} \hat n_b \stackrel{\sss H^{\sss g}\!=0} = -\IIo_{ab} \hat n_c - \IIo_{ac} \hat n_b.$$
Applying this identity, the trace-free Fialkow-Gauss Equation~\nn{Aaron}, and the traced Codazzi--Mainardi equation, the projected Cotton tensor can be written in terms of the hypersurface Cotton tensor:
\begin{align*}
C_{abc}^\top \big|_{\Sigma} &= \left(\nabla_a \Rho_{bc} \right)^\top - (a \leftrightarrow b) \\
&= \nabla^\top_a \Rho^\top_{bc} +\IIo_{ab} \Rho_{\hat{n} c}^\top + \IIo_{ac} \Rho_{\hat{n} b}^\top + \hat n_b \IIo_a^{d} \Rho_{d c}^\top + \hat n_c \IIo_a^{d} \Rho_{b d}^\top - ( a \leftrightarrow b)  \\
&= \nabla^\top_a \Rho^\top_{bc} + \hat n_b \IIo_a^{d} \Rho_{d c}^\top + \hat n_c \IIo_a^{d} \Rho_{b d}^\top - ( a \leftrightarrow b) + \tfrac{2}{d-2} \IIo_{c[a} \bar \nabla \csdot \IIo_{b]} \\
&= \bar \nabla_a \Rho^\top_{bc} - ( a \leftrightarrow b) + \tfrac{2}{d-2} \IIo_{c[a} \bar \nabla \csdot \IIo_{b]} \\
&= \bar \nabla_a \bar{\Rho}_{bc} + \bar \nabla_a \Fo_{bc} + \tfrac{\bar g_{bc}}{2(d-1)(d-2)} \bar \nabla_a K - ( a \leftrightarrow b) + \tfrac{2}{d-2} \IIo_{c[a} \bar \nabla \csdot \IIo_{b]} \\
&= \overline{C}_{abc} + 2\bar \nabla_{[a} \Fo_{b]c} + \tfrac{2}{d-2} \IIo_{c[a} \bar \nabla \csdot \IIo_{b]} - \tfrac{1}{(d-1)(d-2)} \bar{g}_{c[a} \bar \nabla_{b]} K.
\end{align*}
The second above relies on previous display, while the third relies the trace of the Codazzi--Mainardi equation.
The second last line uses the Fialkow--Gau\ss\ equation.
This completes the proof.

\end{proof} 

\begin{remark}
The corollary does not contain an explicit formula for the tractor $V_{AB}$ for reasons of brevity only. It  
measures  the difference between hypersurface and bulk Bach tensors. While explicit knowledge of the tensor  content of $V_{AB}$ is unnecessary for the computations that follow, it is nonetheless interesting. 
A computer-aided computation gives
$$V_{AB} = \bar q(U_{ab}) + \tfrac{1}{(d-1)(d-4)(d-5)}\hh \bar{h}_{AB} U,$$
where, for $d \neq 7$, \begin{equation}
\begin{split}
\Gamma(\ce \Sigma[-4]) \ni U 
=&\phantom{+}\tfrac{d-3}{d-1} K^2 + 2\IIo^{ad} \IIo^{bc} \overline{W}_{abcd} - 2(d-3) \IIo \csdot \Fo \csdot \IIo + (d-3)(d-5) \Fo^2  \\
&+ \tfrac{1}{d-7} \left(\bar{D}_A L_{BC} \right) \left(\hdb^A L^{BC} \right) -  L^{BC} N^A N^D \delta_R W_{ABCD}\, ,
\end{split}
\label{dasboot}
\end{equation}
\begin{align*}
\Gamma(\odot^2 T^*_{\circ} \Sigma[-2]) \ni
U_{ab} =&\,  \tfrac{1}{d-5} \overline{B}_{ab}
- \tfrac{1}{d-4} B^\top_{(ab)\circ}  
+\tfrac 2{d-7} E \Fo_{ab}
\\&
+ \tfrac{1}{6(d-1)(d-2)} \bar{\nabla}_{(a} \bar{\nabla}_{b)\circ} K
- \tfrac{1}{(d-2)(d-3)}  \TwoN_{ab} \bar{\nabla}\csdot \bar{\nabla} \csdot \TwoN 
+ \tfrac{1}{(d-2)^2} \bar{\nabla}\csdot \TwoN_{(a} \bar{\nabla}\csdot \TwoN_{b)\circ} 
\\&
+ 2 H C^\top_{\hat n(ab)} 
+  H^2 W_{\hat nab\hat n}  
- \tfrac{1}{3(d-1)(d-2)} K \overline{\Rho}_{(ab)\circ} 
- \tfrac{1}{d-3} \TwoN_{ab} \TwoN  \csdot \overline{\Rho}  
\, .
\end{align*}
Here the operator $E\in \operatorname{End}\big(\Gamma(\odot_\circ^2 T^*\Sigma))$ is defined by
$$E  \mathring X_{ab}:=
\bar \Delta \mathring X_{ab}- 
\bar\nabla_{(a}\nabla \csdot \mathring X_{b)\circ}
-(d-3) \bar \Rho_{c(a}\mathring X^c_{b)\circ}
-2 \bar J \mathring X_{ab}\, .$$
When  $\bar d=6$, the operator $E$ defines a conformally invariant map
$\Gamma(\odot_\circ^2 T^*\Sigma)\to \Gamma(\odot_\circ^2 T^*\Sigma[-2])$. Note that this calculation recovers the fifth fundamental form described in the introduction.
\end{remark}

\color{black}

A further corollary of the Gau\ss--Thomas equation in Theorem~\ref{GTF} characterizes the Fialkow tractor in terms of the $W$-tractor and tractor second fundamental form and
generalizes Equation~\nn{Aaron}.
\begin{corollary} \label{Fialkow-tractor1}
Let $7\neq d>5$. Then the Fialkow tractor obeys
\begin{align*}
(d-3) F_{AB} &= \left( L^C_A L_{CB} - \tfrac{1}{d-1} K \overline{h}_{AB} - W_{NABN} \right) - \tfrac{1}{d-1} X_{(A} \hdb_{B)} K
- \tfrac{1}{(d-4)(d-5)} X_A X_B U \, ,
\end{align*}
where $U\in \Gamma(\ce \Sigma[-4])$ is the density built from curvatures given in Equation~\nn{dasboot}.
\end{corollary}

\begin{proof}
The proof  amounts to tracing Equation~\nn{Wt-Wb} with the hypersurface tractor metric. Note that $U$ is given in the previous remark.
\end{proof}

 \bigskip

\subsection{Normal Operators}

Conformally invariant  operators that take derivatives in directions normal to the hypersurface are important, especially for the construction of higher fundamental forms.
The first of these is obtained by asking
 how the Thomas-$D$ operator acts in normal directions; indeed
 there exists  a canonical conformal Robin operator produced this way~\cite{Goal}:

\begin{definition}\label{swallow}
The operator
$$
\delta_R: \Gamma(\ct^\Phi M[w])
\to \Gamma(\ct^\Phi M[w-1])\big|_\Sigma
$$
defined, for $w\neq 1-\frac d2$, by 
$$
\delta_R T := N^A \hd_A T \big|_\Sigma\, ,
$$
and, for $w=1-\frac d2$, by \begin{equation}\label{Rob}
\delta_R T \stackrel g{:=} \big(\nabla_{\hat n} -(1-\tfrac d2) H^g\big) T\big|_\Sigma\, ,
\end{equation}
for any $g\in\cc$,
is termed the {\it tractor Robin operator}.
\end{definition} 

\begin{remark}\label{throwmehere}
In fact, for any weight $w$ tractor $T$ and $g\in \cc$, the tractor Robin operator obeys
$$
\delta_R T\stackrel g=
(\nabla_{\hat n} T - w H^g T)|_\Sigma\, ,
$$
which shows, by a dimensional continuation argument, that Equation~\nn{Rob} defines a tractor in the stated codomain.
Also, in the special case that $T$ is a conformal density, the above is the conformally invariant Robin combination of Neumann and Dirichlet operators first constructed by Cherrier~\cite{Cherrier}.
\end{remark}

Higher order conformally invariant analogs of the Robin operator acting on  a weight $w$ tractor~$T$, can  be defined as follows~\cite{GPt} 
$$
\Gamma(\ct^\Phi M[w-k])\big|_\Sigma\ni\delta^{(k)}_R T:=
 N^{A_1} \cdots
N^{A_{k-1}} 
\delta_{\rm R}
D_{A_1}\cdots D_{A_{k-1}}T\, .
$$
It is not difficult to verify that the operator $\delta^{(k)}_R$ has transverse order at most~$k$ (see Section~\ref{herecomethI}).  This bound is saturated at generic weights. One such example of the above family of operators is given below.
\begin{lemma} 
Let $\tau \in \Gamma(\ce M[w])$. Then, given $g \in \cc$ where $\tau=[g;t]$, 
\begin{align*}
\delta_R^{(2)} t \stackrel g=& (d+2w-3) \Big[ \hat n^a \hat n^b \nabla_a \nabla_b t - (2w-1) H \nabla_{\hat n} t + w(P_{\hat n \hat n} + \tfrac12 (2w-1) H^2 )t \Big] \\
& -\Big[\bar{\Delta} + w \Big(\bar{J} + \tfrac{1}{2(d-2)} K \Big) \Big] t\,.
\end{align*}
When $w = \frac{3-d}{2}$,
$$\delta_R^{(2)} =
-\square_Y -\tfrac{d-3}{4(d-2)} K \, ,$$
where $\square_Y := \bar{\Delta}+\tfrac{3-d}{2} \bar{J}$ is the Yamabe operator of $(\Sigma,\bar \cc)$. \end{lemma}

\begin{proof}
The proof can be found in \cite{GPt}.
\end{proof}

\color{black}

Observe that we can extend $\delta^{(k)}_R$ to
a map $\Gamma(T^\varphi M[w])\to 
\Gamma(T^\varphi \Sigma[w-k])$
 acting on trace-free sections of
 $T^\varphi M[w]$,
 by the composition of maps
 $
\bar q{}^*\circ \bar r \circ \otop \circ
\delta^{(k)}_R\circ
q 
$ whenever this is defined at weight~$w$, 
and will denote this also by
$\delta^{(k)}_R$ with $\delta^{(1)}_R\equiv \delta_R$. We use the same notation for the operator generalizing this to a slightly larger class of weights.  
\begin{lemma} \label{ho-robin} 
Let $t \in \Gamma(\odot^2_\circ T^*M[w])$. Then, given $g\in \cc$ and when $w\neq 3$,
\begin{eqnarray*}
\delta_R\: t_{ab}&\stackrel g=& \otop \left[\left(\nabla_{\hat n} - (w-2) H \right) t_{ab} + \tfrac{2}{w-3} \bar{\nabla}_{(a} t_{\hat n b) \circ}^\top  \right]
\, ,
\end{eqnarray*}
and when $w\neq 4,-d,2-d$,  
\begin{small}
\begin{eqnarray*}
\delta_R^{(2)} t_{ab}&\stackrel g= \otop \!\!\!\!\!&
\Big\{(d+2w-7) \Big[ \left(\hat n^c \hat n^d \nabla_c \nabla_d t_{ab} \right)
- (2w-5) H \left(\nabla_{\hat n} t_{ab} \right)
+ \tfrac{4}{w-4} \bar{\nabla}_{(a} \left( \hat{n} \csdot \nabla_{\hat n} t \right)_{b)}^\top \\&&\qquad\qquad\qquad
+ (w-2)\Big( P_{\hat n \hat n} + \tfrac{1}{2} (2w-5) H^2\Big) t_{ab} 
- \tfrac{4(w-2)}{w-4} H \bar{\nabla}_{(a} t_{\hat{n} b)}^\top
 - \tfrac{2w}{w-4} (\bar{\nabla}_{(a} H) t_{{\hat n} b)} \Big] \\[2mm]&&
- \tfrac{4}{d+w-2} \IIo_{ab} \bar{g}^{cd} \left(\nabla_{\hat n} t_{cd} \right)  \\[3mm]&&
- \big(\bar{\Delta}+(w-2)\bar{J} \big)t_{ab}^\top
-\tfrac{4}{w-4} \bar{\nabla}_{(a} \bar{\nabla} \csdot t_{b)}^\top
+ 4 \bar{P}_{(a} \csdot t_{b) }
+ 4 \Fo_{(a} \csdot t_{b)}
\\[.5mm]&&
+\tfrac{4}{d+w-2} \IIo_{ab} \bar{\nabla} \csdot t_{\hat{n}}^\top
- 4 \IIo_{c(a} \bar{\nabla}^c t_{\hat{n} b)}^\top 
- \tfrac{4}{w-4} \bar{\nabla}_{(a} \big(\IIo_{b)}^c t_{\hat{n} c}^\top\big)
- \tfrac{4(d+w-5)}{d-2} (\bar{\nabla} \csdot \IIo)_{(a} t_{\hat{n} b)}
 \\[1mm]&&
+\tfrac{2(d+2w-6)}{(w-3)(w-4)} \bar{\nabla}_a \bar{\nabla}_b t_{\hat n \hat n} 
- \tfrac{2(d+2w-6)}{w-3} \bar{P}_{ab} t_{\hat n \hat n}
- \tfrac{4(w-2)}{d+w-2} H \IIo_{ab} t_{\hat n \hat n}
\\[2mm]&&
+ 2 \IIo^2_{(a} \csdot t_{b)}
- \tfrac{4}{d+w-2} \IIo_{ab} \IIo \csdot t
- \left( \tfrac{2}{d-1} + \tfrac{w-6}{2(d-2)} \right) K    t_{ab} 
- 2 \IIo_{ab}^2 t_{\hat n \hat n}
 \Big\} \,.
\end{eqnarray*} 
\end{small}
\end{lemma} 
\begin{proof}
We proceed by working in a generic dimension $d$ and with generic weight $w$ rank-$2$ trace-free symmetric tensors $t \in \Gamma(\odot^2_\circ T^*M[w])$. For generic weights and dimensions, we may compute the composition of maps $\bar q{}^*\circ \bar r \circ \otop \circ \delta^{(k)}_R\circ q$. In both the cases $k=1,2$, we must compute $\bar{q}^* \circ \bar{r} \circ \otop$, so  we first  compute this operator on a general rank-$2$ tractor $T \in \Gamma(\odot^2_\circ \ct M[w])$. Note that $\otop$ maps $T \mapsto \mathring{\bar{T}}$ where $\mathring{\bar{T}} := (I_A^{A'} I_B^{B'} - \tfrac{1}{d+1} I_{AB} I^{A'B'}) T_{A'B'}|_{\Sigma}$. Because $\bar{r}$ achieves $X \csdot 
\bar r\big(
\mathring{\bar{T}} \big)= 0$, we have that $(\bar{q}^* \circ \bar{r} \circ \otop)(T)_{ab} = \bar{Z}_{Aa} \bar{Z}_{Bb} \bar{r}(\mathring{\bar{T}})^{AB}$. Thus, for generic dimensions and weights,  we have that
\begin{align*}
(\bar{q}^* \circ \bar{r} \circ \otop)(T)_{ab} 
=& \phantom{-} \bar{Z}_{Aa} \bar{Z}_{Bb} \mathring{\bar{T}}^{AB} \\
&-\tfrac{2}{w} \bar{Z}_{B(a} \bar{\nabla}_{b)} (X_C \mathring{\bar{T}}^{CB}) 
+ \tfrac{2}{w(d+1)} \bar{\bm g}_{ab} \hdb_C (X_D  \mathring{\bar{T}}^{CD}) \\
&+\tfrac{1}{w(w+1)} \bar{Z}_{Bb} \bar{\nabla}_a \hdb^B (X_C X_D  \mathring{\bar{T}}^{CD}) \\
&+ \tfrac{8 \bar{\bm g}_{ab}}{(d-1)(d+1)(d+2w+1)} \hdb_C (X_D  \mathring{\bar{T}}^{CD})\, .
\end{align*}
We now  compute each of the terms above. First, using Equation~\nn{Iconnect}, we obtain
$$\bar{Z}_{Bb} \bar{\nabla}_a (X_C \mathring{\bar{T}}^{CB}) = {\bar{\nabla}}_a^{} \mathring{\bar{T}}^{+}_b + \bar{\bm g}_{ab} \mathring{\bar{T}}^{+-}  + \bar{P}_{ab} \mathring{\bar{T}}^{++}  \,,$$
$$\bar{Z}_{Bb} \bar{\nabla}_a \hdb^B (X_C X_D  \mathring{\bar{T}}^{CD}) =\big[
 \bar{\nabla}_a \bar{\nabla}_b- \tfrac{ \bar{\bm g}_{ab}}{d+2w+1} (\bar{\Delta} + (w+2) \bar{J})  + (w+2) \bar{P}_{ab} \big]  \mathring{\bar{T}}^{++}\,.$$
Next (see for example the Appendix B of~\cite{Abrar}), we have
$$\hdb_C (X_D  \mathring{\bar{T}}^{CD}) = \tfrac{1}{d+2w-1} \big(-[\bar{\Delta} - (d+w-1) \bar{J}] \mathring{\bar{T}}^{++} + (d+2w+1) \bar{\nabla}_a \mathring{\bar{T}}^{+a} + (d+w-1)(d+2w+1) \mathring{\bar{T}}^{+-} \big)\,.$$
Finally, because $ \mathring{\bar{T}}^{AB} $ is hypersurface {tractor}-trace-free, we have that $0 =  \mathring{\bar{T}}_a^a + 2  \mathring{\bar{T}}^{+-}$. Thus, 
$$\bar{Z}_{Aa} \bar{Z}_{Bb} \mathring{\bar{T}}^{AB} + \tfrac{2 \bar{\bm g}_{ab}}{d-1}  T^{+-} = \mathring{\bar{T}}_{ab} - \tfrac{1}{d-1} \bar{\bm g}_{ab} \mathring{\bar{T}}_c^c =: \mathring{\bar{T}}_{(ab)\circ}\,.$$
Substituting these identities into the above display for $(\bar{q}^* \circ \bar{r} \circ \otop)(T) $ gives
\begin{align*}
(\bar{q}^* \circ \bar{r} \circ \otop)(T)_{ab} 
=& 
 \tfrac{1}{w(w+1)} \bar{\nabla}_{(a} \bar{\nabla}_{b)\circ} \mathring{\bar{T}}^{++} 
 - \tfrac{2}{w} \bar{\nabla}^{}_{(a}  \mathring{\bar{T}}^{+}_{b) \circ}
- \tfrac{1}{w+1} \bar{P}_{(ab)\circ}  \mathring{\bar{T}}^{++} 
+ \mathring{\bar{T}}_{(ab)\circ} \,.
\end{align*}
Proving the lemma now amounts to computing the components of $\mathring{\bar{T}}$  when $T = \delta_{\rm R} \circ q(t)$ or $T = \delta_{\rm R}^{(2)} \circ q(t)$. Note that by construction,
\begin{align*}
\mathring{\bar{T}}^{++} &\eqSig T^{++} \, ,\\
\mathring{\bar{T}}^{+}_{b} &\eqSig \bar{\bm g}^{c}_{b} T^{+}_{c}\, , \\
\mathring{\bar{T}}_{(ab)\circ} &\eqSig \bar{\bm g}^{c}_a \bar{\bm g}^{d}_b T_{cd} + \tfrac{1}{d-1}
\bar{\bm g}_{ab} (2 T^{+-} + T_{\hat n \hat n} )\,.
\end{align*}
Thus, we can simplify our calculations by only computing the components of $T$ appearing on the right hand side above.

\bigskip
\noindent
We begin with $T = \delta_{\rm R} \circ q(t)$. We can use Equation~\nn{usemeoft} and Definition~\nn{swallow} to show that
\begin{align*}
X_A T^{AB} = X_A \delta_{\rm R} q(t)^{AB} &= \delta_{\rm R} X_A q(t)^{AB} - N_A q(t)^{AB} \\
&= - N_A q(t)^{AB} \,,
\end{align*}
where the second equality holds because $X \csdot q(t_{ab}) = 0$ by definition. Thus,  using~\nn{insertions}, we have
$$T^{++} = 0\,, \qquad T^{+}_a = -t_{\hat{n} a}\,, \qquad T^{+-} = \frac{\hat{n} \csdot \nabla \csdot t}{d+w-2}\,.$$
Using again~\nn{insertions} as well as Equation~\nn{Iconnect}, Remark~\nn{throwmehere}, and the fact that $q(t)$ has weight~$w-2$, we have that
$$T_{ab} = \nabla_{\hat n} t_{ab} - (w-2) H t_{ab} - \frac{2 \hat{n}_{(a} \nabla \csdot t_{b)}}{d+w-2}\, ,$$
so that (in the scale $g$)
\begin{align*}
\mathring{\bar{T}}_{(ab)\circ} &\eqSig \top [\nabla_{\hat n} t_{ab} - (w-2) H t_{ab}] + \tfrac{1}{d-1}\bar{g}_{ab} (\hat{n}^a \hat{n}^b \nabla_{\hat n} t_{ab} - (w-2) H t_{\hat n \hat n}) \\
&\eqSig  \otop [\nabla_{\hat n} t_{ab} - (w-2) H t_{ab}]\,.
\end{align*}
Combining  the above and noting that $T$ has weight $w-3$, we have
$$\bar{q}^* \circ \bar{r} \circ \otop \circ \delta_{\rm R} \circ q(t) =  \otop \Big[\nabla_{\hat n} t_{ab} - (w-2) H t_{ab} + \tfrac{2}{w-3} \bar{\nabla}_{(a} t_{\hat n b)\circ}^\top \Big]\,.$$

\bigskip
\noindent
We finish the proof by handling $T = \delta_{\rm R}^{(2)} \circ q(t)$. The tractor computations in this case become unwieldy but not difficult. To manage these, we used the computer algebra system FORM~\cite{Jos};  this computation is documented in~\cite{FormFiles}. The lemma results from this computation. 
\end{proof}

To uncover further invariants of the embedding $\Sigma\hookrightarrow (M,\cc)$ and associated operators, we introduce more powerful holographic machinery.

\subsection{Holography}\label{herecomethI}

A useful way to treat hypersurfaces is in terms of defining functions: recall that $s$ is a {\it defining function}
for $\Sigma$ if $ s \in C^\infty{M}$, $\Sigma=\{P\in {M}\,  |\, s(P)=0\}$, and $\ext s|_\Sigma$ is nowhere vanishing. We shall also assume that $s>0$ on $M^+$.
A weight $w=1$ conformal density given by $\sigma=[g;s]$, where $s$ is a defining function and $g\in\cc$,  is called a {\it defining density}.
Defining functions can be used to analyze hypersurface invariants; for example, the second fundamental form can be written as $\II_{ab} = {\mathscr P}_{ab}\big|_{\Sigma}
$
where 
$$
{\mathscr P}_{ab}=
\Big(
\nabla_a - |\ext s|_g^{-1}\, (\nabla_a s) \nabla_{|\ext s|_g^{-1} \operatorname{grad}  s}\Big)\, 
\frac{\nabla_b s}{\hh |\ext s|_g}\, .
$$ The quantity $ {\mathscr P}_{ab}$ is an example of a {\it
  preinvariant}, namely a smooth (and suitably polynomial) function of
generic metrics $g$ and defining function $s$, whose restriction to
$\Sigma$ is independent of the choice of defining function $s$, and
therefore defines some hypersurface invariant, see~\cite{Will1} for
the detailed definition. Let us choose local coordinates $(s,y^i)$ in a
neighborhood of $\Sigma$ where $i=1,\cdots, d-1$ 
 such that along $\Sigma$  the vector fields $\partial/\partial y^i$ are tangent to
$\Sigma$. Then the {\it transverse order} of a preinvariant~${\mathscr P}$
is defined by writing~${\mathscr P}$ as a function of the coordinate components of $g$ and their
derivatives, and computing the minimum, over all such coordinate representations of~${\mathscr P}$, of the highest order of $\partial/\partial s$
derivatives of~$g$ upon restriction 
to $\Sigma$. 
The transverse order is an invariant of the corresponding hypersurface
invariant: for example, the second fundamental form has transverse
order one.  We will also say that an operator $\operatorname{O}$ has
transverse order $k\in {\mathbb Z}_{\geq 0}$ when there exists $v$ in
the domain of $\operatorname{O}$ such that $ \operatorname{O}(s^k
v)\big|_\Sigma\neq 0$, but $ \operatorname{O}(s^{k+1}
v^\prime)\big|_\Sigma=0$, for all $v^\prime$ in the domain of
$\operatorname{O}$.

We can use this notion of preinvariants to show that Definitions~\ref{SLAM} and~\ref{DUNK} are well-posed, as established by the following result.

\begin{lemma}\label{Sparky}
Suppose $2 \leq n\leq d-1$. Then if the conformal embedding $\Sigma\hookrightarrow (M,\cc)$ is such that at least one $\ell$th fundamental form vanishes for every $2\leq \ell <n$, then up to an overall non-zero coefficient, there is a unique ${\it n}$th fundamental form. 
\end{lemma}

\begin{proof}
We begin by considering the leading transverse order term in the preinvariant expression for an $n$th fundamental form.
Using coordinates $\{s, y^i\}$ in a collar neighborhood $I \times \Sigma \subset M$, we can express any preinvariant in terms of a defining function $s$, partial derivatives $\partial_s$ and $\partial_i=\partial/\partial y^i$, the metric components~$g_{ab}$, and the components of its inverse~$g^{ab}$. Because fundamental forms are conformally invariant, it is useful to view part of the preinvariant alphabet---the metric~$g$ and the defining function~$s$---as representatives of weighted densities: the metric is a weight $2$ representative of the conformal class of metrics $\cc$ and the defining function is a representative of a defining density $\sigma = [g;s]$ with weight $1$. We next determine the leading transverse order term in the preinvariant expression for an $n$th fundamental form using Definition~\ref{ff-def}.

From the definition of transverse order, the leading derivative term in the preinvariant for the~$n$th fundamental form must be of the form
$${\rm O}_{(ab)}{}^{cd}\, \partial_s^{n-1} g_{cd} |_{\Sigma}\, ,$$
where, as an operator, ${\rm O}_{(ab)}{}^{cd}$ has transverse order $0$. 
Moreover the above is annihilated by the normal and hypersurface trace.
Note that the weight of an $n$th fundamental form is $3-n$ and its transverse order is $n-1$. By considering only conformal transformations by a constant we may still  analyze expressions such as that displayed above in terms of weights. Because the weight of the operator $\partial_s$ is $-1$, the weight of the above display is $3-n+w_{\rm O}$ where the operator~${\rm O}_{(ab)}{}^{cd}$ has weight~$w_{\rm O}$. Hence we must  have that~$w_{\rm O} = 0$. By an elementary weight argument, ones sees that this operator is algebraic and therefore made only from the metric, its inverse, and a preinvariant for the conormal. Together with
elementary $O(d)$ and $O(d-1)$ representation theory, this   implies that (along $\Sigma$) this operator must be a non-zero multiple of the trace-free hypersurface projector, and hence 
proportional to
$$\otop_{\!\rm e}\hh\big(\partial_s^{n-1} g_{ab} 
\big)\big|_{\Sigma}\,,$$
where $\otop_{\!\rm e}$ is any preinvariant expression for the operator $\otop$.

Now suppose that $L^{(n)}$ and~$L^{(n)}{}'$ are two $n$th fundamental forms with the same coefficient for the above-displayed term and the conformal embedding $\Sigma \hookrightarrow (M,\cc)$ is such that an $\ell$th fundamental form vanishes for every $\ell < n$. We then seek to show that  $L^{(n)} - L^{(n)}{}' \eqSig 0$. Clearly, because $L^{(n)}$ and~$L^{(n)}{}'$ have the same leading term, their difference must have transverse order at most $n-2$. Put another way,
$$L^{(n)}_{ab} - L^{(n)}_{ab}{}' = {\rm P}_{(ab)}{}^{cd} \partial_s^{n-2} g_{cd}\big|_{\Sigma} + \text{ lower-order terms}\,,$$
where ${\rm P}_{(ab)}{}^{cd}$ is a preinvariant operator with weight $-1$ and transverse operator order $0$. But then $\partial_s^{n-2} g_{ab}|_{\Sigma}$ can be rewritten as an $(n-1)$th fundamental form plus lower-order terms. 
Thus, a proof via induction only requires that we check that the second fundamental form is unique up to an overall non-zero  coefficient, which is again easily verified
by an elementary weight and representation theoretic argument. 
\end{proof}

Given a hypersurface embedding $\Sigma \hookrightarrow (M,{\mathcal S})$,
where $M$ is a smooth manifold endowed with any structure ${\mathcal S}$, and a hypersurface invariant $\overline{\mathcal P}$, we call a smooth extension ${\mathcal P}$ of $\overline{\mathcal P}$ to $M$, a {\it holographic formula} when 
${\mathcal P}$ is canonically determined by 
the structure ${\mathcal S}$. 
In many cases, we use this terminology also when ${\mathcal P}$ (or ${\mathcal P}({\mathcal S})$)  is only uniquely determined up to some asymptotic order in a defining function. The notion of a holographic formula extends to  geometric differential operators in a straightforward way.

\medskip

The possibility  to define hypersurface invariants
in terms of preinvariants that do not depend on any particular choice of defining function can  be leveraged by a clever choice of defining function; such a choice leads to holographic formul\ae.
In fact, a key result~\cite{Loewner,Aviles,Maz,ACF} is that there is a unique metric~$g^o$
on the $d$-dimensional manifold  $M^+$ with boundary $\partial M^+=\Sigma$, whose scalar curvature obeys
\begin{equation}\label{Sco}
Sc^{g^o}=-d(d-1)\, ,
\end{equation}
and such that in $M^+$
$$g^o=s^{-2} g\, ,$$
for some suitably smooth defining function~$s$ and $g\in \cc$.
The problem of finding $s$ such that  the pair $(g,s)$ give such a metric~$g^o$ is a general  case of the Loewner--Nirenberg problem~\cite{Loewner}. The metric ~$g^o$ is called the {\it singular 
Yamabe metric}. Observe that if $(g,s)$ determines the singular Yamabe metric~$g^o$, then so does  $(\Omega^2g,\Omega s)$. In other words, $g^o$ is determined by the conformal density $\sigma:=[g;s]\in \Gamma(\ce M[1])$.

An important insight is that, because $\sigma$ is uniquely defined on $M^+$ by the data $\Sigma\hookrightarrow (M,\cc )$, the jets  of $\sigma$ can be used to efficiently study the conformal hypersurface embedding. Recasting Equation~\nn{Sco} as a tractor equation is extremely helpful for this. To understand why, first recall that formul\ae\ for Riemannian hypersurface invariants simplify when given a {\it unit defining function}~$s$, {\it viz.} a defining function that satisfies
$$g(n,n)=|\ext  s|_g^2=1\, ,$$ 
where $n:=\ext s$. (It suffices for many applications to find $s$ obeying the above condition in a neighborhood of $\Sigma$.)
For example, the second fundamental form is given by the Hessian of~$s$ restricted to $\Sigma$.
There is a neat conformal
analog of this picture: A {\it unit conformal defining density}
 is a weight one  conformal density~$\sigma$ subject to the condition
 \begin{equation}\label{II}
 h(I_\sigma,I_\sigma)=1\, ,
 \end{equation}
 where $I_\sigma:=\hat D \sigma$. 
 In general, we call a tractor obtained by acting with the hatted Thomas-$D$ operator on a scale, a {\it scale tractor}, and drop the subscript $\sigma$ when it is clear to do so (the same convention will be applied to other objects whose dependence of $\sigma$ is denoted this way).
 
 Now, note that
 given $g\in \cc$ where $\sigma=[g;s]$, 
\begin{equation}\label{Is}
I^A =\hd \sigma \stackrel g= \begin{pmatrix}
s \\
\ext  s\\
-\tfrac1d(\Delta^g s + J^g s)
\end{pmatrix}\, ,
\end{equation}
so that
$$
I^2:=h(I,I)\stackrel g = |\ext  s|_g^2 -\frac2d\,  s (\Delta^g s + J^g s)\, .
$$
Away from $\Sigma$, written in terms of the metric $g^o=s^{-2}g$, the right hand side of the above display equals $-2 J^{g^o}/d=-\Sc^{g^o}/\big(d(d-1)\big)$.
Hence the unit conformal Condition~\nn{II} is equivalent to the singular Yamabe Equation~\nn{Sco}.

Given the data $\Sigma\hookrightarrow (M,\cc)$, the one-sided solution $g^o$ to the singular Yamabe problem of Equation~\nn{Sco} in $M^+$ depends on  global information of $M^+$ and $\cc$. However,  locally determined metrics  $g^o$ on~$M^+ \subset M^d$ that are smooth up to the boundary and obey 
$$
\Sc^{g^o}=-d(d-1)+{\mathcal O}(s^d)
$$
in an open neighborhood of $\Sigma$,
where $s$ is any defining function for $\Sigma$,  always exist~\cite{ACF} and are easily constructed~\cite{Will1}. We term such a metric $g^o$ an {\it asymptotic singular Yamabe metric}.
Here the notation~${\mathcal O}(s^k)$ denotes $s^k$ times any smooth function, and we will use a similar notation involving powers of densities in an obvious way.
In terms of defining densities~$\sigma$ and their corresponding scale tractor, the above-displayed equation stipulates  that
\begin{equation}\label{stip}
I^2_\sigma = 1 + \sigma^d {\mathcal B}\, ,
\end{equation}
where ${\mathcal B}\in \Gamma(\ce M[-d])$. 
Solutions $\sigma$ to the above equation are termed {\it asymptotic unit defining densities}.
While these  are not unique, the quantity
\begin{equation}\label{fruity}
B_\Sigma := {\mathcal B}\big|_\Sigma \in \Gamma(\ce \Sigma[-d]) 
\end{equation}
is, and is termed the {\it obstruction density}. The obstruction density~$B_\Sigma$ is a local invariant of the conformal embedding $\Sigma\hookrightarrow (M,\cc)$, and is the obstruction to solving $I_\sigma^2=1$ smoothly on~$\overline{M^+}$~\cite{ACF,Will1}. Moreover, $B_\Sigma$ is variational, meaning that it is the functional gradient with respect to variations of the embedding $\Sigma\hookrightarrow M$
of a Willmore-type energy functional~\cite{GrSing}
(see also~\cite{RenVol}).

The key point now is that the first $d$ jets of $\sigma$ are uniquely defined by solving Equation~\nn{stip}, and therefore can be used to efficiently construct invariants of the conformal embedding~$\Sigma\hookrightarrow (M,\cc)$.
These are given as holographic formul\ae. 
 For example, the first two jets  are required for the scale tractor $I_\sigma|_\Sigma$ and this gives a holographic formula for the normal tractor.
  \begin{lemma}[Gover~\cite{Goal}]\label{LemGo}
Let $\sigma$ be a defining density for $\Sigma$ that obeys
$$
I_\sigma^2 = 1+{\mathcal O}(\sigma^2)\, ,
$$
where $I_\sigma= \hd \sigma$.
Then 
$$
I_\sigma\big|_\Sigma = N\, .
$$
 \end{lemma}
 Recall that the normal tractor $N$ encodes both the 
 unit normal 
 and mean curvature (in any scale). The trace-free part of the second fundamental form can be obtained holographically in several ways, the first example is given below.
 \begin{lemma}
Let $\sigma$ be a defining density for $\Sigma$ that obeys
$$
I_\sigma^2 = 1+{\mathcal O}(\sigma^2)\, ,
$$
where $I_\sigma= \hd \sigma$.
Then the projecting part of $\nabla I_\sigma|_\Sigma$ equals the trace-free second fundamental form.
 \end{lemma} 
 
 \begin{proof}
 Let $g\in \cc$ and $\sigma=[g;s]$.
 First note that $I_\sigma^2 = 1+{\mathcal O}(\sigma^2)$
 implies that
 $$
 |n
 |_g^2 = 1 + \tfrac{2s}{d}  (\nabla \csdot n  + Js) + {\mathcal O}(s^2)\, . 
 $$
 Then using the scale tractor
Equation~\nn{Is} and tractor connection Equation~\nn{Iconnect}, it follows that the projecting part of
$\nabla I_\sigma|_\Sigma$ 
is 
$$\Big(\nabla n  - \tfrac1d g \nabla\csdot n \Big)\Big|_\Sigma
=
\Big(\nabla^\top n + \hat n \nabla_{ n} n
- g H\Big)\Big|_\Sigma\, .
$$
In the above $n:=\ext s$ and, using the first display of this proof,  it follows that $n|_\Sigma = \hat n$.
In the above we have also used that Lemma~\ref{LemGo}
implies that $\tfrac 1d \nabla \csdot n|_\Sigma = H$. Moreover $\nabla_n n = \frac12 \ext |n|_g^2
=\frac1d n \nabla \!\cdot\! n +{\mathcal O}(s)$ which equals $\hat n H$ along $\Sigma$. Using $\bar g = g|_\Sigma - \hat n \odot \hat n$, and $\nabla^\top n |_\Sigma = \nabla^\top(n/|n|_g)  |_\Sigma $, the result now follows.

 \end{proof}

 Thanks to Lemma~\ref{LemGo}  we now have a canonical extension $I_\sigma$ of the normal tractor~$N$, and thus can construct a  canonical extension of the Robin operator of Definition~\nn{swallow}.
 \begin{lemma}
The operator 
$$
I\csdot D:\Gamma(\ct^\Phi M[w])
\to \Gamma(\ct^\Phi M[w-1])\, ,
$$
where 
$
I\csdot D\, T := I^A D_A T
$,
obeys
$$
I\csdot D \, T\big|_\Sigma = (d+2w-2) \, \delta_{\rm R} T\, .
$$
 \end{lemma}
 
 \begin{proof}
 The  proof of the above lemma is given in~\cite{Goal} (see also~\cite{BGnonlocal}) 
 and uses that $I\csdot D\,  T$, in a choice of scale
 $g\in \cc$ for which $\sigma=[g;s]$, is   given by
 $$
 I\csdot D T = (d+2w-2)(\nabla_n + w \rho_s) T
 -s (\Delta + w J^g) T\, ,
 $$  
 where $\rho_s:=-\frac1d (\Delta^g s + s J^g)$.
 \end{proof}

The above lemma presages a useful algebraic relationship between several objects present in the calculus so far.
\begin{proposition}[\cite{GW}] \label{sl2-algebra}
Suppose $\sigma \in \Gamma(\ce M[1])$ obeys $I_\sigma^2 \neq 0$, and denote by ${\sf h} : \Gamma(\ct^\Phi M[w]) \rightarrow \Gamma(\ct^\Phi M[w])$ the operator defined by ${\sf h}f = (d+2w)f$.
Then, viewing 
$x:=\sigma:\Gamma(\ct^\Phi M[w]) \rightarrow \Gamma(\ct^\Phi M[w+1])$
as a multiplicative operator and $y:= -\frac{1}{I^2} I \csdot D : \Gamma(\ct^\Phi M[w]) \rightarrow \Gamma(\ct^\Phi M[w-1])$ as a differential operator, commutators of the operators $(x, {\sf h}, y)$ satisfy the $\mathfrak{sl}(2)$ defining relations,
$$[{\sf h},x] = 2x \,, \;\; \left[x, y \right] = {\sf h}\,, \;\; \left[{\sf h}, y \right] = -2 y\, .$$
\end{proposition}

\begin{remark}
In many cases, we assume that $\sigma$ solves the singular Yamabe problem to some order $k \geq 2$. In these situations, we may neglect the $\frac{1}{I^2}$ coefficient when using the above-displayed relations because it equals unity  to sufficiently high order for the problem at hand. 
\end{remark}

The next lemma gives a holographic formula for the tangential Thomas-$D$ operator of Proposition~\ref{Dt-complex}.
\begin{lemma}\label{Dt-simple}
Let $w + \frac{d}{2} \neq 1, \frac{3}{2}, 2$ and $\sigma$ 
be a
unit conformal defining density for $\Sigma$ 
that obeys $I_{\sigma}^2 = 1 + \mathcal{O}(\sigma^2)$. 
Then, along $\Sigma$, 
\begin{equation}\label{DT}
\hd^T_A = \hd_A - I_A I \csdot \hd + \frac{X_A}{d+2w-3}\,  I\csdot \hd^2\, .
\end{equation}
\end{lemma}

\begin{proof}
Using $I^2 = 1 + \mathcal{O}(\sigma^2)$ and Proposition~\ref{leib-failure}, 
$$\IdD^2 \big|_{\Sigma} = \left.\left[ I_A I \csdot \hd \hd^A + \frac{X_A K}{d-2} \hd^A \right]\right|_{\Sigma} = N_A N_B \hd^B \hd^A + \frac{wK}{d-2}\, .$$
Substituting this identity into
the  definition of $\hd^T$ 
in Proposition~\ref{Dt-complex}
completes the proof.
\end{proof}

 \section{Fundamental Forms}
 \label{forms-section}

The singular Yamabe problem 
provides a canonical extension of $\IIo$.
Hence, in the spirit of holography, canonical  fundamental forms 
can be constructed by applying high transverse order  operators  to  this extension.

\subsection{The Canonical Extension of $\IIo$} \label{can-ext-IIo}
As noted in Lemma~\ref{thisismylabel}, the tractor analog $L$ of $\IIo$ is (essentially) given 
 by the tangential Thomas operator $\hd^T$ acting on an extension $N^{\rm e}$ of $N$. The scale tractor $I_\sigma$ of a unit conformal  defining density $\sigma$ gives a  canonical extension of the normal tractor $N$ and, in turn, of the tractor second fundamental form $L$. The projecting part of  the latter gives a conformally-invariant canonical extension of $\IIo$. We will compute this extension explicitly, but doing so requires the promised proof of Lemma~\ref{thisismylabel}.

\begin{proof}[Proof of Lemma~\ref{thisismylabel}]
Let $\sigma$ be an asymptotic  unit conformal defining density for $\Sigma \hookrightarrow (M,\cc)$. 
When not stated explicitly, 
in what follows $P^{AB}$ denotes $\hat D^A I_\sigma^B$ for $\sigma$ an asymptotic unit defining density. Because  $d>3$ we have that 
$I^2 = 1 + \mathcal{O}(\sigma^4)$. We will  apply Lemma~\ref{Dt-simple} and  that
the normal tractor obeys $N=I|_{\Sigma} $. Evaluating $\hd_A I^2$ yields, after some massaging using Proposition~\ref{leib-failure}, the identity \begin{equation}\label{thestatementwearelookingfor}I\csdot\hd I_A = \frac{\big(\hd I \big)^2\,  X_A}{d-2} + \mathcal{O}(\sigma^2)\, .\end{equation}
Defining the extension $K^{\rm e} := \big(\hd I_{\sigma} \big)^2$ of the rigidity density $K$, it follows using the above identity that $$\delta_R K^{\rm e}|_{\Sigma} = -2(d-3) M\, ,$$ see \cite{Will2}. Thus, noting that $I\csdot \hd X = I$ (see Equation~\nn{DX}), and employing the holographic formula for $\hd^T$ in Lemma~\ref{Dt-simple}, we then have the following identity:
$$\hd_{(A}^T I^{}_{B)} \big|_{\Sigma} = \hd_{(A} I_{B)} \big|_{\Sigma} - \tfrac{d-4}{(d-2)(d-3)} I_{(A} X_{B)} K |_{\Sigma} - \tfrac{2}{d-2} X_A X_B M\, .$$
A holographic formula for the tractor second fundamental form is given in~\cite{Will2}:
\begin{align} \label{L-P-formula}
\hd^{}_A I^{}_B |_{\Sigma} = L_{AB} + \tfrac{2}{d-2} I_{(A} X_{B)} K |_{\Sigma} + \tfrac{3d-8}{(d-2)(d-3)} X_A X_B M\,.
\end{align}
This formula applied to the second display above gives
$$\hd_{(A}^T I^{}_{B)} \big|_{\Sigma} = L_{AB} + \tfrac1{d-3} I_{(A} X_{B)} K |_{\Sigma} + \tfrac1{d-3} X_A X_B M\, .$$
This completes the proof.
\end{proof} 

\begin{remark}
In the proof above, the tractor $P:=\hd \hh I$ played a key {\it r\^ole} generating an extension of the tractor second fundamental form. The same tractor (and  higher transverse order counterparts) can also be used to  construct  tractor higher fundamental forms as well as  the rigidity density $K$ (and normal derivatives thereof). 
Note that we use the notation $P$ for the tractor $\hat D \hh I$ because, away from $\Sigma$, and computed in the scale $\sigma$,  one has that $q^*(P)$ equals the Schouten tensor of the singular metric $g^o$.
In what follows we will see that the tractor $P$ provides an interpolation between the Schouten tensor of the singular metric $g^o$ and trace-free second fundamental form.
\end{remark}

Note that $X \csdot P = 0$ because $X \csdot \hd \hh  T= wT$ for any weight $w$ tractor $T$, so  $q^*(P_{AB}) \in \Gamma(\odot_\circ^2 T^*M[1])$. Hence, writing $\sigma = [g;s]$, the canonical extension of $\IIo$ is given by
\begin{equation}\label{showmykeys}\IIo^{\rm e}_{ab} := q^*(P_{AB}) \stackrel g= \nabla_a n_b + s P_{ab} + \rho g_{ab}\,,\end{equation}
where $\rho = -\frac{1}{d} (\Delta^g s + J^gs)$. This result was computed by directly computing $\hd I$. Note that the above may be written $(\nabla_{(a} \nabla_{b)\circ}+\Rho_{(ab)\circ})s$, where the second order, conformally invariant operator in parentheses is termed the {\it almost Einstein operator}.
\medskip

\subsection{Canonical Transverse Differential Operators}

One could construct higher fundamental forms by extending Lemma~\ref{ho-robin} to $\delta_{\rm R}^{(k)}$ for $k>2$ and applying these operators to $\IIo^{\rm e}$, but this is inefficient. Fortunately, holography 
provides a canonical  extension $I \csdot D$ of the Robin operator. Indeed, applying powers of 
the operator $\ID_\sigma$, generically given by
\begin{equation}\label{home-brand}
\ID_\sigma = q^* \circ r \circ I \csdot D \circ q\, ,
\end{equation}
to $\IIo^{\rm e}$
 is the key step for producing formul\ae\ for higher fundamental forms. 
To compute $\ID$ (dropping the subscript $\sigma$ for brevity), we first need the following result.

\begin{lemma}\label{conformal-grad} 
Let $\theta = [g;t] \in \Gamma(\otimes^r T^* M[w])$ and $\sigma = [g;s] \in \Gamma(\ce M[1])$. Then
$$\nabla^\sigma \theta := [g; s \nabla t - (w-r) \ext s \otimes t +  (\ext s \circledast g)^\sharp t]\in\Gamma(\otimes^{r+1} T^* M[w+1]) \,,$$
where  for  a general covector $\omega$, we denote $(\omega \circledast g)_a^\sharp t^{}_b :=  \omega_b^{}  t_a^{} - g_{ba}^{} \omega^c t_c^{}$ (when $r=1$) and extends, in the standard Leibniz way, to higher rank $r$ tensors.
\end{lemma}

\begin{proof}
First note that, using the notation provided in the lemma, the Levi-Civita connection acting on $\theta := [g,t] \in \Gamma(\otimes^r T^*M[w])$ obeys 
$$\nabla^{\Omega^2 g} t^{\Omega^2 g} = \nabla^{\Omega^2 g}\big( \Omega^w t^g\big) = \Omega^w \left(\nabla^g t^g + (w-r) \Upsilon \otimes t^g -  (\Upsilon \circledast g)^\sharp t^g \right).$$
Here $\Upsilon := \ext \log \Omega$.
Further, denoting $n:=\ext s$, $n^{\Omega^2 g} = \ext (\Omega s) = \Omega (n^g + \sigma \Upsilon)$. 
Therefore, $n^{\Omega^2 g} \otimes t^{\Omega^2 g} = \Omega^{w+1} (n^g + \sigma \Upsilon) \otimes t^g$ and $(n^{\Omega^2 g} \circledast \Omega^2 g)^\sharp = \Omega ([n^g + s \Upsilon] \circledast g)^\sharp $. Combining these conformal transformations with the above display completes the proof.
\end{proof}

\noindent
Note that if $\sigma$ is a defining density for a hypersurface $\Sigma$ subject to $I_\sigma^2 = 1+{\mathcal O}(\sigma)$ then, along~$\Sigma$, the operator $\nabla^\sigma$ is tensor multiplication by the unit conormal $\hat n$. 

\smallskip

The operator $\ID$ of Equation~\nn{home-brand} is not defined for many weights, so we instead make the following definition.
\begin{definition} \label{ID-tensor-def} 
Let $\sigma\in \Gamma(\ce M[1])$ be any weight one density, $I_{\sigma} = \hd \sigma$, and $\tau_{ab} \in \Gamma(\odot^2_\circ T^*M[w])$ be any rank two, trace-free density-valued symmetric tensor where $w\neq 3,2-d$. Also let $\widehat M := M\setminus\mathcal{Z}(\sigma)$. Then we define the map
$$\ID_\sigma : \Gamma(\odot^2_\circ T^*\widehat M[w]) \rightarrow \Gamma(\odot^2_\circ T^*\widehat M[w-1])$$
by the following formula:
\begin{multline*}
\sigma\hh \ID_\sigma \tau_{ab} := -\bg^{cd} \Big(\nabla^\sigma_c \nabla^\sigma_d \tau_{ab}^{} + \tfrac{2d}{(w-3)(d+w-2)} \,  \nabla^\sigma_{(a} \nabla^\sigma_{|c} \tau_{d|b) \circ}^{} - \tfrac{4}{d} [\nabla^\sigma_{(a}, \nabla^\sigma_{|c}] \tau_{d|b)\circ}^{} \Big)\\\hspace{-.5cm} - \tfrac{4}{d} \sigma^2\hh W^c{}_{ab}{}^d \tau_{cd}^{} + \Big[
\Big(w-2+\tfrac{d-1}2\Big)^2
-\Big(\tfrac{d-1}2\Big)^2+2
 \Big] I_{\sigma}^2 \, \tau_{ab}^{} \, .
\end{multline*}
\end{definition}

The 
combination of terms in the above display is distinguished because 
they in fact allow the operator $\ID_\sigma$ to be defined along $\mathcal{Z}(\sigma)$. This result is described in the following lemma.

\begin{lemma} \label{ID-tensor} 
Let $\sigma\in \Gamma(\ce M[1])$ be any weight one density and $\tau_{ab} \in \Gamma(\odot^2_\circ T^*M[w])$
any rank two, trace-free density-valued symmetric tensor where $w\neq 3,2-d$.
Let $g\in \cc$ for which $\sigma=[g;s]$, $\tau_{ab}=[g;t_{ab}]$, and $\ID_\sigma \tau_{ab} = [g; \ID_s t_{ab}]$. Then
\begin{align*}
\ID_s t_{ab}{ =}& (d\!+\!2w\!-\!6) \!\left(\left[\nabla_n \!+ (w\!-\!2) \rho \right] t_{ab} - \tfrac{2(w-2)}{(w-3)(d+w-2)} n_{(a} \nabla \csdot t_{b) \circ} + \tfrac{2}{w-3} \left[n_c \nabla_{(a} t_{b) \circ}^c \!+ \!(\nabla n)_{(a} \csdot t_{b)\circ} \right] \right)  \\
&- s \left( \Delta t_{ab} + (w-2) J t_{ab} + \tfrac{2d}{(w-3)(d+w-2)} \nabla_{(a} \nabla \csdot t_{b)\circ}  - 4 P_{(a} \csdot t_{b) \circ}  \right)\, ,
\end{align*}
with $n:=\ext s$ and $\rho=-\frac 1d(\Delta s + J s)$, and $\ID_\sigma$ is a well-defined map
$$\ID_\sigma : \Gamma(\odot^2_\circ T^*M[w]) \rightarrow \Gamma(\odot^2_\circ T^*M[w-1])\,.$$
\end{lemma}

\begin{proof}
The proof amounts to a computation of $\ID_\sigma \tau_{ab}$ in a choice of scale away from $\mathcal{Z}(\sigma)$, applying Lemma~\ref{conformal-grad} twice,
and the identity $I_\sigma^2 \stackrel{g}= n^2 + 2 s \rho$. The resulting tensor is proportional to $s$ and hence the apparent singularity of $\ID_\sigma$ along $\mathcal{Z}(\sigma)$ in Definition~\ref{ID-tensor-def} is removable.
\end{proof}

For weights at which the composition of maps $q^* \circ r \circ I \csdot D \circ q$ acting on trace-free symmetric rank-2 tensors is defined, it is not difficult (but somewhat tedious) to show that Equation~\nn{home-brand} holds.
In this sense, the operator $\ID_\sigma$ is the action of $I \csdot D$ on rank-two symmetric tensors.  Also we define the analog of the operator $I \csdot \hd$ when $w \neq 3 -\frac{d}{2}$ by
$$\IDh_\sigma := \tfrac{1}{d+2w-6} \ID_\sigma\,.$$
For weights at which the appropriate composition of maps is defined, we also have that $\IDh_\sigma = q^* \circ r \circ I \csdot \hd \circ q$. \color{black}

Note that, given 
a defining density 
$\sigma$ subject to $I_\sigma^2|_\Sigma=1$, when $w\neq 3,3-\frac d2$,
the   operator~$\delta_R$ acting on tensors is expressed in terms of $\IDh_\sigma$ according~to  
$$\delta_R = \mathring\top \circ    \IDh_\sigma
  - \tfrac{2}{w-3} \IIo\, (\hat n \hh \cdot \, )^2  \, .$$
\color{black}

\medskip

\begin{remark}
Proposition~\ref{sl2-algebra} shows that  the operator $I \csdot D$ obeys an $\mathfrak{sl}(2)$ Lie algebra. 
Moreover, it is not hard to check that 
$$[\nabla^\sigma, \sigma] = 0\, . $$
Hence
 we may expect that $\ID_\sigma$ and $\sigma$ obey a similar algebra. Indeed
suppose $\sigma \in \Gamma(\ce M[1])$ obeys $I_\sigma^2 \neq 0$, and call ${\sf h} : \Gamma(\odot^2_\circ T^*M[w]) \rightarrow \Gamma(\odot^2_\circ T^*M[w])$ the operator defined by ${\sf h}\tau = (d+2w-4)\tau$
and view  
$\sigma:\Gamma(\odot^2_\circ T^*M[w]) \rightarrow \Gamma(\odot^2_\circ T^*M[w+1])$
as a multiplicative operator.
Then, acting on the subspace of sections $\tau_{ab}$ subject to the divergence condition
$$
\bg^{ab} \nabla^\sigma_a \tau_{bc}^{}=0\, ,
$$
straightforward algebra shows that commutators of the operators $(\sigma, h, -\frac{1}{I^2} \ID )$ satisfy the $\mathfrak{sl}(2)$ defining relations,
$$[{\sf h},\sigma] = 2\sigma \,, \;\; \left[\sigma, -\tfrac{1}{I^2} \ID \right] = {\sf h}\,, \;\; \left[{\sf h}, -\tfrac{1}{I^2} \ID \right] = -2 \left(-\tfrac{1}{I^2} \ID \right)\, .$$
Away from $\Sigma$ and in the $g^o=\bg /\sigma^2$ scale, the operator $\ID$ acting on a trace- and divergence-free symmetric tensor $\tau_{ab}$ is given by
$$
\ID\,  \tau_{ab} \stackrel{g^o}=
-\Big(\Delta
-\tfrac{2J}d\big[
\big(w-2+\tfrac{d-1}2\big)^2
-\big(\tfrac{d-1}2\big)^2-2
 \big] 
\Big)
 \tau_{ab}
+ 4  P^{}_{(a} \csdot \tau^{}_{b)\circ}\, .
$$
The above algebra 
is likely a key ingredient for constructing a calculus for the study of
tensor wave equations (see~\cite{Forms} for  the analogous problem for differential forms).
\end{remark}

\medskip

\subsection{Constructing Higher Fundamental Forms}
One might think that 
having constructed the canonical conformally-invariant operator $\ID_{\sigma}$ with transverse order $1$, as well as a canonical extension of $\IIo$, that we could now directly compute  higher fundamental forms by applying powers of $\ID_{\sigma}$ to $\IIo^{\rm e}$. However, as could be predicted by~\cite{GPt}, the $k$th power of 
$\ID_{\sigma}$ does not always produce an operator of transverse order $k$. 
More precisely, the candidate $(k-2)$th 
fundamental form
$$
\ID_{\sigma}^k \IIo^{\rm e}\big|_\Sigma\, ,
$$
has transverse order strictly less than $k+1$ when $k\geq\frac{d-1}2$. In this section, we first prove that fundamental forms with transverse order $n \geq \frac{d+1}{2}$ can indeed not  be constructed
by applying powers of $\ID$ to $\IIo^{\rm e}$.
We will then explicitly construct canonical fundamental forms with transverse order $n < \frac{d+1}{2}$. Finally we  construct canonical conditional fundamental forms with transverse order $n \geq \frac{d+1}{2}$ 
and use these ingredients to prove Theorem~\ref{uber-PE}.

\medskip

We begin by  defining operators that  
often
have the  appropriate transverse order.

\begin{definition}
Let $\Sigma \hookrightarrow (M^d, \cc)$ be a conformal embedding with asymptotic unit defining density $\sigma$.
Then, when $k=0$,  define $\delta_{d,w}^{(0)}:=\otop:\Gamma(\odot^2_\circ T^*M[w]) \rightarrow \Gamma(\odot^2_\circ T^*\Sigma[w])$.
For $k=1$ and $w\neq 3$,  call $$\delta_{d,w}^{(1)}:=\delta_{\rm R}: \Gamma(\odot^2_\circ T^*M[w]) \rightarrow \Gamma(\odot^2_\circ T^*\Sigma[w-1])\, ,$$
while for $k\in {\mathbb Z}_{\geq 2}$ and $w\notin
\{2-d,\ldots, k-d\}\cup
 \{3,\ldots ,k+2\}
$,  let
$$\delta_{d,w}^{(k)} := \delta_{\rm R} \circ \ID^{k-1}_\sigma : \Gamma(\odot^2_\circ T^*M[w]) \rightarrow \Gamma(\odot^2_\circ T^*\Sigma[w-k])
\, .$$
\end{definition}

Acting on sections of $\odot_{\circ}^2 T^* M[w]$ for $w\neq 3$, the transverse order of $\delta_{d,w}^{(1)}$ is $1$. When $w \not \in \mathbb{Z}$, the operator $\delta_{d,w}^{(k)}$ is always defined.
Moreover, when $2w$ is not an integer the operators $\delta^{(k)}_{d,w}$ for $k\geq 2$ 
have 
transverse order~$k$
(see Equation~\nn{finallyIhaveaname} below). However, when $k \in \mathbb{Z}_{\geq 2}$ and $2w \in \mathbb{Z}$, the operator $\delta_{d,w}^{(k)}$ may not be defined or it could fail to have transverse order $k$. In particular, for~$w$ an integer and $k\geq 2$, the operators $\delta_{d,w}^{(k)}$ are only defined in the three regions where $w$ obeys
$w<2-d$, $k-d<w<3$, or $k+2<w$
(the second of these could be empty).
 The following lemma characterizes the transverse order of $\delta_{d,w}^{(k)}$ in these cases:

\begin{lemma} \label{tr-deg}
Fix $d \geq 3$ and $ k\geq 2$, and let $w \in \mathbb{Z}$ be such that 
$$w<2-d\, , \: k-d<w<3\, , \:\mbox{ or } \: k+2<w\, .$$
Then, the transverse order of $\delta_{d,w}^{(k)}$ is strictly less than $k$ if and only if $$\frac{7-d}{2} \leq w<3\: \mbox{ and }\:\frac{d+2w-3}{2} \leq k \leq d + 
2w - 5
\,.$$
\end{lemma}

\begin{proof}

To evaluate the transverse order of $\delta_{d,w}^{(k)}$, we compute the coefficient of $\nabla_{n}^{k}$. For that  we first examine the leading derivative structure of the operator $\ID_\sigma$. From  Lemma~\ref{ID-tensor}, acting on a weight $w$ tensor $t_{ab} \in \Gamma(\odot_\circ^2 T^* M[w])$ 
 the only terms with non-zero transverse order are  $\nabla_n t_{ab}$, $n_{(a} \nabla \cdot t_{b)\circ}$, $n_c \nabla_{(a}^{} t^c_{b)\circ}$, $s \Delta t_{ab}$, and $s \nabla_{(a} \nabla \cdot t_{b) \circ}$ (here~$\sigma=[g;s
]$ and we work in the scale $g$). In a choice of  coordinates $(s, y^1, \ldots, y^{n-1})$,  we can write
$$\nabla_a = n_a \partial_s + \text{ltots}\,,$$
where $n = \ext s$ and $\text{ltots}$ denotes
terms of lower transverse order.
Thus, we have that 
\begin{align*}
n_c \nabla_{(a}^{} t^c_{b)\circ} &= n_c n_{(a} \partial_s t_{b)\circ}^c + \text{ltots} \; \text{ and } \; 
s \nabla_{(a} \nabla \cdot t_{b) \circ}=
s n_{(a} \partial_s  \nabla \cdot t_{b) \circ} + \text{ltots}\,.
\end{align*}
Because $\sigma$ is an asymptotic unit defining density, we have that $\nabla_n \circ n \eqSig n (\nabla_n + H)$ and so
the operator $\delta_{\rm R}$ composed with the conormal $n_a$ has transverse order zero.
 Therefore, only the terms~$\nabla_n$ and $s \Delta$ in $\ID_{\sigma}$ can contribute to the leading transverse derivatives in the operator~$\delta_{d,w}^{(k)}$.

From the above, and again consulting Lemma~\ref{ID-tensor}, we  conclude that
$$\delta_{d,w}^{(k)} = \otop \circ \nabla_n \circ \prod_{i=0}^{k-2}  \left[ (d+2w - 2i -6)  \nabla_n - s \Delta \right] + \text{ltots}\,.$$
Because $\Delta = \nabla_n^2 + \text{ltots}$, we have that 
$$\delta_{d,w}^{(2)} \eqSig \otop \circ  (d+2w-7) \nabla_n^2 + \text{ltots}\,.$$
We now proceed  inductively to find the general coefficient of the leading transverse derivative term. Suppose that for $k \geq 3$,
$$\delta_{d,w}^{(k-1)} = \otop\circ \left[ \prod_{i=1}^{k-2} (d+2w-k-i-3) \right] \nabla_n^{k-1} + \text{ltots}\,.$$
Then, because
$$\delta_{d,w}^{(k)} = \delta_{d,w-1}^{(k-1)} \circ \ID_\sigma\,,$$
we have that
\begin{equation*}
\nonumber
\delta_{d,w}^{(k)} = 
\otop\circ 
\left[ \prod_{i=1}^{k-2} (d+2w-k-i-5) \right] \nabla_n^{k-1} \circ \left[(d+2w-6) \nabla_n - s \Delta \right] + \text{ltots} \, .\end{equation*}
Hence, as required, we find
\begin{equation}
\delta_{d,w}^{(k)} 
=
\otop\circ 
\left[ \prod_{i=1}^{k-1} (d+2w-k-i-4) \right] \nabla_n^{k} + \text{ltots} \,.\label{finallyIhaveaname}
\end{equation}

To find when $\delta_{d,w}^{(k)}$ has transverse order strictly less than $k$, we study when  the leading coefficient in the above display vanishes. That is, we wish to find $k$ and $w$ that obey
\begin{align} \label{k-ineq}
d+2w-2k-3 \leq 0 \leq d+2w-k-5\,.
\end{align}
Because $k \geq 2$, we find that the right inequality implies that $w \geq \frac{7-d}{2}$. Further, the inequality above can be rewritten in terms of $k$ to obtain $\frac{d+2w-3}{2} \leq k \leq d + 2w - 5$. 
Equation~\ref{k-ineq} has no solutions for $w$ in the range  $w<2-d$ since that would require $d<-3$.
Therefore, when $w < 2-d$, the transverse order of $\delta_{d,w}^{(k)}$ is $k$. 
The left inequality of~\ref{k-ineq} rules out $w>k+2$ so 
the only remaining case is
$k-d<w<3$, 
which, in combination with Equation~\ref{k-ineq}, gives the ranges of $k$ and $w$ quoted in the Lemma.

\end{proof}

An asymptotic unit defining density~$\sigma$ only allows the construction of local invariants with transverse order at most $d-1$, so the highest fundamental form we could
determine from it is
the $d$th fundamental form. 
In view of the above lemma,  we  next focus on fundamental forms built from $\sigma$ of transverse order less than $\frac{d+1}2$.

\begin{definition} \label{ff-lower}
Let $d \geq 3$ and let $2 \leq n < \frac{d+3}{2}$. The \textit{canonical $n$th fundamental form} $\mathring{\underline{\overline{\rm{n}}}}$ is  defined by
$$\mathring{\underline{\overline{\rm{n}}}} := \delta_{d,1}^{(n-2)} \IIo^{\rm e}\, .$$
\end{definition}

\begin{corollary} \label{cff-ff}
The canonical $n$th fundamental form is a fundamental form.
\end{corollary}

\begin{proof}
From Equation~\ref{showmykeys}, in a choice of scale $\sigma=[g;s]$, we have $\IIo^{\rm e}_{ab} = \nabla_{(a} n_{b)\circ} + s \mathring{P}_{ab}$. Therefore, we have that $\nabla_n^k \IIo^{\rm e}_{ab} \eqSig \nabla_n^{k}  \nabla_{(a} n_{b)\circ} + k \nabla_n^{k-1} \mathring{P}_{ab}$ for any positive integer $k$. There exists a scale~$g$ 
for which~$s$ is a unit defining function  for~$\Sigma$ (see, for example~\cite{GrahamLee}); thus in what follows, we can assume that $|ds|^2 = 1$, so that
$\nabla_n n_a = 0$. Therefore, we can write that
\begin{align*}
\nabla_n \nabla_{a} n_{b} &= R_{n ab n} - (\nabla_{a} n^c)(\nabla_{c} n_b) \\
&= W_{nabn}  - P_{ab} + 2n_{(a} P_{b)n} - g_{ab} P_{nn} + \text{ltots}\,.
\end{align*}
Applying the above display to $\nabla_n^k \IIo^{\rm e}_{ab}$, we have
\begin{align*}
\nabla_n^k \IIo^{\rm e}_{ab} &\eqSig \nabla_n^{k-1} W_{nabn} + (k-1) \nabla_n^{k-1} \mathring{P}_{ab} +2 n_{(a} \nabla_n^{k-1} P_{nb)\circ} + \text{ltots}\,.
\end{align*}
To verify that the first two transverse order $k+1$ terms in the above do not cancel in general, 
suppose that $(M,\cc)$ is  conformally flat: For generic $\Sigma$, the metric $g\in \cc$ for which $|ds|^2=1$ has  $P \neq 0$ which  has transverse order $2$. Because the Weyl tensor here vanishes, 
the combination $\nabla_n^{k-1} W_{nabn} + (k-1) \nabla_n^{k-1} \mathring{P}_{ab}$ has transverse order $k+1$; this must hold for general conformally-curved manifolds.

Using Lemma~\ref{tr-deg}, if $n \in \mathbb{Z}$ satisfies $2\leq n < \frac{d+3}{2}$, then 
$$\delta_{d,1}^{(n-2)} = \alpha\,  \otop \circ \nabla_n^{n-2} + \text{ltots}$$
for some non-zero coefficient $\alpha$. Thus, in a choice of scale where $|ds|^2 = 1$, we have
$$\mathring{\underline{\overline{\rm{n}}}}_{ab} = \alpha\,  \otop \circ \nabla_n^{n-2} \IIo^{\rm e}_{ab} + \text{ltots} = \alpha \otop \circ \nabla_n^{n-3}  (W_{nabn} + (n-3) \mathring{P}_{ab}) + \text{ltots}$$
and so $\mathring{\underline{\overline{\rm{n}}}}$ has transverse order $n-1$.
Further, by construction,  $\mathring{\underline{\overline{\rm{n}}}}$ is  a conformal tensor density of weight $3-n$. The corollary follows.
\end{proof}

\begin{remark}
In dimensions
$d=3,4$, the canonical $n$th fundamental form is defined for all~$n \leq d-1$. In $d = 4$ dimensions a well-defined  fourth fundamental form also exists  and was given in Equation~\nn{mynumberisugly}.
When $d > 3$, the operator $\delta_{d,1}^{(1)}$ can be used to compute
\begin{eqnarray*}
\IIIo_{ab} :=
& \delta_{d,1}^{(1)} \IIo_{ab}^{\rm e}
&=
-\IIo^2_{(ab)\circ}
+W_{\hat n ab \hat n}
\, .\\[1mm]
\end{eqnarray*}
For $d > 5$,  applying the operator~$\delta_{d,1}^{(2)}$   gives
\begin{eqnarray*}
 \IVo_{ab} :=
 &\delta_{d,1}^{(2)}  \IIo_{ab}^{\rm e}
 &=-(d-4)(d-5) C_{n(ab)}^\top  - (d-4)(d-5) H W_{\hat n ab \hat n} - (d-4) \bar \nabla^c W_{c(ab) \hat n}^\top \\&&\phantom{=}+ 2 W_{c \hat n \hat n (a}^{} \IIo_{b)\circ}^c + (d^2 - 7d+18) \Fo_{(a} \csdot \IIo_{b)\circ} + (d-6) \bar{W}^c{}_{ab}{}^d \IIo_{cd} \\&&\phantom{=}+ \tfrac{d^3-10d^2+25d-10}{(d-1)(d-2)} K \IIo_{ab}
 \, .
\end{eqnarray*}
The above computation was performed using the symbolic manipulation program FORM~\cite{Jos}. Documentation of our FORM code can be found in~\cite{FormFiles}. The above canonical fundamental forms have leading transverse derivative terms that are identical to the fundamental forms  displayed in the introduction up to an overall non-zero coefficient. Also note that, in dimensions~$d > 3$, the canonical third fundamental form $\IIIo$ recovers the trace-free Fialkow tensor:
\begin{equation}
\label{whysonegative}
\IIIo = -(d-3) \Fo\,.
\end{equation}
\end{remark}

We now turn our attention to conditional fundamental forms for 
$d \geq 5$. 
This relies on hyperumbilicity.
Corollary~\ref{cff-ff} lets us reframe  hyperumbilicity in terms of canonical fundamental forms. In particular, a conformal embedding is hyperumbilic iff $\IIo = \dots = \mathring{\underline{\overline{\rm{k}}}}  = 0$, where $k = \lceil \frac{d+1}{2} \rceil$. We now study consequences of hyperumbilicity.

Our next task is to construct  canonical criteria for conditional  fundamental forms. We begin with technical lemmata.

\begin{lemma} \label{dnIIoe0}
Let $d \geq 3$, $k = \lceil \frac{d+1}{2} \rceil$, and $\Sigma \hookrightarrow (M^d, \cc)$ be a hyperumbilic conformal embedding with corresponding asymptotic unit defining density $\sigma$.
Then, for all $0 \leq \ell \leq k-2$,
$$\nabla_n^{\ell} \IIo^{\rm e} \eqSig 0\,.$$
\end{lemma}

\begin{proof}
We proceed by induction on $\ell$. The $\ell = 0$ case is clear since hyperumbilicity implies that~$\IIo = 0$. Next, suppose that $\nabla_n^{\ell -1} \IIo^{\rm e} \eqSig 0$ for all $1 \leq \ell \leq k-2$. We then compute directly:
\begin{align*}
\nabla_n^{\ell} \IIo_{ab}^{\rm e} \eqSig& (\bar{g}_a^c + n_a n^c)(\bar{g}_b^d + n_b n^d) \nabla_n^\ell \IIo_{cd}^{\rm e} \\
\eqSig& \otop \circ \nabla_n^\ell \IIo_{ab}^{\rm e} + \frac{ \bar{g}_{ab}}{d-1} \bar{g}^{cd} \nabla_n^\ell \IIo_{cd}^{\rm e} + 2 n_{(a} (n \csdot \nabla_n^\ell \IIo_{b)}^{\rm e})^\top + n_a n_b \,n \csdot (\nabla_n^\ell \IIo^{\rm e}) \csdot n \\
\eqSig& \otop \circ \nabla_n^\ell \IIo_{ab}^{\rm e} + 2 n_{(a} (n \csdot \nabla_n^\ell \IIo_{b)}^{\rm e})^\top + (n_a n_b - \tfrac{ \bar{g}_{ab}}{d-1})  \hh n \csdot (\nabla_n^\ell \IIo^{\rm e}) \csdot n\,,
\end{align*}
where the third line holds because $g^{ab} \IIo^{\rm e}_{ab} = 0$. By hyperumbilicity, we have that $\mathring{\underline{\overline{\rm{\ell\!-\!2}}}}  = \delta_{d,1}^{(\ell)} \IIo^{\rm e} = 0$ for $0 \leq \ell \leq k-2$. Consulting Equation~\ref{finallyIhaveaname},  we see that $\otop\circ \nabla_n^{\ell} \IIo^{\rm e}|_{\Sigma} = 0$ for all $0 \leq \ell \leq k-2$ (the lower order terms in Equation~\ref{finallyIhaveaname} also vanish by hyperumbilicity). Thus, it only remains to show  that $n^a 
\nabla_n^\ell \IIo_{ab}^{\rm e} \eqSig 0$.

By the inductive assumption and Leibniz rule, we also have that
$$n^a \nabla_n^\ell \IIo_{ab}^{\rm e} \eqSig \nabla_n^\ell (n^a \IIo_{ab}^{\rm e})\,.$$
Using the tractor identity  $I^A P_{AB} = \frac{K X_B}{d-2}$ and that $P = q(\IIo^{\rm e})$, we have  $n^a \IIo_{ab}^{\rm e} = \frac{s \nabla \cdot \IIo^{\rm e}}{d-1}$ where $\sigma$ is $s$ in the scale $g$. Applying this identity to the above display, we have that
\begin{align*}
n^a \nabla_n^\ell \IIo_{ab}^{\rm e} \eqSig& \tfrac{1}{d-1} \nabla_n^{\ell} (s \nabla^a \IIo_{ab}^{\rm e}) \\
\eqSig& \tfrac{\ell}{d-1} \nabla_n^{\ell -1} \nabla^a \IIo_{ab}^{\rm e} \\
\eqSig& \tfrac{\ell}{d-1} \nabla^a \nabla_n^{\ell -1} \IIo_{ab}^{\rm e} \\
\eqSig& \tfrac{\ell}{d-1} \big(\nabla^{\top \!\hh a}\hh \nabla_n^{\ell -1} \IIo_{ab}^{\rm e} + n^a \nabla_n^\ell \IIo_{ab}^{\rm e} \big) \\
\eqSig& \tfrac{\ell}{d-1} n^a \nabla_n^\ell \IIo_{ab}^{\rm e}\,,
\end{align*}
where the second, third,  and fifth lines follow from the inductive assumption. Thus, we have that $n^a \nabla_n^\ell \IIo_{ab}^{\rm e} \eqSig 0$ so long as $\ell \neq d-1$, but $\ell \leq k-2 < d-1$, so the lemma follows.

\end{proof}

\begin{lemma} \label{PsQ}
Let $d \geq 5$, $k = \lceil \frac{d+1}{2} \rceil$, and $\Sigma \hookrightarrow (M^d, \cc)$ be a hyperumbilic conformal embedding with corresponding asymptotic unit defining density $\sigma$.
Then,
$$P_{AB} = \sigma^{k-1} Q_{AB} + \sigma^{k-2} X_{(A} T_{B)} + \sigma^{k-3} X_A X_B U\,,$$
where $Q \in \Gamma(\odot^2_\circ \mathcal{T} M [-k]) \cap \ker X$, $T \in \Gamma(\mathcal{T} M[-k]) \cap \ker X$, and $U \in \Gamma(\ce M [-k])$.
\end{lemma}
\begin{proof}
As usual $P_{AB}$ denotes $ \hd_A \hd_B \sigma$.
From Lemma~\ref{dnIIoe0} we have that $\nabla_n^{m} \IIo^{\rm e}|_{\Sigma} = 0$ for all $0 \leq m \leq k-2$, so that $\IIo^{\rm e} = \sigma^{k-1} Q_{ab}$ for some $Q \in \Gamma(\odot^2_{\circ} T^* M[2-k])$.
Because $P = q(\IIo^{\rm e}) = q(\sigma^{k-1} Q)$, of interest is the difference $q(\sigma^{k-1} Q) - \sigma^{k-1} q(Q)=[q,\sigma^{k-1}]Q$. Note, because $k \leq \frac{d}{2} + 1$, $q(Q)$ is defined for $d \geq 5$, and from Lemma~\ref{insertions}) it follows that, for any integer $\ell\geq2$,
\begin{align} \label{[q,s^k]}
q(\sigma^\ell Q)_{AB}-\sigma^\ell q(Q)_{AB}
= \sigma^{\ell -1} X_{(A} Z^b_{B)}  t_b+ X_A X_B \mathcal{O}(\sigma^{\ell -2})\,,
\end{align}
for some $t_b \in \Gamma(T^*M[1-k])$. 
Clearly $\sigma^{\ell-1} Z^b t_b = \sigma^{\ell-1} q(t) + X \mathcal{O}(\sigma^{\ell -1})$, so setting $\ell = k-1$,
Equation~\ref{[q,s^k]}
becomes
$$q(\sigma^{k-1} Q)_{AB} - \sigma^{k-1} q(Q)_{AB} = \sigma^{k-2} X_{(A} q(t)_{B)} + \sigma^{k-3} X_A X_B U\,,$$
for some $U \in \Gamma(\ce M[-k])$. Defining $T_A := q(t)_A \in \Gamma(\mathcal{T} M[-k])$ and $Q_{AB} := q(Q)_{AB} \in \Gamma(\odot^2_\circ \mathcal{T} M [-k])$, the lemma follows.
\end{proof}

\begin{remark}\label{Q-td}
For hyperumbilic embeddings, $\otop (\nabla_n^m \IIo^{\rm e})= 0$ for all $0 \leq m \leq k -2$ (where $k = \lceil \frac{d+1}{2} \rceil$).  Also, as shown in the proof of Lemma~\ref{cff-ff}, the quantity $\otop (\nabla_n^{k-1} \IIo^{\rm e})$ has transverse order~$k$, so it follows from the above proof that  the transverse order of $Q_{ab}|_{\Sigma}$ is $k$.
\end{remark}

\medskip

Armed with the above technical lemma, as well as Lemma~\ref{Dlog}, we can  construct a set of canonical conditional fundamental forms.

\begin{definition}
Let $d\geq5$ and $\Sigma \hookrightarrow (M^d, \cc)$ be a conformal embedding with corresponding asymptotic unit defining density $\sigma$. Further, let $\tau \in \ce M[1]$ be any true scale. Then, for $\frac{d+3}{2} \leq n \leq d-1$, define the \textit{$n$th canonical conditional  fundamental form} by
$$\mathring{\underline{\overline{\rm{n}}}} := \bar{q}^* \circ \bar{r} \circ \otop \circ \left( I \csdot D^{n-2} (P \log \tau) - \log \tau I \csdot D^{n-2} P \right)\in \Gamma(\odot^2_\circ T^*\Sigma[3-n])\,.$$
\end{definition}
\begin{remark}
It follows that the expression $I \csdot D^{n-2} P \log \tau - \log \tau\hh I \csdot D^{n-2} P$ in the above definition is a tractor by repeated application of Lemma~\ref{Dlog}.
\end{remark}

\begin{proposition} \label{cond-ffs}
Let $\Sigma \hookrightarrow (M^d, \cc)$ with $d \geq 5$ be a hyperumbilic conformal embedding and let $k := \lceil \frac{d+1}{2} \rceil$. Then, for all $k+1 \leq n \leq d-1$, the $n$th canonical conditional fundamental form is a fundamental form.
\end{proposition}
\begin{proof}
We begin by showing that the transverse order of the $n$th canonical conditional fundamental form is $n-1$. Recycling  the computation in the  proof of Lemma~\ref{tr-deg}, 
but promoting the weight $w$ to an operator (as necessary to  act on log densities),  
we find that
$$ I \csdot D^{n-2} \stackrel\Sigma=  \nabla_n^{n-2} \circ \left[ \prod_{i=1}^{n-2} (d+2 \underline{w} -n-i+2) \right] + \text{ltots} \, ,$$
and in turn
$$ I \csdot D^{n-2} \circ P_{AB} \stackrel\Sigma= \nabla_n^{n-2} \circ P_{AB} \circ \left[ \prod_{i=1}^{n-2} (d+2 \underline{w} - n-i) \right] + \text{ltots} \circ P_{AB}\,.$$
Because $k+1 \leq n \leq d-1$, we have that $d-n-(n-2) = d-2n+2 \leq 0$ and $d-n-1 \geq 0$. Therefore the product above has $\underline{w}$ as one of its factors,  so we have that
$$\prod_{i=1}^{n-2} (d+2 \underline{w} - n-i) = \alpha \underline{w} + \mathcal{O}(\underline{w}^2)\,,$$
for some $\alpha \neq 0$. So, remembering that $\otop$ includes restriction to $\Sigma$ and $\underline{w}\log \tau=1$ while $\underline{w}^2 \log\tau=0$, we can write
$$\otop \circ I \csdot D^{n-2} (P_{AB} \log \tau) = \alpha \otop(\nabla_n^{n-2} P_{AB}) + \text{ltots} (P_{AB} \log \tau)\,.$$
Further, again using $\underline{w} \;1 = 0$, we have that $\log \tau\,  I \csdot D^{n-2} P_{AB} = \log\tau\, \text{ltots} (P_{AB})$. Note that~$P = q(\IIo^{\rm e})$ and, from the proof of Corollary~\ref{cff-ff},  we have $\otop (\nabla_n^m \IIo^{\rm e})$ has transverse order~$m+1$. Hence, because 
$Z^A_a Z^B_b (\otop  \circ I \hh\csdot \hh  D^{n-2})
(P_{AB})$ has transverse degree $n-1$ and the hypersurface-intrinsic operator~$\bar{q}^* \circ \bar{r}$ acting on $(\otop \circ I\hh\csdot \hh D^{n-2})(
P_{AB})$  cannot change its transverse order, we conclude that the $n$th canonical conditional fundamental form has transverse order $n-1$.

\smallskip
\color{black}
Finally
 we need to show that
 the $n$th canonical conditional fundamental form is independent of~$\tau$. Suppose that $\Sigma$ is embedded hyperumbilically and 
let $\ell := n-k-1$. Note that $0 \leq \ell \leq d-k-2$.
Then, from Lemma~\ref{PsQ}, we have that
\begin{equation*}
\mathring{\underline{\overline{\rm{n}}}} := (\bar{q}^* \circ \bar{r} \circ \otop)
(\Pi_n)
\,,
\end{equation*}
where
\begin{multline*}
\Pi_n := I \csdot D^{\ell + (k - 1)}\big([\sigma^{k-1} Q_{AB}  + \sigma^{k-2} X_{(A} T_{B)} + \sigma^{k-3} X_A X_B U ] \log \tau \big)\\- \log \tau \,  I \csdot D^{\ell + (k - 1)} \big(\sigma^{k-1} Q_{AB}  + \sigma^{k-2} X_{(A} T_{B)} + \sigma^{k-3} X_A X_B U \big) \,.
\end{multline*}
Employing a quadratic Casimir of the $\frak{sl}(2)$ algebra, $$4yx + 2{\sf h} + {\sf h}^2=4xy-2{\sf h}+{\sf h}^2 \eqSig {\sf h}({\sf h}-2)\, ,$$ we find the enveloping algebra recursion relation
$$y^{\ell + m + 1} x^{m+1} \eqSig - y^{\ell + m} x^{m} (\ell + m + 1) ({\sf h} + m - \ell)\,,$$
which can be solved to yield (for any non-negative integer $m$)
\begin{align} \label{solved-sl2}
y^{\ell+m} x^m \eqSig (-1)^m y^\ell \prod_{i=1}^m (\ell + i)({\sf h}-\ell + i - 1)\,.
\end{align}
Note that when $m = 0$ in the above display, our convention is to define the product to be $1$.
In the $\frak{sl}(2)$
notations of Lemma~\ref{sl2-algebra}
we now have
\begin{alignat*}{2}
\Pi_n =\;&&- (-1)^{n}\log \tau&  \Big[ y^{\ell + (k-1)} x^{k-1} Q_{AB} \\
&& &+ y^{\ell + 1 + (k-2)} x^{k-2} (X_{(A} T_{B)})
+ y^{\ell + 2 + (k-3)} x^{k-3} (X_A X_B U) \Big] \\
 &&+(-1)^{n} & \Big[ y^{\ell + (k-1)} x^{k-1} (Q_{AB} \log \tau) \\
 &&&+ y^{\ell + 1 + (k-2)} x^{k-2} (X_{(A} T_{B)} \log \tau) 
+ y^{\ell + 2 + (k-3)} x^{k-3} (X_A X_B U \log \tau) \Big]\,.
\end{alignat*}

We define the 
 polynomials 
 \begin{equation}
 \label{F}
 F_{\ell,w,i}(u) := (\ell + i)(u + 2w - \ell + i - 1)\, ,
 \end{equation} which obey
$F_{\ell+1,w+1,i}(u) = F_{\ell, w, i+1}(u)$.
Then, using Equation~\ref{solved-sl2} and the fact that $Q$, $X \odot T$ and $X^2 \,U$ have weights $-k$, $1-k$, and $2-k$, respectively, we have that:
\begin{equation*} 
\begin{aligned}
y^{\ell + (k-1)} x^{k-1} \circ Q_{AB} \eqSig\;& (-1)^{k-1} y^{\ell} \circ Q_{AB} \circ \prod_{i=1}^{k-1} F_{\ell, -k, i}({\sf h})\,, \\
y^{\ell +1 + (k-2)} x^{k-2} \circ X_{(A} T_{B)} \eqSig\;& (-1)^{k-2} y^{\ell+1} \circ  X_{(A} T_{B)} \circ \prod_{i=1}^{k-2} F_{\ell + 1,1-k,i}({\sf h}) \\ \eqSig
& (-1)^{k-2} y^{\ell+1} \circ  X_{(A} T_{B)} \circ \prod_{i=2}^{k-1} F_{\ell,-k,i}({\sf h})\,, \\
y^{\ell+2 + (k-3)} x^{k-3} \circ X_A X_B U \eqSig\;& (-1)^{k-3} y^{\ell+2} \circ X_A X_B U \circ \prod_{i=1}^{k-3} F_{\ell + 2, 2-k,i}({\sf h}) \\ \eqSig
& (-1)^{k-1} y^{\ell+2} \circ X_A X_B U \circ \prod_{i=3}^{k-1}F_{\ell,-k,i}({\sf h})\,.
\end{aligned}
\end{equation*}

\noindent
Defining the  polynomial $$f_j(u) := 
\prod_{i=j}^{k-1} (\ell + i)(u - 2k - \ell + i - 1)=
\prod_{i=j}^{k-1}F_{\ell,-k,i}(u)\, ,$$
and remembering that ${\sf h}\hh (1)=d$,  we may rewrite $\Pi_n$:

\begin{align}
\begin{alignedat}{2} 
 \Pi_n   = &&- (-1)^{n+k-1} \log \tau &\Big \{ f_1(d)\hh  y^\ell Q_{AB}\\
&&& - f_2(d)\hh  y^{\ell + 1} (X_{(A} T_{B)})  + f_3(d) y^{\ell + 2} (X_A X_B U) \Big\} \\
&&+(-1)^{n+k-1} &\Big\{ y^\ell \big[Q_{AB} f_1({\sf h}) (\log \tau) \big] \\ 
&&&-y^{\ell + 1} \big[X_{(A} T_{B)} f_2({\sf h}) (\log \tau) \big] + y^{\ell + 2}\big[ X_A X_B U f_3({\sf h}) (\log \tau) \big] \Big\} \,.
\end{alignedat} \label{Pi-n}
\end{align}

Note that
\begin{align}\label{top-Fs}
f_2(u) =  F_{\ell,-k,2}(u) f_3(u) \qquad \text{and} \qquad f_1(u) = F_{\ell, -k,1}(u) F_{\ell,-k,2}(u) f_3(u) \,.
\end{align}
Now, for any polynomial $f$ for which $d$ is a root, the operator $f({\sf h})=f(d+2\underline{w})$ obeys
$$
f({\sf h}) = 2 f'(d)\hh \underline{w} + \mathcal{O}(\underline{w}^2)\, ,
$$
and so
$$f({\sf h}) \log \tau = 2 f'(d)\, .$$ Thus, if $f_3(d) = 0$ then $f_1(d)=0=f_2(d)$ and moreover  $f_1({\sf h}) \log \tau$, $f_2({\sf h}) \log \tau$, and $f_3({\sf h}) \log \tau$ are independent of $\tau$ and hence $\Pi_n$ would be independent of $\tau$.

\medskip

To establish $\tau$ independence, we analyze three cases. In the first two cases, we determine for which choices of $n$ we have that $f_3(d) = 0$, the result of which depends on dimension parity. In the third case, we handle all of the choices of $n$ excluded from the analysis in the first two cases. Note that, because $d \geq 5$, $k:=\lceil\frac{d+1}2\rceil$, and the lemma is only concerned with $n$ satisfying $k+1 \leq n \leq d-1$, we have that $k \geq 3$ so $n \geq 4$.
Thus, when $k=3$ we must have $d=5 $ and~$n=4$. We begin with   the first, even dimension parity, case.

\smallskip

\noindent
\textbf{Case 1:} If $k \geq 4$ and $d \geq 6$ is even, then $\mathring{\underline{\overline{\rm{n}}}}$ is independent of $\tau$ when $k+1 \leq n \leq d-1$:

\medskip
 To see this first observe that
the polynomial $f_3$  satisfies $f_3(d) = 0$ when
$$d - 2k - \ell + 2 \leq 0 \leq d - k - \ell -2\,.$$
Using that $\ell:=n-k-1$, the above display implies that $ d-k+3 \leq n$. From the hypothesis, we only consider $n \leq d-1$, so we have that $k \geq 4$. Rewriting the inequality above as a condition on~$n$,  we have
\begin{align} \label{n-constraint}
d-k+3
 \leq n \leq d-1\,.
\end{align}
Using that $d$ is even we have $d = 2k-2$, and so Equation~\ref{n-constraint} becomes
$$k+1 \leq n \leq d-1\,,$$
as required.

\medskip

\noindent
\textbf{Case 2:} If $k \geq 4$ and $d \geq 7$ is odd, then $\mathring{\underline{\overline{\rm{n}}}}$ is independent of $\tau$ when $k+2 \leq n \leq d-1$:

\medskip

Because $d$ is odd, $d = 2k-1$. Then, Equation~\ref{n-constraint} becomes
$$k+2 \leq n \leq d-1\,,$$
as required.

\medskip

So far we have  dealt neither with the case $k=3$ nor the case $n=k+1\geq 5$ and $d$ odd. These are both encompassed by studying $n=k+1\geq 4$ and $d$ odd.

\smallskip

\noindent
\textbf{Case 3:} If $n = k+1$ and $d$ is odd, then $\mathring{\underline{\overline{\rm{n}}}}$ is independent of $\tau$:

\medskip

Because $d$ is odd, we have that $d = 2k-1$. Because $n = k+1$, we have that $\ell = 0$. Moreover (see  Equation~\nn{F})
$$
F_{0,-k,2}(d) = 0
\mbox{ and }
F_{0,-k,2}(d+2\underline{w}) = 4 \underline{w}
\, .$$
Thus, from Equations~\ref{Pi-n} and~\ref{top-Fs} and remembering that $\underline{w}\log\tau=1$,  we have that the only terms in $\Pi_n$ that depend on $\tau$ are the terms proportional to $ y^2 X_A X_B U$. Because $y^2 \propto I \csdot \hd^2$, we can
employ Equation~\nn{usemeoft} and Proposition~\ref{leib-failure}
 to see that
$$y^2 \circ X_A X_B  \eqSig {\mathcal E}(X) + {\mathcal E}(I) + {\mathcal E}(h)\,.$$
But, because $\mathring{\underline{\overline{\rm{n}}}} = \bar{q}^* \circ \bar{r} \circ \otop (\Pi_n)$ and
 $$(\bar{q}^* \circ \bar{r} \circ \otop)(t_{(ab)} Z^a_A Z^b_B+\alpha X_{(A} V_{B)} + \beta I_{(A} V_{B)}' + \gamma h_{AB} S) 
=(\bar{q}^* \circ \bar{r} \circ \otop)(t_{(ab)} Z^a_A Z^b_B)\, ,
$$ for any non-vanishing $t_{ab}\in 
\Gamma(\odot^2 T^*M[w+2])
$, $V \in \Gamma(\ct M[w-1])$, $V' \in \Gamma(\ct M[w])$, and $S \in \Gamma(\ce M[w])$ for generic $w$,
 as well as any coefficients $\alpha,\beta,\gamma$,
  we have that $\mathring{\underline{\overline{\rm{n}}}}$ is independent of~$\tau$ when $n = k+1$ and $d$ odd.

\bigskip

The above three cases prove that, for all $d \geq 5$ and $n$ satisfying $k+1 \leq n \leq d-1$, the $n$th canonical conditional fundamental form $\mathring{\underline{\overline{\rm{n}}}}$ is independent of $\tau$. 
\end{proof}

We have already seen that a conformal embedding $\Sigma \hookrightarrow (M^d,\cc)$ is hyperumbilic when all its canonical fundamental forms vanish. The above proposition in fact establishes that 
a hyperumbilic conformal embedding $\Sigma \hookrightarrow (M^d,\cc)$ is \"uberumbilic when all its canonical conditional fundamental forms vanish.

\begin{example} We can use Proposition~\ref{cond-ffs}
to 
compute $\IVo$ in $d = 5$ for hyperumbilic embeddings (so $\IIo = \IIIo = 0$) and find
$$\IVo_{ab} 
\stackrel{\Sigma_{\rm hyp}}= 
2 C_{\hat n (ab)}^\top\, .
$$
This computation relies on 
 the expression for $I \csdot D^2$ in terms of the tractor connection and weight operators, which can be found in~\cite{Will2}. 
\end{example}

\begin{remark} \label{conj-remark}
Notice that the above fourth canonical conditional fundamental form and the fourth conditional fundamental form provided in Equation~\nn{mynumberisveryugly} of the introduction 
differ because they have  differing 
 invariance criteria; 
 the latter requires only umbilicity.
\end{remark}

We require a corollary of Lemma~\ref{dnIIoe0} before we prove our main result:
\begin{corollary} \label{dnIIoe0-2}
Let $d \geq 3$ and $\Sigma \hookrightarrow (M^d, \cc)$ be an \"uberumbilic conformal embedding with corresponding asymptotic unit defining density $\sigma$.
Then, for all $0 \leq \ell \leq d-3$,
$$\nabla_n^{\ell} \IIo^{\rm e} \eqSig 0\,.$$
\end{corollary}
\begin{proof}
In $d = 3$ dimensions, \"uberumbilicity is equivalent to umbilicity and thus the corollary holds trivially. When $d = 4$, \"uberumbilicity is equivalent to hyperumbilicity, and thus  the result follows from Lemma~\ref{dnIIoe0}. In dimensions $d \geq 5$, the proof follows that of Lemma~\ref{dnIIoe0} {\it mutatis mutandis}.
\end{proof}

We can now prove our main result.

\begin{proof}[Proof of Theorem~\ref{uber-PE}]
In the interior of $M$, the trace--free Schouten tensor of the singular metric $g^o=s^{-2}g$, determined by the conformal density $\sigma=[g;s]$, obeys~\cite{BEG}
$$
\sigma\mathring P^{g^o}= q^*(\nabla I_\sigma)\, ,
$$
and this condition extends to the boundary $\Sigma$.
Thus the 
asymptotic Poincar\'e--Einstein Condition~\nn{aPE} holds exactly when
$$
q^*(\nabla I_\sigma)={\mathcal O}(\sigma^{d-2})\, .
$$ 
 Because $\IIo^{\rm e} = q^*(\nabla I)$,
it only remains to establish that 
$\Sigma$ is \"uberumbilic iff $\IIo^{\rm e} = \sigma^{d-2} T$ 
for
some~$T \in \Gamma(\odot^2_{\circ} T^* M[3-d])$. 
Clearly  
if $\IIo^{\rm e} = \sigma^{d-2} T$, then $\Sigma$ is \"uberumbilic. The converse follows from Corollary~\ref{dnIIoe0-2}.
\end{proof}

Remark~\ref{conj-remark}
provides an example of a conditional higher fundamental forms that exists  when  
the hyperumbilicity condition
is weakened.
 Moreover, both the fourth and fifth fundamental forms given in the introduction exist in all even dimensions $d \geq 4$ 
 without
the supposition of any conditions whatsoever. 
 The dimension parity dichotomy between Cases 1 and 2 in the proof of Proposition~\ref{cond-ffs}
 also suggests that even dimensions may be more amenable for constructing invariant fundamental forms. Rather than searching for higher fundamental forms (with no invariance conditions) in even dimensions, one can consider 
weakened conditions in any dimension.
 In fact, when $d = 5$, 
 a conditional fourth fundamental form exists if the embedding is umbilic, while when $d = 7$, a conditional fifth fundamental form exists if both $\IIo$ and $\IIIo$ vanish. 
Existence results for higher conditional forms can be established for 
  weakened analogs of   
 the hyperumbilicity condition  so long as the following hypothesis holds:

\begin{hyp} \label{conj-PsQ}
Let 
$\Sigma \hookrightarrow (M^d, \cc)$ be a conformal embedding with corresponding asymptotic unit defining density $\sigma$ 
such that 
$$P_{AB} = \sigma^{k-1} Q_{AB}\,,$$
where $Q \in \Gamma(\odot^2 \mathcal{T} M [-k]) \cap \ker X$ and the integer $k$ obeys $1\leq k\leq d-1$.
\end{hyp}

\begin{proposition}
Let $d\geq 5$ and $\Sigma \hookrightarrow (M^d, \cc)$ be a conformal embedding 
subject to Hypothesis~\ref{conj-PsQ}
for some $k$.
Then for all $d-k+1 \leq n \leq d-1$, the $n$th canonical conditional fundamental form is a fundamental form.
\end{proposition}
\begin{proof}
The proof  follows that of Proposition~\ref{cond-ffs}
{\it mutatis mutandis}.
\end{proof}

\begin{remark} 
Hypothesis~\ref{conj-PsQ}
 is a significant strengthening of Lemma~\ref{PsQ} but, for $k \leq 3$, is provably true by direct computation so long as $\IIo = \dots = \mathring{\underline{\overline{\rm{k}}}} = 0$.
\end{remark}

From the statement of this proposition,
when $d = 5$, we see that to construct $\IVo$, we only need $k = 2$ and in turn need only require that $\IIo = 0$, as stated above. Similarly, for $d = 7$, we see that to construct~$\Vo$, we need only that $k = 3$ and so we only require that $\IIo = \IIIo = 0$.

\subsection{Tractor Fundamental Forms}

\color{black}
In this section, we give  various
identities relating tractor fundamental forms and bulk tractors. 
These are particularly useful 
for any holographic computation involving extrinsic conformal embedding data.

A canonical tractor $n$th fundamental form is given by $\bar{q}(\mathring{\underline{\overline{\rm{n}}}})$ in dimensions $d > n+1$. 
The tractor second fundamental form $L:=\bar q(\IIo)$ is defined for $d>3$, and its holographic formula was
given in Equation~\nn{L-P-formula}. A holographic formula for
the canonical tractor third fundamental form, valid when $7\neq d>5$, is
\begin{equation}\label{Iwasdemoted}
\bar{q}(\IIIo) =  \dot{P}_{AB}^t - \tfrac{2}{(d-3)(d-5)} X_{(A} \bar D \csdot \dot{P}^t_{B)} + \tfrac{1}{(d-3)(d-4)(d-5)} X_A X_B \bar D \csdot \hdb \csdot \dot{P}^t\, .
\end{equation}
Here $\dot{P}_{AB} := I \csdot \hd P_{AB}$ and $\dot{P}_{AB}^t:=\bar{r} \circ \otop (\dot{P}_{AB})$;
the above result is a direct application of Lemma~\ref{D-free-tractor} and the definition of $\IIIo$ in  
 Definition~\ref{ff-lower}.
Just as the Fialkow tensor is related to the canonical third fundamental form by a factor $-(d-3)$, see Equation~\nn{whysonegative},  the Fialkow \textit{tractor} and the canonical tractor third fundamental form obey 
$$\bar{q}(\IIIo)_{AB} = -(d-3) F_{AB}\, .$$ 
The above relationship between $F$ and $\bar{q}(\IIIo)$ and Equation~\nn{Iwasdemoted}
 yield a corollary to Corollary~\ref{Fialkow-tractor1}
that gives an analog of the Fialkow--Gau\ss\ equation~\nn{Aaron}:
\begin{corollary}[Fialkow--Gau\ss--Thomas Equation] \label{Fialkow-tractor2}
Let $7 \neq d > 5$ and $\sigma$ be an asymptotic unit conformal defining
density   for $\Sigma$. Then,
\begin{multline*}
\left( L^C_A L_{CB} - \tfrac{1}{d-1} K \overline{h}_{AB} - W_{NABN} \right) - \tfrac{1}{d-1} X_{(A} \hdb_{B)} K - \tfrac{1}{(d-4)(d-5)} X_A X_B U \\
= -  \dot{P}_{AB}^t + \tfrac{2}{(d-3)(d-5)} X_{(A} \bar D \csdot \dot{P}^t_{B)} - \tfrac{1}{(d-3)(d-4)(d-5)} X_A X_B \bar D \csdot \hdb \csdot \dot{P}^t 
=(d-3) F_{AB}
\,,
\end{multline*}
where $U\in \Gamma(\ce \Sigma[-4])$ is the density given in Equation~\nn{dasboot}.
\end{corollary}

\medskip

\noindent

The rigidity density $K = \IIo^2$ is an important conformal hypersurface invariant that measures the difference between ambient and hypersurface scalar curvatures:
$$
K\eqSig
2(d-2)\big(
J -P(\hat n,\hat n) - \bar J + \tfrac{d-1}2 H^2\big)\, .
$$
The above follows by tracing the Fialkow--Gau\ss\ equation~\nn{Aaron}. Since we have canonical extension $\IIo^{\rm e}$ of $\IIo$, the same follows for $K$. There are many situations in which normal derivatives of~$K^{\rm e}$ are needed.
\begin{definition}
Let $K^{\rm e} := P_{AB} P^{AB}$ as in the proof of Lemma~\ref{thisismylabel}, and 
\begin{eqnarray*}
\Kd &:=& \delta_R K^{\rm e} \in \Gamma(\ce \Sigma[-3])\,,
\\
 \Kdd&:=& \delta_R \hh I \csdot \hd K^{\rm e} \in \Gamma(\ce \Sigma[-4])\,, \qquad 
 d\neq 6\,, \\
  \dddot K&:=& \delta_R \hh I \csdot \hd^2 K^{\rm e} \in \Gamma(\ce \Sigma[-5])\,, \qquad 
 d\neq 6,8\,, \\[-1mm]
 &\vdots&
 \end{eqnarray*}
\end{definition}

\noindent
 In particular, because $P_{AB} P^{AB} = (\IIo^{\rm e})^2$, there are (rather useful) formul\ae\ for  $\Kd,\Kdd$, and $\dddot K$ in terms of the fundamental forms $\IIo$, $\IIIo$, $\IVo$, $\Vo$, and hypersurface derivatives thereof.
 For this we 
introduce the following notational device. \begin{definition}\label{DefconnII}
Let
$$\Pi_{(2)} := \otop \circ q(\IIo^{\rm e})
\in \Gamma(\odot^2_\circ \ct \Sigma[-1])
\,,$$
and, for $3 \leq m < d$ such that $m \not \in \{\frac{d+1}{2}, \frac{d+3}{2}, \frac{d+5}{2}\}$, let
$$\Pi_{(m)} := (\bar{r} \circ \otop \circ \delta_{\rm R} \circ q) \circ \ID^{m-3}_{\sigma} (\IIo^{\rm e})
\in \Gamma(\odot^2_\circ \ct \Sigma[1-m])
\,.$$
When $3 \leq m < d$ 
we define
$$\tilde{\Pi}_{(m)} := \bar{q}(\mathring{\overline{\underline m}})\,.$$
\end{definition}

\begin{remark}
The values $\{\frac{d+1}{2}, \frac{d+3}{2}, \frac{d+5}{2}\}$ are treated on a separate footing in Definition~\ref{DefconnII} for reasons of definedness of the operator $\bar r$; see Lemma~\ref{Pantheon}.
Also, in dimensions $d$ such that $m \not \in \{\frac{d+1}{2}, \frac{d+3}{2}, \frac{d+5}{2}\}$, by construction we have that $\tilde{\Pi}_{(m)} = \Pi_{(m)} + {\mathcal E}(X)$, since $\tilde \Pi_{(m)}=(\bar q\circ  \bar q^*)( \Pi_{(m)})$ and 
 Lemma~\ref{D-free-tractor} says
$\bar q\circ  \bar q^*=\operatorname{Id} +  {\mathcal E}(X)$.
So, when $m \in \{\frac{d+1}{2}, \frac{d+3}{2}, \frac{d+5}{2}\}$, if $T \in  \Gamma(\odot^2 \ct M[w]) \cap \ker X_\lrcorner$, we may define $\Pi_{(m)}^{AB} T_{AB} := \tilde{\Pi}_{(m)}^{AB} T_{AB}$.
\end{remark}

By construction, the rank two, trace-free hypersurface tractors $\Pi_{(m)}$  produce the corresponding fundamental form $\mathring{\overline{\underline m}}$  upon acting by the  extraction map $\bar{q}^*$. In general, if $ \bar{q}^*(T^{AB})=t_{ab}$ and $ \bar{q}^*(U^{AB})=u_{ab} $, we have that $t_{ab} u^{ab} = T_{AB} U^{AB}$. For this reason, 
the tractors~$\Pi_{(m)}$ can be used to 
  compute holographic formul\ae\ for scalars built from contractions of fundamental forms.
  These formul\ae\ are simpler than those for their constituent fundamental forms and are therefore particularly useful for computations of scalar densities, such as integrands for Willmore-like energies (see for example~\cite{Will2,GR}). We now give two such results.

\begin{lemma} \label{IIIo2}
Let $d>4$, then the square of the third fundamental form has a holographic formula given by
$$\IIIo^2 = \dot{P}_{AB} \dot{P}^{AB}\big|_\Sigma - \tfrac{3d-2}{(d-1)(d-2)^2} K^2 \,.$$
\end{lemma}

\begin{proof}
It follows from Definition~\ref{DefconnII} that  $\IIIo^2 = \Pi_{(3)\hh AB} \Pi_{(3)}^{AB}$. Because the only appearance of $\Pi_{(3)}$ in this proof is when it is squared, any instances where $\tilde{\Pi}_{(3)}$ would be required can be replaced with $\Pi_{(3)}$. So, the proof amounts to relating $\dot{P}$ to $ \Pi_{(3)}$. As previously noted, $q(\IIo^{\rm e}) = P$, so $ \Pi_{(3)}^{AB} = \bar{r} \circ \otop (\dot{P})$. In order to relate $ \Pi_{(3)}$ to $\dot{P}|_{\Sigma}$ explicitly, following Lemma~\ref{Pantheon} specialized to hypersurface tractors, we need  to rewrite $X_A\dot{P}^{AB}$:
\begin{align}
\label{XPdot}
\begin{split}
X_A \dot{P}^{AB} =&\phantom{-} \:I \csdot \hd X_A P^{AB} - I_A P^{AB} + \tfrac{2 \sigma}{d-2} \hd_A P^{AB} \\
=& -I_A \hd_B I^A \\
=& -\big(\tfrac{1}{2} \hd_B I^2 + \tfrac{X_B}{d-2} (\hd I)^2 \big)\\
=& -\frac{K X_B}{d-2}+{\mathcal O}(\sigma^{d-2})
\,,
\end{split}
\end{align}
where the first and third lines are results of Proposition~\ref{leib-failure}, the second line results from the properties of $P$, and the last line uses the definition of $K$. 
Further, observe that
via Proposition~\ref{leib-failure}, we have
\begin{align}\label{starsare}
I \csdot \dot{P}^B = \tfrac{1}{d-2}\big(\dot{K} X^B + K I^B\big) - \tfrac{2 \sigma }{d-4}\big(\tfrac1{d-2}\hd^B K -P^{AC} \hd_C P_A^B\big)+{\mathcal O}(\sigma^{d-3})
\eqSig
\tfrac1{d-2}(\Kd X+ KN)
\, .
\end{align}
Using the above identities,  the definition of $\bar{r}$ and~$\otop$, as well as the standard operator identity for~$\hdb \circ X$, 
a tedious calculation along $\Sigma$ yields \color{black}
\begin{align*}
 \Pi_{(3)\hh AB} \eqSig \dot{P}_{AB} - \tfrac{d}{(d-1)(d-2)} X_{(A} \hdb_{B)}K - \tfrac{2}{d-2} \Kd I_{(A} X_{B)} - \tfrac{1}{d-2} K I_A I_B - \tfrac{1}{(d-1)(d-2)} K  I_{AB}\,.
\end{align*}
Squaring this identity gives the 
quoted result.
\end{proof}

\begin{lemma} \label{IIoIVo} 
Let $d > 6$, then the product of the second and fourth fundamental forms has a holographic formula given by
$$\IIo \csdot \IVo \eqSig (d-4) \Big(P_{AB} \ddot{P}^{AB} + \tfrac{4}{(d-2)^2} K^2\Big)\, .$$
\end{lemma}

\begin{proof}
First, because $d > 5$, we see that $\IVo$ is a canonical fundamental form (so not conditional). As in the previous lemma, we note that $\IIo \cdot \IVo = \Pi_{(2)\, AB} \Pi_{(4)}^{AB} $. Moreover, Equation~\nn{thestatementwearelookingfor} implies that $N^A P_{AB} = \frac1{d-2} K X_B$ so 
$$\Pi_{(2)} = P |_{\Sigma} - \tfrac{2 K}{d-2} N \odot X\, .$$
Thus we are  tasked with computing 
$P_{AB} \Pi_{(4)}^{AB}$ along $\Sigma$.
Remembering that $X \csdot\hh P = 0$, it is sufficient to compute~$\Pi_{(4)}^{AB}$ modulo terms proportional to  $X$.
Using Definition~\ref{DefconnII}, we compute~$\Pi_{(4)}$ in steps. Recall, from Equation~\nn{home-brand}, that
$$\Pi_{(4)} = (d-4)\, (\bar r \circ \otop \circ \delta_R \circ q \circ q^* \circ r \circ I \csdot \hd \circ q)(\IIo^{\rm e})\,.$$ We outline this calculation  proceeding from right to left in this sequence of operators.
\medskip

First, as shown previously, $q(\IIo^{\rm e}) = P$, so $I \csdot \hd \hh q(\IIo^{\rm e}) = \dot{P}$. Therefore, using Lemma~\ref{Pantheon}, we next compute $r(\dot{P})$:
$$r(\dot{P})_{AB} = \dot{P}_{AB} - \tfrac{d+2}{4d(d-2)} X_{(A} \hd_{B)} K - \tfrac{2}{d(d-2)} h_{AB} K+
 {\mathcal O}( \sigma^{d-2})+
 {\mathcal O}( \sigma^{d-4})X_{(A}T_{B)}\,,$$
for some tractor $T_B$. Here we used Equation~\nn{XPdot}
and the oft-used operator identity for $\hd \circ X$ given in Equation~\nn{usemeoft}.

\medskip

\noindent
Next, we need to compute $(q \circ q^* \circ r)(\dot{P})$. Before we continue, we consider the operators that come next: We are only interested in the $\bar{Z}_A \bar{Z}_B$ component of the tractor $\Pi_{(4)}$, so we can ignore terms proportional to $X$ in $(\otop \circ \delta_R \circ q \circ q^* \circ r)(\dot{P})$ when doing this computation. Further, we can ignore terms proportional to $I_A$ when computing $(\delta_R \circ q \circ q^* \circ r)(\dot{P})$ because these terms are projected out by $\otop$. The various projections, therefore, amount to ignoring terms proportional to~$X$ or $I$ when computing $(q \circ q^* \circ r)(\dot{P})$, because $\delta_R \circ I \stackrel{\Sigma} = {\mathcal E}(X) + {\mathcal E}(I)$ and $\delta_R\circ X\stackrel{\Sigma} = {\mathcal E}(X) + {\mathcal E}(I)$.
From 
Lemma~\ref{D-free-tractor}, 
$(q \circ q^* \circ r)(\dot{P}) - r(\dot{P}) = {\mathcal E}(X)$, so 
$$(q \circ q^* \circ r)(\dot{P}) = \dot{P}- \tfrac{2}{d(d-2)} K h_{} + {\mathcal E}(X)+{\mathcal E} (I)+ {\mathcal O}(\sigma^{d-2})\,.$$

\noindent
Next, note that $(\delta_R \circ q \circ q^* \circ r)(\dot{P}) = \ddot{P} - \tfrac{2}{d(d-2)} \Kd h + {\mathcal E}(X)+{\mathcal E} (I)+{\mathcal O}(\sigma^{d-3})$, and we can apply the operator~$\otop$ to obtain
$$(\otop \circ \delta_R \circ q \circ q^* \circ r)(\dot{P}) = \otop(\ddot{P}) + {\mathcal E}(X)\,.$$
Next, it is useful to note that $\otop(\ddot{P}_{AB}) \eqSig I_A^{A'} I_B^{B'} \ddot{P}_{A'B'} + I_{AB} U$ for some $U \in \Gamma(\ce \Sigma[-3])$, and also that $I_{AB} P^{AB} \eqSig 0$. Thus, finishing the calculation amounts to computing
\begin{align} \label{display-before-last}
P^{AB} \bar{r}(I_A^{A'} I_B^{B'} \ddot{P}_{A'B'} + I_{AB} U)\,.
\end{align}
For this, we need the identity
$$X^A \ddot{P}_{AB} \eqSig -\tfrac{2}{d-2} \big(\Kd X_B + K N_B\big)\,,$$
which is derived from Proposition~\ref{leib-failure}, Equation~\nn{usemeoft}, and Equation~\nn{starsare}. Because we are contracting on $P$, any terms proportional to $X$ or the tractor first fundamental form produced by $\bar r$  in Equation~\ref{display-before-last} can be discarded. Hence, again consulting Lemma~\ref{Pantheon}, we find that
\begin{align*}
\Pi_{(2)\hh AB} \Pi_{(4)}^{AB} \eqSig& (d-4) P_{AB} I^A_{A'} I^B_{B'} \ddot{P}_{A'B'} \\
\eqSig& (d-4) \big(P_{AB} - \tfrac{2}{d-2} K N_{(A} X_{B)} \big) \ddot{P}^{AB} \\
\eqSig& (d-4) \big(P_{AB} \ddot{P}^{AB} + \tfrac{4}{(d-2)^2} K^2 \big)\,,
\end{align*}
where the first equality is a result of the identity $I_{AB} P^{AB} \eqSig 0$, the second equality is an application of Equation~\nn{XPdot} 
to yield an 
identity for $I \cdot {P}$,
and the last equality is a consequence of the display above expressing $X \cdot \ddot{P}$. 
\end{proof}

\medskip

One more technical lemma is necessary in order to produce
formul\ae\ for $\Kd$ and $\Kdd$ in terms of the canonical fundamental forms.
\begin{lemma} \label{DP2} 
Let $d > 5$. Then,
\begin{align*}
(\hd P)^2 \eqSig& (\hdb L)^2 + \IIIo^2 + \tfrac{2}{(d-4)(d-5)} \IIo \cdot \IVo - \tfrac{4(d-7)}{(d-3)(d-5)} \IIo \cdot \IIIo \cdot \IIo + \tfrac{2(d-7)}{d-5} \IIo^4 \\
&- \tfrac{2(3d^3 - 34d^2 + 100d-73)}{(d-1)(d-2)^2 (d-5)} K^2 \,.
\end{align*}
\end{lemma}
\begin{proof}
The proof is a tedious but straightforward application of Equation~\ref{L-P-formula},  Lemmas
~\ref{GTF},~\ref{IIIo2}, \ref{IIoIVo} and standard reorderings of tractor operators based on Proposition~\ref{leib-failure}. 
\end{proof}

\medskip

Employing these lemmas, we have   formul\ae\ for $\Kd$ and $\Kdd$.
\begin{proposition}
Let $d > 4$. Then,
\begin{eqnarray}
\Kd = 2 \IIo \cdot \IIIo\, . \label{Kdot}
\end{eqnarray}
If $d> 6$,
\begin{eqnarray}
\Kdd &\eqSig&\, - \tfrac{2}{d-6} (\hdb L)^2 + \tfrac{2(d-7)}{d-6} \IIIo^2 + \tfrac{2(d-7)}{(d-5)(d-6)} \IIo \cdot \IVo + \tfrac{8(d-7)}{(d-3)(d-5)(d-6)} \IIo \csdot \IIIo \csdot \IIo
\nonumber
\\ &&\qquad- \tfrac{4(d-7)}{(d-5)(d-6)} \IIo^4   + \tfrac{10d^3 - 110d^2 + 296d -172}{(d-1)(d-2)^2 (d-5)(d-6)} K^2  \, . \label{Kddot}
\end{eqnarray}
Moreover, if $d=5$ then
\begin{eqnarray}\label{staggering}
\nonumber
\Kdd&\stackrel\Sigma=& -4\hh  \IIo \csdot \bar{\Delta} \IIo
 +\tfrac{20}3\IIo \csdot \bar{\nabla} \bar{\nabla} \csdot \IIo
+  \tfrac89 \big(\bar{\nabla} \csdot \IIo\big)^2
+ \bar \Delta K\\
&&\phantom{=}
  -4 \hh \IIo \csdot C_n^{\top} 
  +20\hh \IIo^2 \csdot \bar{P} 
 +2 \bar{J} K 
 -4 H \IIo^3
  -4 H \IIo \csdot \IIIo \\
  \nonumber
&&\phantom{=}\:\:\:  
+4\hh \IIIo \csdot \IIIo 
- 2 \hh\IIo^2 \csdot \IIIo
+ \tfrac{31}{18} K^2  
+8\hh \IIo^{ad} \IIo^{bc}\overline{W}_{abcd} \, .
\end{eqnarray}

\end{proposition}

\begin{proof}
We first prove that $\Kd = 2 \IIo \cdot \IIIo$
(note that a proof was already given in~\cite{Will2}).
 To do so, consider the product $\IIo \cdot \IIIo$. By inserting these fundamental forms into tractors, we find that $\IIo \cdot \IIIo = L \cdot \bar{q}(\IIIo)$. From 
Equation~\nn{Iwasdemoted}
and the fact that $X \csdot L = 0$, we have that $\IIo \cdot \IIIo = L \csdot \dot{P}^t$. Further, $L \eqSig P + {\mathcal E}(X) + {\mathcal E}(N)$ and $X \csdot \dot{P}^t = 0=N \csdot \dot{P}^t $, so $\IIo \csdot \IIIo \eqSig P \csdot \dot{P}^t$. By definition, $\dot{P}^t = \bar{r}(I_A^{A'} I_B^{B'} \dot{P}_{A'B'} + I_{AB} U)$ for some $U \in \Gamma(\ce \Sigma[-2])$. Thus, because $X \csdot \dot{P} = {\mathcal E}(X)$ (see Equation~\nn{XPdot}), using Equation~\nn{Pantheon} we have that $P \cdot 
\dot P^t
\eqSig P \cdot \dot{P}$ and in turn $\IIo \cdot \IIIo \eqSig P \cdot \dot{P}$. But from Proposition~\ref{leib-failure}, we have that $\Kd \eqSig 2 P \cdot \dot{P}$, so the first claim of the lemma follows.

\smallskip

To prove the second claim, for which $d > 6$, we first apply Proposition~\ref{leib-failure} twice to $P_{AB}P^{AB}$ and find  that
$$\Kdd \eqSig 
 2 \dot{P}^2
+2 P \cdot \ddot{P}  - \tfrac{2}{d-6} (\hd P)^2\,.$$
Applying Lemmas~\ref{IIIo2}, \ref{IIoIVo}, and ~\ref{DP2}, we obtain the  second claim.

Finally, we turn to the third claim with~$d = 5$. Because Lemmas~\ref{IIoIVo} and \ref{DP2} do not hold when~$d = 5$, we need to use a different method. Also if  $\Sigma \hookrightarrow (M^5, \cc)$ is a generic conformally embedded hypersurface,  the tensor  $\IVo$ is only a canonical   conditional fundamental form, so in particular it cannot appear in an otherwise conformally-invariant expression for $\Kdd$. Thus, to compute $\Kdd$ in this case, we resort to a Riemannian computation, and use that when~$d = 5$ (see~\cite{Will2}),
$$I \csdot \hd^2 K^{\rm e} \eqSig \Big[ \Delta^\top - 2\bar{J} +\tfrac{1}{3} K + 2\nabla_n^2 + 4 \big(2 H \nabla_n - P_{nn} - \tfrac{1}{3} K + \tfrac{5}{2} H^2 \big) \Big] (\nabla n + s P + g \rho)^2\,.$$
The expression for $\nabla_n^2 \rho$ along $\Sigma$ may also be found in~\cite{Will2}.
The remaining terms were handled by using the computer algebra system FORM~\cite{Jos}; this computation is documented in~\cite{FormFiles}.
\end{proof}

It only remains to prove Theorem~\ref{goodwill} from  the introduction.
\begin{proof}[Proof of Theorem~\ref{goodwill}]
The quantity $\ddot K$ is a density of weight $-4$ and hence can be invariantly integrated over the hypersurface $\Sigma$ when $d=5$.
Moreover the 
Codazzi--Mainardi Equation~\nn{trfreecod} can be used to
verify  the 
identity
$$
\bar\nabla^c W_{c(ab)\hat n}^\top
=
\bar \Delta \IIo_{ab}
-\tfrac{d-1}{d-2}\bar \nabla_{(a} \bar \nabla \csdot \IIo_{b)\circ}
-(d-1) \IIo_{(a}\csdot \bar P_{b)\circ}
-\bar J\hh  \IIo_{ab}
-\IIo^{cd}\bar W_{cabd}\, ,
$$
which is valid in any bulk dimension~$d\geq 4$.
It is straightforward to check that the left hand side of the above display 
is conformally invariant precisely when $d=5$.  It 
can then be used to verify that the integrand in Display~\nn{Samaritan}
differs from a non-zero multiple  of $\ddot K$ only by total divergences and manifestly invariant terms. 
(Note that the choice of $g\in \cc$ is used to trivialize density bundles in the  integrand of Display~\nn{Samaritan}.)
\end{proof}

\section*{Acknowledgements}

A.W.~was also supported by a Simons Foundation Collaboration Grants for Mathematicians ID 317562 and 686131, and  thanks the University of Auckland for warm hospitality.
A.W. and A.R.G.
 gratefully acknowledge support from the Royal Society of New Zealand via Marsden Grants 16-UOA-051 and 19-UOA-008.

\color{black}

\bibliographystyle{plain}
\bibliography{conformal}

\end{document}